\documentclass{article} 
\usepackage{moreverb}

\newcommand\BibTeX{{\rmfamily B\kern-.05em \textsc{i\kern-.025em b}\kern-.08em
T\kern-.1667em\lower.7ex\hbox{E}\kern-.125emX}}

\usepackage[utf8]{inputenc} 

\usepackage{geometry} 
\usepackage{graphicx} 

\usepackage[parfill]{parskip} 

\usepackage{booktabs} 
\usepackage{array} 
\usepackage{paralist} 
\usepackage{verbatim} 
\usepackage{subcaption} 
\usepackage{amsmath, amsthm, amssymb} 
\usepackage{color}
\usepackage{float}
\usepackage{mathtools}
\usepackage{mathrsfs}
\usepackage{pgfplots}
\usepackage{textcomp}
\usepackage{multirow}
\usepackage{algorithm, algorithmicx, algpseudocode}
\pgfplotsset{compat=1.14}

\usepackage{fancyhdr} 
\title{\bfseries Classification and image processing with a semi-discrete scheme for fidelity forced Allen--Cahn on graphs}
\author{Jeremy Budd$^1$, Yves van Gennip$^2$, Jonas Latz$^3$\\
\\
\small{$^{1,2}$Delft Institute of Applied Mathematics (DIAM)},\\
\small{Technische Universiteit Delft, }\small{Delft, The Netherlands.}\\
\small{$^3$Department of Applied Mathematics and Theoretical Physics (DAMTP)},\\ 
\small{University of Cambridge, Cambridge, United Kingdom.}\\
\\
   \small{ \textup{\texttt{$^1$j.m.budd-1@tudelft.nl \qquad $^2$y.vangennip@tudelft.nl \qquad $^3$jl2160@cam.ac.uk}}}\\
}
\date{} 
\numberwithin{equation}{section}
\newtheoremstyle{exampstyle}
  {4pt} 
  {4pt} 
  {\itshape} 
  {} 
  {\bfseries} 
  {.} 
  {.5em} 
  {} 
\theoremstyle{exampstyle}
\newtheorem{thm}{Theorem}[section]
\newtheorem{mydef}[thm]{Definition}
\newtheorem{example}[thm]{Example}

\newtheorem{prop}[thm]{Proposition}

\newtheorem{lem}[thm]{Lemma}
\newtheorem{cor}[thm]{Corollary}
\newtheorem*{nb}{Note}

\newcommand{\be}{\begin{equation}}
\newcommand{\ee}{\end{equation}}

\newcommand{\V}{\mathcal{V}}
\newcommand{\E}{\mathcal{E}}
\newcommand{\bigO}{\mathcal{O}}
\DeclareMathOperator*{\argmin}{argmin}

\DeclareMathOperator{\GL}{GL_{\mathit\varepsilon}}
\DeclareMathOperator{\fGL}{GL_{\mathit{\varepsilon,\mu},\mathit{\tilde f}}}

\DeclarePairedDelimiter\ceil{\lceil}{\rceil}

\DeclarePairedDelimiter\ip{\langle}{\rangle_\V}

\makeatletter
\def\@cite#1#2{{\normalfont[{#1\if@tempswa , #2\fi}]}}
\makeatother

\begin{document}
\maketitle

\begin{abstract}
This paper introduces a semi-discrete implicit Euler (SDIE) scheme for the Allen--Cahn equation (ACE) with fidelity forcing on graphs. Bertozzi and Flenner (2012) pioneered the use of this differential equation as a method for graph classification problems, such as semi-supervised learning and image segmentation. In Merkurjev, Kosti\'c, and Bertozzi (2013), a Merriman--Bence--Osher (MBO) scheme with fidelity forcing was used instead, as the MBO scheme is heuristically similar to the ACE. This paper rigorously establishes the graph MBO scheme with fidelity forcing as a special case of an SDIE scheme for the graph ACE with fidelity forcing. This connection requires using the double-obstacle potential in the ACE, as was shown in Budd and Van Gennip (2020) for ACE without fidelity forcing. We also prove that solutions of the SDIE scheme converge to solutions of the graph ACE with fidelity forcing as the SDIE time step tends to zero.

Next, we develop the SDIE scheme as a classification algorithm. We also introduce some innovations into the algorithms for the SDIE and MBO schemes. For large graphs, we use a QR decomposition method to compute an eigendecomposition from a Nystr\"om extension, which outperforms the method used in e.g. Bertozzi and Flenner (2012) in accuracy, stability, and speed. Moreover, we replace the Euler discretisation for the scheme's diffusion step by a computation based on the Strang formula for matrix exponentials. We apply this algorithm to a number of image segmentation problems, and compare the performance of the SDIE and MBO schemes. We find that whilst the general SDIE scheme does not perform better than the MBO special case at this task, our other innovations lead to a significantly better segmentation than that from previous literature. We also empirically quantify the uncertainty that this segmentation inherits from the randomness in the Nystr\"om extension.

{\bf 2010 AMS Classification.} 34B45, 35R02, 34A12, 65N12, 05C99.

\end{abstract}
\section{Introduction}
In this paper, we investigate the Allen--Cahn gradient flow of the Ginzburg--Landau functional on a graph, and the Merriman--Bence--Osher (MBO) scheme on a graph, with fidelity forcing. We extend  to the case of fidelity forcing the definition of the semi-discrete implicit Euler (SDIE) scheme  introduced in \cite{Budd} for the graph Allen--Cahn equation (ACE), and prove that the key results of \cite{Budd} hold true in the fidelity forced setting, i.e.
\begin{itemize}
\item the MBO scheme with fidelity forcing is a special case of the SDIE scheme with fidelity forcing; and
\item the SDIE solution converges to the solution of Allen--Cahn with fidelity forcing as the SDIE time step tends to zero. 
\end{itemize}

We then demonstrate how to employ the SDIE scheme as a classification algorithm, making a number of improvements upon the MBO-based classification in \cite{MKB}. In particular, we have developed a stable method for extracting an eigendecomposition or singular value decomposition (SVD) from the Nystr\"om extension \cite{Nys,FBCM} that is both faster and more accurate than the previous method used in \cite{MKB,BF}. Finally, we test the performance of this scheme as an alternative to graph MBO as a method for image processing on the ``two cows'' segmentation task considered in \cite{MKB,BF}.

Given an edge-weighted graph, the goal of two-class graph classification is to partition the vertex set into two subsets in such a way that the total weight of edges within each subset is high and the weight of edges between the two subsets is low. Classification differs from clustering by the addition of some \emph{a priori} knowledge, i.e. for certain vertices the correct classification is known beforehand. Graph classification has many applications, such as semi-supervised learning and image segmentation \cite{BF,birdspot}.

All programming for this paper was done in \textsc{Matlab}R2019a. Except within algorithm environments and URLs, all uses of \texttt{typewriter font} indicate in-built \textsc{Matlab} functions. 
\subsection{Contributions of this work}
In this paper we have:
\begin{itemize}
\item Defined a double-obstacle ACE with fidelity forcing (Definition \ref{ACEdef}), and extended the theory of \cite{Budd} to this equation (Theorem \ref{fACEthm}).
\item Defined an SDIE scheme for this ACE (Definition \ref{fSDdef}) and following \cite{Budd} proved that this scheme is a generalisation of the fidelity forced MBO scheme (Theorem \ref{obsMMthm}), derived a Lyapunov functional for the SDIE scheme (Theorem \ref{fLyapthm}), and proved that the scheme converges to the ACE solution as the time-step tends to zero (Theorem \ref{SDlimit}).
\item Described how to employ the SDIE scheme as a generalisation of the MBO-based classification algorithm in \cite{MKB}. 
\item Developed a method,  inspired by \cite{BK}, using the QR decomposition to extract an approximate SVD of the normalised graph Laplacian from the Nystr\"om extension (Algorithm \ref{nysQR}), which avoids the potential for errors in the method from \cite{MKB,BF} that can arise from taking the square root of a non-positive-semi-definite matrix, and empirically produces much better performance than the \cite{MKB,BF} method (Fig. \ref{Fig_Timings_Error_LR}) in  accuracy, stability, and speed.
\item Developed a method using the quadratic error Strang formula for matrix exponentials \cite{Strang} for computing fidelity forced graph diffusion (Algorithm \ref{SDalg}), which empirically incurs a lower error than the error incurred by the semi-implicit Euler method used in \cite{MKB} (Fig. \ref{LPFfig}), and explored other techniques with the potential to further reduce error (Table \ref{btable}).
\item Demonstrated the application of these algorithms to image segmentation, particularly the ``two cows'' images from \cite{MKB,BF}, compared the quality of the segmentation to those produced in \cite{MKB,BF} (Fig. \ref{fig:usvsBFMKB}), and investigated the uncertainty in these segmentations (Fig. \ref{Fig_MonteCarlo}), which is inherited from the randomisation in Nystr\"om.
\end{itemize}

This work extends the work in \cite{Budd} in four key ways. Firstly, introducing fidelity forcing changes the character of the dynamics, e.g. making graph diffusion affine, which changes a number of results/proofs, and it is thus of interest that the SDIE link continues to hold between the MBO scheme and the ACE. Secondly, this work for the first time considers the SDIE scheme as a tool for applications. Thirdly, in developing the scheme for applications we have made a number of improvements to the methods used in the previous literature \cite{MKB} for MBO-based classification, which result in a better segmentation of the ``two cows'' image than that produced in \cite{MKB} or \cite{BF}. Fourthly, we quantify the randomness that the segmentation inherits from the Nystr\"om extension.

\subsection{Background}

In the continuum, a major class of techniques for classification problems relies upon the minimisation of total variation (TV), e.g. the famous Mumford--Shah \cite{MS} and Chan--Vese \cite{CV} algorithms. These methods are linked to Ginzburg--Landau methods by the fact that the Ginzburg--Landau functional $\Gamma$-converges to TV \cite{MM,KS} (a result that continues to hold in the graph context \cite{vGB}). This motivated a common technique of minimising the Ginzburg--Landau functional in place of TV, e.g. in \cite{ET} two-class Chan--Vese segmentation was implemented by replacing TV with the Ginzburg--Landau functional; the resulting energy was minimised by using a fidelity forced MBO scheme.

Inspired by this continuum work, in \cite{BF} a method for graph classification was introduced based on minimising the Ginzburg--Landau functional on a graph by evolving the graph Allen--Cahn equation (ACE). The \emph{a priori} information was incorporated by including a fidelity forcing term, leading to the equation
\[
\frac{du}{dt} = -\Delta u - \frac1\varepsilon W'\circ u - \hat\mu P_Z (u-\tilde f),
\]
where $u$ is a labelling function which, due to the influence of a double-well potential (e.g. $W(x) = x^2 (x-1)^2$) will take values close to $0$ and $1$, indicating the two classes. The \emph{a priori} knowledge is encoded in the reference $\tilde f$ which is supported on $Z$, a subset of the node set with corresponding projection operator $P_Z$. In the first term $\Delta$ denotes the graph Laplacian and $\varepsilon, \hat\mu>0$ are parameters. All these ingredients will be explained in more detail in Sections~\ref{Gwork} and~\ref{ACEsec}.

In \cite{MKB} an alternative method was introduced: a graph Merriman--Bence--Osher (MBO) scheme with fidelity forcing. The original MBO scheme, introduced in a continuum setting in \cite{MBO92} to approximate motion by mean curvature, is an iterative scheme consisting of diffusion alternated with a thresholding step. In \cite{MKB} this scheme was discretised for use on graphs and the fidelity forcing term $- M(u-\tilde f)$ (where $M$ is a diagonal non-negative matrix, see Section~\ref{ACEsec} for details) was added to the diffusion. Heuristically, this MBO scheme was expected to behave similarly to the graph ACE as the thresholding step resembles a ``hard'' version of the ``soft'' double-well potential nonlinearity in the ACE.

In \cite{Budd} it was shown that the graph MBO scheme {\it without} fidelity forcing could be obtained as a special case of a semi-discrete implicit Euler (SDIE) scheme for the ACE (without fidelity forcing), if the smooth double-well potential was replaced by the double-obstacle potential defined in \eqref{Wobs}, and that solutions to the SDIE scheme converge to the solution of the graph ACE as the time step converges to zero. This double-obstacle potential was studied for the continuum ACE in \cite{BE1991,BE1992,BE1993} and was used in the graph context in \cite{BKS2018}. In \cite{volumeBudd} a result similar to that obtained in \cite{Budd} was obtained for a mass-conserving graph MBO scheme. In this paper such a result will be established for the graph MBO scheme {\it with} fidelity forcing.

In \cite{vGGOB} it was shown that the graph MBO scheme {\it pins} (or {\it freezes}) when the diffusion time is chosen too small, meaning that a single iteration of the scheme will not introduce any change as the diffusion step will not have pushed the value at any node past the threshold. In \cite{Budd} it was argued that the SDIE scheme for graph ACE provides a relaxation of the MBO scheme: The hard threshold is replaced by a gradual threshold, which should allow for the use of smaller diffusion times without experiencing pinning. The current paper investigates what impact that has in practical problems. 

\subsection{Groundwork}\label{Gwork}
We briefly summarise the framework for analysis on graphs, following the summary in \cite{Budd} of the detailed presentation in \cite{vGGOB}. 
A graph $G = (V,E)$ will henceforth be defined to be a finite, simple, undirected, weighted, and connected graph without self-loops with vertex set $V$, edge set $E\subseteq V^2$, and weights $\{\omega_{ij}\}_{i,j\in V}$ with $\omega_{ij}\geq 0$, $\omega_{ij} = \omega_{ji}$, $\omega_{ii}=0$, and $\omega_{ij} > 0$ if and only if $ij\in E$. We define the following  function spaces on $G$  (where $X\subseteq \mathbb{R}$, and $T\subseteq\mathbb{R}$ an interval):
\begin{align*}
	&\V := \left\{ u: V\rightarrow\mathbb{R} \right\} , &\V_{X} := \left\{ u: V\rightarrow X \right\}&,  &\mathcal{E} := \left\{ \varphi: E\rightarrow\mathbb{R}
\right\}.&\\
	&\V_{t\in T} := \left\{ u: T\rightarrow\V \right\} , &\V_{X,t\in T} := \left\{ u: T\rightarrow \V_X \right\}.&
\end{align*}
Defining $d_i := \sum_{j\in V} \omega_{ij}$ to be the \emph{degree} of vertex $i\in V$, we define inner products on $\V$ (or $\V_X$) and $\E$ (where $r\in [0,1]$):
\begin{align*}
	&\ip{u,v} := \sum_{i\in V} u_i v_i d_i^r, &\langle\varphi,\phi\rangle_\mathcal{E}:=\frac{1}{2}\sum_{i,j\in V} \varphi_{ij} \phi_{ij}\omega_{ij},&
\end{align*} and define the inner product on $\V_{t\in T}$ (or $\V_{X,t\in T}$) \[ (u,v)_{t\in T}:=\int_T \left\langle u(t),v(t)\right\rangle_\V \;dt = \sum_{i\in V} d_i^r\, (u_i,v_i)_{L^2(T;\mathbb{R})}.\] 
These induce inner product norms $||\cdot||_\V$, $||\cdot||_\mathcal{E}$, and $||\cdot||_{t\in T}$. We also define on $\V$ the norm \[ ||u||_\infty := \max_{i\in V} |u_i|.\]
Next, we define the $L^2$
space: \begin{align*} &L^2(T;\V) : =\left\{ u\in\V_{t\in T} \mid ||u||_{t\in T} <\infty \right\}, 
\end{align*}
and, for $T$ an open interval, we define the {Sobolev space} $H^1(T;\V)$ as the set of $u \in L^2(T;\V)$ with weak derivative $du/dt \in L^2(T;\V)$ defined by
\[  \forall \varphi\in C^\infty_c(T;\V)\:\:\left(u,\frac{d\varphi}{dt}\right)_{t\in T} = -\left(\frac{du}{dt},\varphi\right)_{t\in T} \] 
where $C^\infty_c(T;\V)$ is the set of $\varphi \in \V_{t\in T}$ that are infinitely differentiable with respect to time $t\in T$ and are compactly supported in $T$. By \cite[Proposition 2.1]{Budd},
$u\in H^1(T;\V)$ if and only if $u_i \in H^1(T;\mathbb{R})$ for each $i\in V$. 
We define the local $H^1$ space on any interval $T$ (and likewise define the local $L^2$ space $L^2_{loc}(T;\V)$): \[ H^1_{loc}(T;\V) :=\left\{u\in \V_{t\in T}\,\middle|\,\forall a,b\in T, \: u\in H^1((a,b);\V) \right\}.\] 

For $A\subseteq V$, we define the \emph{characteristic function} of $A$, $\chi_A\in\V$, by
\[ (\chi_A)_i := \begin{cases} 1, & \text{if }i\in A,\\
 0, & \text{if }i\notin A.\end{cases}\]
Next, we introduce the graph gradient and Laplacian:
\begin{align*}
	&(\nabla u)_{ij}:=\begin{cases}u_j -u_i, & ij\in E,\\ 0, &\text{otherwise,} \end{cases} &(\Delta u)_i:=d_i^{-r}\sum_{j\in V}\omega_{ij}(u_i-u_j).&
\end{align*}
Note that $\Delta$ is positive semi-definite and self-adjoint with respect to $\V$. As shown in \cite{vGGOB}, these operators are related via:
\[
\ip{u,\Delta v} = \langle \nabla u, \nabla v \rangle_\mathcal{E}.
\]
We can interpret $\Delta$ as a matrix. Define $D := \operatorname{diag}(d)$ (i.e. $D_{ii}:=d_i$, and $D_{ij}:=0$ otherwise) to be the \emph{diagonal matrix of degrees}. Then writing $\omega$ for the matrix of weights $\omega_{ij}$ we get 
\[
\Delta:= D^{-r}(D-\omega).
\]
 From $\Delta$ we define the \emph{graph diffusion operator}: \[e^{-t\Delta}u:=\sum_{n\geq 0} \frac{(-1)^n t^n}{n!}\Delta^n u\] where $v(t)=e^{-t\Delta}u $ is the unique solution to ${dv}/{dt} = -\Delta v$ with $v(0) = u$. Note that $e^{-t\Delta}\mathbf{1} = \mathbf{1}$, where $\mathbf{1}$ is the vector of ones.
	By \cite[Proposition 2.2]{Budd} if $u\in H^1(T;\V)$ and $T$ is bounded below, then $e^{-t\Delta}u\in H^1(T;\V)$ with \[\frac{d}{dt}\left(e^{-t\Delta}u\right) = e^{-t\Delta}\frac{du}{dt} -  e^{-t\Delta}\Delta u.\] 
We recall from functional analysis the notation, for any linear operator $F:\V\rightarrow\V$,
\begin{align*} &\sigma(F):=\{\lambda :\text{$\lambda$ an eigenvalue of $F$}\}\\
&\rho(F):=\max\{|\lambda| : \lambda \in \sigma(F)\}\\
&||F|| := \sup_{||u||_\V = 1} ||Fu||_\V
\end{align*}
and recall the standard result that if $F$ is self-adjoint then $||F|| = \rho(F)$.

Finally, we recall some notation from \cite{Budd}: for problems of the form ${\argmin}_x\: f(x)$ we write $f \simeq g$ and say $f$ and $g$ are \emph{equivalent} when $g(x) = af(x) + b$ for $a>0$ and $b$ independent of $x$. As a result, replacing $f$ by $g$ does not affect the minimisers.

Lastly, we define the non-fidelity-forced versions of the graph MBO scheme, the graph ACE and the SDIE scheme.  

The MBO scheme is an iterative, two-step process, originally developed in \cite{MBO92} to approximate motion by mean curvature. On a graph, it is defined in \cite{MKB} by the following iteration: for $u_n\in \V_{\{0,1\}}$, and $\tau >0$ the \emph{time step},
\begin{enumerate}
\item $v_n:= e^{-\tau\Delta}u_n$, i.e. the diffused state of $u_n$ after a time $\tau$.
\item $(u_{n+1})_i = \begin{cases} 1, &\text{ if }(v_n)_i \geq 1/2, \\ 0, &\text{ if } (v_n)_i < 1/2. \end{cases}$
\end{enumerate}

 To define the \emph{graph Allen\textendash Cahn equation} (\emph{ACE}), we first define the \emph{graph Ginzburg\textendash Landau functional} as in {\cite{Budd}} by 
\begin{equation*}
	\label{GL}
	\GL(u) := \frac{1}{2}\left|\left|\nabla u\right|\right|_\mathcal{E}^2 +\frac{1}{\varepsilon}\left\langle W\circ u,\mathbf{1} \right \rangle_\V
\end{equation*}
where $W$ is a double-well potential and $\varepsilon>0$ is a scaling parameter. Then the ACE results from taking the $\ip{\cdot,\cdot}$ gradient flow of $\GL$, which for $W$ differentiable is given by the ODE (where $\nabla_\V$ is the Hilbert space gradient on $\V$):
\begin{equation*}
	\label{AC}
	\frac{du}{dt} = -\nabla_\V\GL(u) = -\Delta u - \frac{1}{\varepsilon} W'\circ u .
\end{equation*}

To facilitate the SDIE link from \cite{Budd} between the ACE and the MBO scheme, we will henceforth take $W$ to be defined as:
\begin{equation}
	\label{Wobs}
	W(x) := \begin{cases}
    \frac{1}{2}x(1-x), & \text{for } 0 \leq x \leq 1, \\
    \infty, & \text{otherwise,}  \end{cases}
\end{equation}
the \emph{double-obstacle potential} studied by Blowey and Elliott \cite{BE1991,BE1992,BE1993} in the continuum and Bosch, Klamt, and Stoll \cite{BKS2018} on graphs.
As $W$ is not differentiable, we redefine the ACE via the subdifferential of $W$. As in \cite{Budd} we say that a pair $(u,\beta)\in\V_{[0,1],t\in T}\times\V_{t\in T}$ is a solution to the double-obstacle ACE on an interval $T$ if $u\in H_{loc}^1(T;\V)$ and for a.e. $t\in T$ 
\begin{align*}
\label{ACobs2}
	&\varepsilon \frac{du}{dt}(t) + \varepsilon\Delta u(t)  +\frac{1}{2}\mathbf{1}-u(t)= \beta(t), &\beta(t)\in\mathcal{B}(u(t))
\end{align*}
where $\mathcal{B}(u)$ is the set (for $I_{[0,1]}(x):=0$ if $x\in[0,1]$ and $I_{[0,1]}(x):=\infty$ otherwise)
\begin{equation}
\label{oldbeta}
	 \mathcal{B}(u) :=\left\{  \alpha\in\V\: \middle|\: \forall i\in V,  \alpha_i\in -\partial I_{[0,1]}(u_i)
	\right\}.
\end{equation}
That is, $\mathcal{B}(u)=\emptyset$ if $u\notin\V_{[0,1]}$, and for $u\in\V_{[0,1]}$ it is the set of $\beta\in\V$ such that 
\[\beta_i\in\begin{cases}
		[0,\infty), & u_i=0,\\
		\{0\}, &0<u_i < 1,\\
		(-\infty,0], & u_i=1.
	\end{cases}\]

Finally, the SDIE scheme for the graph ACE is defined in \cite{Budd} by the formula
\[ 	u_{n+1}=e^{-\tau\Delta} u_n- \frac{\tau}{\varepsilon} W'\circ u_{n+1} \]
or more accurately, given the above detail with the subdifferential, 
\begin{equation*}
\label{SDobs}
(1-\lambda)u_{n+1}-e^{-\tau\Delta}u_n+\frac{\lambda}{2}\mathbf{1} =\lambda\beta_{n+1}
\end{equation*}
where $\lambda:=\tau/\varepsilon$ and $\beta_{n+1}\in\mathcal{B}(u_{n+1})$. The key results of \cite{Budd} are then that:
\begin{itemize}
\item When $\tau = \varepsilon$, this scheme is exactly the MBO scheme. 
\item For $\varepsilon$ fixed and $\tau\downarrow 0$, this scheme converges to the solution of the double-obstacle ACE (which is a well-posed ODE).
\end{itemize}

\subsection{Paper outline}
The paper is structured as follows. In Section~\ref{Gwork} we introduced important concepts and notation for the rest of the paper. 
Section~\ref{ACEsec} contains the main theoretical results of this paper. It defines the graph MBO scheme with fidelity forcing, the graph ACE with fidelity forcing, and the SDIE scheme for graph ACE with fidelity forcing. It proves well-posedness for the graph ACE with fidelity forcing and establishes the rigorous link between a particular SDIE scheme and the graph MBO with fidelity forcing. Moreover, it introduces a Lypunov functional for the SDIE scheme with fidelity forcing and proves convergence of solutions of the SDIE schemes to the solution of the graph ACE with fidelity forcing. 
In Section~\ref{classificationsec} we explain how the SDIE schemes can be used for graph classification. In particular, the modifications to the existing MBO-based classification algorithms based on the QR decomposition and Strang formula are introduced. 
Section~\ref{applicationsec} presents a comparison of the SDIE and MBO scheme for an image segmentation applications, and an investigation into the uncertainty in these segmentations.
In Appendix~\ref{MKBapp} it is shown that the application of the Euler method used in \cite{MKB} can be seen as an approximation of the Lie product formula. 

\section{The Allen--Cahn equation, the MBO scheme, and the SDIE scheme with fidelity forcing }\label{ACEsec}
\subsection{The MBO scheme with fidelity forcing}
Following \cite{MKB,ET}, we introduce fidelity forcing into the MBO scheme by first defining a fidelity forced diffusion.
\begin{mydef}[Fidelity forced graph diffusion]
For $u\in H^1_{loc}([0,\infty);\V)$ and $u_0\in \V$ we define fidelity forced diffusion to be\emph{:}
\begin{align}\label{fdiffuse}
&\frac{du}{dt}(t) = -\Delta u(t)-M(u(t)-\tilde f)) =:-Au(t) +M\tilde f, &u(0) = u_0,
\end{align}
where $M:=\operatorname{diag}(\mu)$ for $\mu\in\V_{[0,\infty)}\setminus\{\mathbf{0}\}$ the \emph{fidelity parameter}, $A:=\Delta + M$, and $\tilde f \in \V_{[0,1]}$ is the \emph{reference}.  We define $Z:=\operatorname{supp}(\mu)\neq\emptyset$, which is the \emph{reference data} we enforce fidelity on. Note that $\mu_i$ paramaterises the strength of the fidelity to the reference at vertex $i$.
For the purposes of this section we shall treat $\mu$ and $\tilde f$ \emph{(}and therefore $M$ and $Z$\emph{)} as fixed and given. Moreover, since $\tilde f$ only ever appears in the presence of $M$, we define $f:=M\tilde f$ which is supported only on $Z$. Note that $f_i:=\mu_i\tilde f_i\in[0,\mu_i]$.
\end{mydef}
\begin{nb}
This fidelity term generalises slightly that used \emph{(}for ACE\emph{)} in \cite{BF}, in which $\mu := \hat \mu \chi_Z$ for $\hat\mu>0$ a parameter \emph{(}i.e. fidelity was enforced with equal strength on each vertex of the reference data\emph{)}, and so $M=\hat\mu P_Z$ where $P_Z$ is the projection map\emph{:}
\[
(P_Zu)_i = \begin{cases} u_i, &\text{if }i\in Z,\\
0. &\text{if }i\notin Z. \end{cases}
\]  
This generalisation has practical relevance, for example if one's confidence in the accuracy of the reference was higher at some vertices of the reference data than at others, then due to the link between the value of the fidelity parameter and the statistical precision \emph{(}i.e. the inverse of the variance of the noise\emph{)} of the reference \emph{(}see \cite[Section~3.3]{GraphUQ} for details\emph{)} it might be advantageous for one to use a fidelity parameter that is non-constant on the reference data. 
\end{nb}
\begin{prop}\label{Aspec}
$A$ is invertible with $ \sigma(A)\subseteq (0, ||\Delta|| + ||\mu||_\infty ]$.
\end{prop}
\begin{proof}
For the lower bound, we show that $A$ is strictly positive definite. Let $u \neq \mathbf{0}$ be written $u = v +\alpha\mathbf{1}$ for $v\bot\mathbf{1}$. Then 
\[
\ip{u,Au} =\ip{ v,\Delta v} + \ip{u,M u}
\]
and note that both terms on the right hand side are non-negative. Next, if $v\neq \mathbf{0}$ then   
\[
\ip{u,Au} \geq \ip{ v,\Delta v} = ||\nabla v||^2_\mathcal{E} > 0
\]
since $v\bot \mathbf{1}$ and hence $\nabla v \neq \mathbf{0}$, since $G$ is connected. Else, $v =\mathbf{0}$ so $\alpha \neq 0$ and 
\[
\ip{u,Au} = \alpha^2 \ip{\mathbf{1},\mu}>0.
\]
For the upper bound: $A$ is the sum of self-adjoint matrices, so is self-adjoint and hence has largest eigenvalue equal to $||A|| = ||\Delta + M|| \leq ||\Delta|| + ||M|| = ||\Delta|| + ||\mu||_\infty$. 
\end{proof}
\begin{thm} \label{fdiffusethm}
For given $u_0\in\V$, \eqref{fdiffuse} has a unique solution in $H^1_{loc}([0,\infty);\V)$. The solution $u$ to \eqref{fdiffuse} is $C^1((0,\infty);\V)$ and is given by the map \emph{(}where $I$ denotes the identity matrix\emph{):}
\be\label{fdiffusesoln}\begin{split}
u(t) = \mathcal{S}_t u_0 
							   &:=e^{-tA}u_0 + A^{-1}(I - e^{-tA})f.
\end{split}\ee
This solution map has the following properties\emph{:} 
\begin{enumerate}[i.]
\item If $u_0 \leq v_0$ vertexwise, then for all $t\geq 0$, $\mathcal{S}_tu_0 \leq \mathcal{S}_tv_0$ vertexwise.
\item $\mathcal{S}_t:\V_{[0,1]} \rightarrow \V_{[0,1]}$ for all $t\geq 0$, i.e. if $u_0\in\V_{[0,1]}$ then $u(t)\in\V_{[0,1]}$.
\end{enumerate}
\end{thm}
\begin{proof} 
It is straightforward to check directly that \eqref{fdiffusesoln} satisfies \eqref{fdiffuse} and is $C^1$ on $(0,\infty)$. Uniqueness is given by a standard Picard--Lindel\"of argument (see e.g. \cite[Corollary 2.6]{Teschl}).
\begin{enumerate}[i.]
\item By definition, $\mathcal{S}_t v_0 - \mathcal{S}_t u_0 = e^{-tA}(v_0-u_0)$. Thus it suffices to show that $e^{-tA}$ is a non-negative matrix for $t\geq 0$. Note that the off-diagonal elements of $-tA$ are non-negative: for $i\neq j$, $-tA_{ij}=-t\Delta_{ij} = td_i^r \omega_{ij}\geq 0$. Thus for some $a>0$, $Q:=aI-tA$ is a non-negative matrix and thus $e^Q$ is a non-negative matrix.  It follows that $e^{-tA}= e^{-a}e^Q$ is a non-negative matrix.
\item Let $u_0\in\V_{[0,1]}$ and recall that $\tilde f\in\V_{[0,1]}$. Suppose that for some $t>0$ and some $i\in V$, $u_i(t)<0$. Then \[ \min_{t'\in [0,t]} \min_{i\in V} u_i(t') < 0\] and since each $u_i$ is continuous this minimum is attained at some  $t^*\in [0,t]$ and $i^*\in V$. Fix such a $t^*$. Then for any $i^*$ minimising $u(t^*)$, since $ u_{i^*}(t^*) < 0$ we must have $t^* >0$, so $u_{i^*}$ is differentiable at $t^*$ with $du_{i^*}/dt(t^*) = 0$. However by \eqref{fdiffuse} 
\[ \frac{du_{i^*}}{dt}(t^*) = -(\Delta u(t^*))_{i^*} + \mu_{i^*} (\tilde f_{i^*} - u_{i^*}(t^*)).\]
We claim that we can choose a suitable minimiser $i^*$ such that this is strictly positive. First, since any such $i^*$ is a minimiser of $u(t^*)$, and $\tilde f_i \geq 0$ for all $i$, it follows that each term is non-negative. Next, suppose such an $i^*$ has a neighbour $j$ such that $u_j(t^*) > u_{i^*}(t^*)$, then it follows that $(\Delta u(t^*))_{i^*} <0$ and we have the claim. Otherwise, all the neighbours of that $i^*$ are also minimisers of $u(t^*)$. Repeating this same argument on each of those, we either have the claim for the above reason or we find a minimiser $i^*\in Z$, since $G$ is connected. But in that case $\mu_{i^*}(\tilde f_{i^*} - u_{i^*}(t^*))\geq -\mu_{i^*}u_{i^*}(t^*) > 0$, since $\mu$ is strictly positive on $Z$, and we again have the claim. Hence $du_{i^*}/dt(t^*)>0$, a contradiction. Therefore $u_i(t)\geq 0$ for all $t$. The case for $u_i(t) \leq 1$ is likewise.
\end{enumerate}
\end{proof}
\begin{mydef}[Graph MBO with fidelity forcing]
For $u_0\in \V_{[0,1]}$ we follow \emph{\cite{MKB,ET}}, and define the sequence of MBO iterates by diffusing with fidelity for a time $\tau\geq 0$ and then thresholding, i.e. 
\be
\label{fMBO}
(u_{n+1})_i = \begin{cases} 1, &\text{if }(\mathcal{S}_{\tau}u_n)_i \geq 1/2,\\
0, &\text{if }(\mathcal{S}_{\tau}u_n)_i < 1/2,  \end{cases}
\ee
where $\mathcal{S}_\tau$ is the solution map from \eqref{fdiffusesoln}. Note that \eqref{fMBO} has variational form similar to that given for graph MBO in \emph{\cite{vGGOB}}, which we can then re-write as in \emph{\cite{Budd}:} 
\be
\label{fMBOvar}
u_{n+1}\in \argmin_{u\in\V_{[0,1]}}\: \ip{\mathbf{1}-2\mathcal{S}_\tau u_n,u}\simeq  \frac{1}{2\tau}\ip{\mathbf{1}-u,u} + \frac{\left|\left|u-\mathcal{S}_\tau u_n\right|\right|^2_\V}{2\tau}. 
\ee
\end{mydef}
\subsection{The Allen--Cahn equation with fidelity forcing}
To derive the Allen--Cahn equation (ACE) with fidelity forcing, we re-define the Ginzburg--Landau energy to include a fidelity term (recalling the potential $W$ from \eqref{Wobs}):
\be 
\label{fGL}
\fGL(u) := \frac{1}{2}\left|\left|\nabla u\right|\right|_\mathcal{E}^2 +\frac{1}{\varepsilon}\left\langle W\circ u,\mathbf{1} \right \rangle_\V + \frac{1}{2}\ip{u-\tilde f, M(u-\tilde f)}.
\ee 
Taking the gradient flow of \eqref{fGL} we obtain the Allen--Cahn equation with fidelity:
\begin{align}
\label{fACE}
	&\varepsilon \frac{du}{dt}(t) + \varepsilon(\Delta u(t)+M(u(t)-\tilde f)) +\frac{1}{2}\mathbf{1}-u(t)= \beta(t), &\beta(t)\in\mathcal{B}(u(t)).
\end{align}
Where $\mathcal{B}(u(t))$ is defined as in \eqref{oldbeta}.
Recalling that $A:=\Delta + M$ and $f:=M\tilde f $, we can rewrite the ODE in \eqref{fACE} as 
\[
\varepsilon \frac{du}{dt}(t) + \varepsilon Au(t)-\varepsilon f +\frac{1}{2}\mathbf{1}-u(t)= \beta(t).
\]
As in \cite{Budd}, we can give an explicit expression for $\beta$ given sufficient regularity on $u$.
\begin{thm}\label{betathm}
Let $(u,\beta)$ obey \eqref{fACE} at a.e. $t\in T$, with $u\in H^1_{loc}(T;\V)\cap C^0(T;\V)\cap\V_{[0,1],t\in T}$. Then for all $i\in V$ and a.e. $t\in T$, \be\label{beta2} \beta_i(t) = \begin{cases} \frac{1}{2} +\varepsilon(\Delta u(t))_i-\varepsilon f_i,&u_i(t)=0,\\
0, & u_i(t)\in(0,1),\\
-\frac{1}{2} +\varepsilon(\Delta u(t))_i+\varepsilon(\mu_i-f_i), & u_i(t) =1.\end{cases}\ee
Hence at a.e. $t\in T$,
\[\beta(t)\in\V_{[-1/2,1/2]}.\]
\end{thm}
\begin{proof} Follows as in \cite[Theorem 2.2]{Budd} \emph{mutatis mutandis}.
\end{proof}
Thus following \cite{Budd} we define the double-obstacle ACE with fidelity forcing. 
\begin{mydef}[Double-obstacle ACE with fidelity forcing]\label{ACEdef}
Let $T$ be an interval. Then a pair $(u,\beta)\in\V_{[0,1],t\in T}\times\V_{t\in T}$ is a solution to double-obstacle ACE with fidelity forcing on $T$ when $u\in H_{loc}^1(T;\V)\cap C^0(T;\V)$ and for almost every $t\in T$, 
\begin{align}
\label{fACE2}
	&\varepsilon \frac{du}{dt}(t) + \varepsilon Au(t)-\varepsilon f +\frac{1}{2}\mathbf{1}-u(t)= \beta(t),  &\beta(t)\in\mathcal{B}(u(t)).
\end{align}
We frequently will refer to just $u$ as a solution to \eqref{fACE2}, since $\beta$ is a.e. uniquely determined as a function of $u$ by \eqref{beta2}.
\end{mydef}

We now demonstrate that this has the same key properties, \emph{mutatis mutandis}, as the ACE in \cite{Budd}.

\begin{thm}\label{fACEthm}
Let $T=[0,T_0]$ or $[0,\infty)$. 
Then\emph{:}
\begin{enumerate}[\emph{(}a\emph{)}]
\item \emph{(}Existence\emph{)} For any given $u_0\in\V_{[0,1]}$, there exists a $(u,\beta)$ as in Definition \ref{ACEdef} with $u(0)=u_0$. 
\item \emph{(}Comparison principle\emph{)} If $(u,\beta),(v,\gamma)\in\V_{[0,1],t\in T}\times\V_{t\in T}$ with $u,v\in H_{loc}^1(T;\V)\cap C^0(T;\V)$ satisfy
\begin{align}
\label{fACEsuper}
	&\varepsilon \frac{du}{dt}(t) + \varepsilon Au(t)-\varepsilon f +\frac{1}{2}\mathbf{1}-u(t)\geq \beta(t),  &\beta(t)\in\mathcal{B}(u(t)),\\
\intertext{and}
\label{fACEsub}
	&\varepsilon \frac{dv}{dt}(t) + \varepsilon Av(t)-\varepsilon f +\frac{1}{2}\mathbf{1}-v(t)\leq \gamma(t),  &\gamma(t)\in\mathcal{B}(v(t)),
\end{align}
vertexwise at a.e. $t\in T$, and $v(0)\leq u(0)$ vertexwise, then $v(t) \leq u(t)$ vertexwise for all $t\in T$.
\item \emph{(}Uniqueness\emph{)} If $(u,\beta)$ and $(v,\gamma)$ are as in Definition \ref{ACEdef} with $u(0)=v(0)$ then $u(t)=v(t)$ for all $t\in T$ and $\beta(t)=\gamma(t)$ at a.e. $t\in T$.
\item \emph{(}Gradient flow\emph{)} For $u$ as in Definition \ref{ACEdef}, $\fGL(u(t))$ monotonically decreases.
\item \emph{(}Weak form\emph{)} $u\in\V_{[0,1],t\in T}\cap H^1_{loc}(T;\V)\cap C(T;\V)$ \emph{(}and associated $\beta(t)= \varepsilon \frac{du}{dt}(t) + \varepsilon Au(t)-\varepsilon f +\frac{1}{2}\mathbf{1}-u(t)$ a.e.\emph{)} is a solution to \eqref{fACE2} if and only if for almost every $t\in T$ and all $\eta\in\V_{[0,1]}$
\be\label{fACEweak}
\left\langle \varepsilon \frac{du}{dt} + \varepsilon Au(t)-\varepsilon f +\frac{1}{2}\mathbf{1}-u(t) ,\eta-u(t)\right\rangle_\V\geq 0.
\ee
\item \emph{(}Explicit form\emph{)} $(u,\beta)\in\V_{[0,1],t\in T}\times\V_{t\in T}$ satisfies Definition \ref{ACEdef} if and only if for a.e. $t\in T$, $\beta(t)\in\mathcal{B}(u(t))$, $\beta(t)\in\V_{[-1/2,1/2]}$ and \emph{(}for $B := A - \varepsilon^{-1}I$ and $\varepsilon^{-1}\notin\sigma(A)$\emph{):} 
\be
\label{fACEsoln}
u(t)=e^{-tB} u(0)+B^{-1}\left(I-e^{-tB} \right)\left(f-\frac{1}{2\varepsilon}\mathbf{1} \right)+ \frac{1}{\varepsilon}\int_0^t e^{-(t-s)B}\beta(s) \; ds .
\ee 
\item  \emph{(}Lipschitz regularity\emph{)} For $u$ as in Definition \ref{ACEdef}, if $\varepsilon^{-1}\notin\sigma(A)$, then $u\in C^{0,1}(T;\V)$.
\item \emph{(}Well-posedness\emph{)} Let $u_0,v_0\in\V_{[0,1]}$ define the ACE trajectories $ u, v$ as in Definition \ref{ACEdef}, and suppose
$\varepsilon^{-1}\notin\sigma(A)$. Then, for $\xi_1:=\min\sigma(A)$, 
\be\label{fACEstab2}
|| u(t)- v(t)||_\V \leq e^{-\xi_1 t}e^{t/\varepsilon}||u_0-v_0||_\V.
\ee
\end{enumerate}
\end{thm}
\begin{proof}
\begin{enumerate}[(a)]
\item We prove this as Theorem \ref{SDlimit}.
\item We follow the proof of \cite[Theorem B.2]{Budd}. Letting $w := v-u$ and subtracting \eqref{fACEsuper} from \eqref{fACEsub}, we have that 
\[
\varepsilon\frac{dw}{dt}(t) + \varepsilon A w(t) - w(t) \leq \gamma(t)-\beta(t)
\]
vertexwise at a.e. $t\in T$. Next, take the inner product with $w_+:=\operatorname{max}(w,0)$, the vertexwise positive part of $w$:
\[\varepsilon\left\langle\frac{dw}{dt}(t),w_+(t)\right\rangle_\V+\varepsilon\ip{A w(t),w_+(t)} -\ip{w(t),w_+(t)} \leq \ip{\gamma(t)-\beta(t),w_+(t)}. \] 
As in the proof of \cite[Theorem B.2]{Budd}, the $RHS\leq 0$. For the rest of the proof to go through as in that Theorem, it suffices to check that $\ip{A w(t),w_+(t)} \geq \ip{A w_+(t),w_+(t)} $. But by \cite[Proposition B.1]{Budd}, $\ip{\Delta w(t),w_{+}(t)} \geq \ip{\Delta w_{+}(t),w_{+}(t)}$, and it is clear that $\ip{Mw(t),w_{+}(t)} = \ip{M w_{+}(t),w_{+}(t)}$ since $M$ is diagonal and non-negative, so the proof follows as in \cite[Theorem B.2]{Budd}. 
\item Follows from (b): if $(u,\beta)$ and $(v,\gamma)$ have $u(0)=v(0)$ and both solve \eqref{fACE2}, then applying the comparison principle in both directions gives that $u(t)=v(t)$ at all $t\in T$. Then \eqref{beta2} gives that $\beta(t)=\gamma(t)$ at a.e. $t\in T$.
\item We prove this in Theorem \ref{fGLthm} for the solution given by Theorem \ref{SDlimit}, which by uniqueness is the general solution.
\item Let $u$ solve \eqref{fACE2}. Then for a.e. $t\in T$, $\beta(t)\in\mathcal{B}(u(t))$, and at such $t$ we have $\ip{\beta(t),\eta-u(t)}\geq 0$ for all $\eta\in\V_{[0,1]}$ as in the proof of \cite[Theorem 3.8]{Budd}, so $u$ satisfies \eqref{fACEweak}. Next, if $u$ satisfies \eqref{fACEweak} at $t\in T$, then as in the proof of \cite[Theorem 3.8]{Budd}, $\beta(t):= \varepsilon \frac{du}{dt}(t) + \varepsilon Au(t)-\varepsilon f +\frac{1}{2}\mathbf{1}-u(t)\in\mathcal{B}(u(t))$, as desired.
\item Let $(u,\beta)\in\V_{[0,1],t\in T}\times\V_{t\in T}$. \textbf{If:} We check that \eqref{fACEsoln} satisfies \eqref{fACE2}. Note first that we can rewrite \eqref{fACE2} as 
\[
\frac{du}{dt}(t) + Bu(t) - f + \frac{1}{2\varepsilon}\mathbf{1} = \frac{1}{\varepsilon}\beta(t).
\]
Next, let $u$ be as in \eqref{fACEsoln}. Then it is easy to check that 
\[
\frac{du}{dt}(t) = -Be^{-tB}u(0) +e^{-tB}\left( f - \frac{1}{2\varepsilon}\mathbf{1}\right) + \frac{1}{\varepsilon}\beta(t) - \frac{1}{\varepsilon}B\int_0^t e^{-(t-s)B}\beta(s) \; ds 
\]
and that this satisfies \eqref{fACE2}. Next, we check the regularity of $u$. The continuity of $u$ is immediate, as it is a sum of two smooth terms and the integral of a locally bounded function. To check that $u\in H^1_{loc}$: $u$ is bounded, so is locally $L^2$, and by above $du/dt$ is a sum of (respectively) two smooth functions, a bounded function and the integral of a locally bounded function, so is locally bounded and hence locally $L^2$. \\
\textbf{Only if:} We saw that \eqref{fACEsoln} solves \eqref{fACE2} with $\beta(t)\in\mathcal{B}(u(t))$ and $\beta(t)\in\V_{[-1/2,1/2]}$, and by (c) such solutions are unique.
\item We follow the proof of \cite[Theorem 3.13]{Budd}. Let $0\leq t_1<t_2$. Since \eqref{fACE2} is time-translation invariant, we have by (f) that 
\[
u(t_2)= e^{-(t_2-t_1)B} u(t_1)+B^{-1}\left(I-e^{-(t_2-t_1)B} \right)\left( f-\frac{1}{2\varepsilon}\mathbf{1} \right)+ \frac{1}{\varepsilon}\int_0^{t_2-t_1} e^{-sB}\beta(t_2-s) \; ds 
\]
and so, writing $B_{s}:=(e^{-s B}-I)/s$ for $s>0$ (which we note commutes with $B$),
\[
u(t_2)-u(t_1)= (t_2-t_1)B_{t_2-t_1}\left( u(t_1)-B^{-1}\left( f-\frac{1}{2\varepsilon}\mathbf{1} \right)\right)+ \frac{1}{\varepsilon}\int_0^{t_2-t_1} e^{-sB}\beta(t_2-s) \; ds .
\]
Note that $B_s$ is self-adjoint, and as $-B$ has largest eigenvalue less than $\varepsilon^{-1}$ we have $||B_{ s}||<(e^{s/\varepsilon}-1)/s$, with RHS monotonically increasing in $s$ for $s>0$.\footnote{$\frac{d}{ds}((e^{s/\varepsilon}-1)/s) = s^{-2}e^{s/\varepsilon}\left(e^{-s/\varepsilon}-1+s/\varepsilon\right)>0$ for $s>0$.} Since $f \in\V_{[0,||\mu||_\infty]}$ and for all $t$, $\beta(t)\in\V_{[-1/2,1/2]}$ and $u(t)\in\V_{[0,1]}$, we have for $t_2-t_1<1$:
 \begin{equation*}\begin{split}\frac{\left|\left| u(t_2)- u(t_1)\right|\right|_\V}{t_2-t_1}&\leq ||B_{t_2-t_1}||\cdot \left(1+||\mu||_\infty ||B^{-1}||+ \frac{1}{2\varepsilon}||B^{-1}||\right)||\mathbf{1}||_\V +\frac{1}{\varepsilon} \underset{s\in [0,t_2-t_1]}{\text{ess sup}} \left|\left|e^{-sB}\beta(t_2-s)  \right|\right|_\V \\
 &<  \frac{e^{(t_2-t_1)/\varepsilon}- 1}{ t_2-t_1 }\cdot  \left(1+||\mu||_\infty||B^{-1}||+ \frac{1}{2\varepsilon}||B^{-1}||\right)||\mathbf{1}||_\V + \frac{1}{\varepsilon}\sup_{s\in [0,t_2-t_1]} \left|\left|e^{-sB}\right|\right| \cdot \frac{1}{2}||\mathbf{1}||_\V \\
&\leq \frac{e^{(t_2-t_1)/\varepsilon}- 1}{ t_2-t_1 }\cdot  \left(1+||\mu||_\infty ||B^{-1}||+ \frac{1}{2\varepsilon}||B^{-1}||\right)||\mathbf{1}||_\V + \frac{1}{\varepsilon} e^{(t_2-t_1)/\varepsilon}  \cdot \frac{1}{2}||\mathbf{1}||_\V \\
&< \left (\left(1+||\mu||_\infty ||B^{-1}||+ \frac{1}{2\varepsilon}||B^{-1}||\right)\left(e^{1/\varepsilon}  -1\right) + \frac{1}{2\varepsilon}e^{1/\varepsilon}\right)  ||\mathbf{1}||_\V
\end{split}\end{equation*} and for $t_2 -t_1\geq 1$ we simply have
\[\frac{\left|\left| u(t_2)- u(t_1)\right|\right|_\V}{t_2-t_1}\leq \left|\left| u(t_2)- u(t_1)\right|\right|_\V \leq ||\mathbf{1}||_\V \] completing the proof. 
\item We prove this as Theorem \ref{fACEstab} for the solution given by Theorem \ref{SDlimit}, which by uniqueness is the general solution. 
\end{enumerate}
\vspace{-1.5\baselineskip}
\end{proof}
\begin{nb}
Given the various forward references in the above proof, we take care to avoid circularity by not using the corresponding results until they have been proven.
\end{nb}
\subsection{The SDIE scheme with fidelity forcing and link to the MBO scheme}
\begin{mydef}[SDIE scheme with fidelity forcing, cf. \text{\cite[Definition 4.1]{Budd}}]
\label{fSDdef}
For $u_0\in\V_{[0,1]}$, $n\in\mathbb{N}$, and $\lambda:=\tau/\varepsilon\in[0,1]$ we define the SDIE scheme iteratively\emph{:}
\be
\label{fSD}
(1-\lambda)u_{n+1} -\mathcal{S}_\tau u_n+\frac{\lambda}{2}\mathbf{1} =\lambda\beta_{n+1}
\ee
for a $\beta_{n+1}\in \mathcal{B}(u_{n+1})$ to be characterised in Theorem \ref{obsMMthm}.
\end{mydef}
As in \cite{Budd}, we have the key theorem linking the MBO scheme and the SDIE schemes for the ACE.
\begin{thm}[Cf. \text{\cite[Theorem 4.2]{Budd}}]
\label{obsMMthm}
For $\lambda\in[0,1]$, the pair $(u_{n+1},\beta_{n+1})$ is a solution to the SDIE scheme \eqref{fSD} for some $\beta_{n+1}\in\mathcal{B}(u_{n+1})$ if and only if $u_{n+1}$ solves\emph{:}
 \begin{equation}
	\label{fSDvar}
	u_{n+1}\in \underset{u\in \V_{[0,1]}}{ \argmin }  \: \lambda\left\langle u,\mathbf{1}-u\right \rangle_\V +  \left|\left|u-\mathcal{S}_\tau u_n\right|\right|^2_\V.
\end{equation}
Note that for $\lambda = 1$ \eqref{fSDvar} is equivalent to the variational problem \eqref{fMBOvar} that defines the MBO scheme. Furthermore, \eqref{fSDvar} has unique solution for $\lambda\in[0,1)$ 
\begin{equation}\label{fSDsoln1}
	(u_{n+1})_i=\begin{cases}
		0, &\text{if } \left(\mathcal{S}_\tau u_n\right)_i < \frac{1}{2}\lambda,
\\
		\frac{1}{2} + \frac{\left( \mathcal{S}_\tau u_n\right)_i - 1/2}{1-\lambda}
, &\text{if }\frac{1}{2}\lambda
\leq\left(\mathcal{S}_\tau u_n\right)_i < 1-\frac{1}{2}\lambda,
\\
		1, &\text{if } \left(\mathcal{S}_\tau u_n\right)_i \geq 1-\frac{1}{2}\lambda,
	\end{cases}
\end{equation}
with corresponding $\beta_{n+1} = \lambda^{-1}\left((1-\lambda)u -\mathcal{S}_\tau u_n+\frac{\lambda}{2}\mathbf{1}\right)$, and solutions for $\lambda = 1$ 
\be
(u_{n+1})_i \in \begin{cases}
	\{1\}, &(\mathcal{S}_\tau u_n)_i > 1/2,\\
	[0,1], &(\mathcal{S}_\tau u_n)_i = 1/2,\\
	\{0\}, &(\mathcal{S}_\tau u_n)_i < 1/2.
\end{cases} \ee
\emph{(}i.e. the MBO thresholding\emph{)} with corresponding $\beta_{n+1} = \frac{1}{2}\mathbf{1}- \mathcal{S}_\tau u_n $.
\end{thm}
\begin{proof}
Identical to the proof of \cite[Theorem 4.2]{Budd} with the occurrences of ``$e^{-\tau\Delta}$'' in each instance replaced by ``$\mathcal{S}_\tau $''.
\end{proof}
As in \cite{Budd}, we can plot \eqref{fSDsoln1} to visualise the SDIE scheme \eqref{fSD} as a piecewise linear relaxation of the MBO thresholding rule.
\begin{figure}[h] 
\centering
\begin{tikzpicture}
\begin{axis}[
	height = 6cm,
    axis lines = left,axis line style={-},
    xlabel = $\left(\mathcal{S}_\tau u_n\right)_i$,
    ylabel = {$(u_{n+1})_i$},
    xtick={0,0.5,1},
    ytick={0,0.5,1},
    legend pos=south east,
    extra x ticks={0.15,0.85},
    extra x tick style={
    tick label style={
    ,anchor=north}},
    extra x tick labels={$\frac{1}{2}\lambda$,$1-\frac{1}{2}\lambda$},
xmin=0,xmax=1
]

\addplot [
    domain=0:0.15, 
    samples=100, 
    color=blue,
]
{0};

\addplot [
    domain=0.15:0.85, 
    samples=100, 
    color=blue,
]
{0.5 + (x-0.5)/0.7};
\addplot [
    domain=0.85:1, 
    samples=100, 
    color=blue,
]
{1};
\end{axis}
\begin{axis}[
height = 6cm,
axis y line*=right,
axis x line=none,
	axis line style={-},
    ylabel = {$(\beta_{n+1})_i$},
    ytick={-0.5,0,0.5},
xmin=0,xmax=1,ymin=-0.5,ymax=0.5]

\addplot [
    domain=0:0.15, 
    samples=100, 
    color=red,
]
{0.5-x/0.3};

\addplot [
    domain=0.15:0.85, 
    samples=100, 
    color=red,
]
{0};
\addplot [
    domain=0.85:1, 
    samples=100, 
    color=red,
]
{-0.5+(1-x)/0.3};
\end{axis}
\end{tikzpicture}
\caption{Plot of the SDIE updates $u_{n+1}$ (blue, left axis, see \eqref{fSDsoln1}) and $\beta_{n+1}$ (red, right axis) at vertex $i$ for $0\leq\lambda <1$ as a function of the fidelity forced diffused value at $i$. Cf. \cite[Fig. 1]{Budd}.} 
\label{SDfig}
\end{figure}
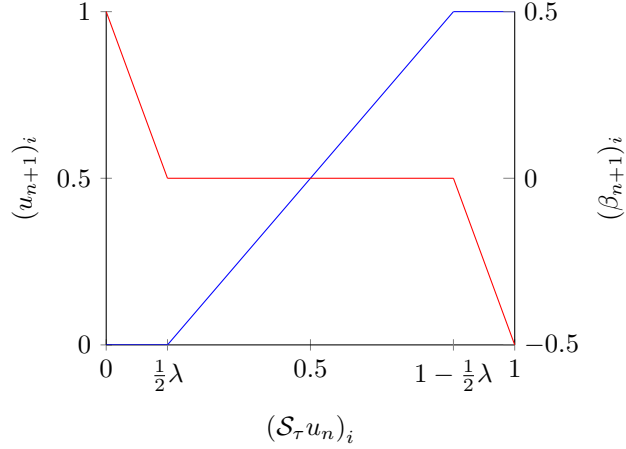 
Next, we note that we have the same Lipschitz continuity property from~\cite{Budd}.
\begin{thm}[\text{Cf. \cite[Theorem 4.4]{Budd}}]
For $\lambda \in[0,1)$\footnote{For the MBO case $\lambda = 1$ the thresholding is discontinuous so we do not get an analogous property.} and all $n\in\mathbb{N}$, if $u_n$ and $v_n$ are defined according to Definition \ref{fSDdef} with initial states $u_0,v_0\in\V_{[0,1]}$ and $\xi_1 := \min\sigma(A)$ then 
\be\label{fSDstab}
||u_n-v_n||_\V \leq e^{-n\xi_1 \tau}(1-\lambda)^{-n}||u_0-v_0||_\V.
\ee
\end{thm}
\begin{proof} Follows as in \cite[Theorem 4.4]{Budd} \emph{mutatis mutandis}.
\end{proof}
\subsection{A Lyapunov functional for the SDIE scheme}
\begin{lem}[\text{Cf. \cite[Lemma 4.6]{vGGOB}}] The functional on $\V$ \begin{equation*}
	J(u) := \ip{u,\mathbf{1}-2A^{-1}(I - e^{-\tau A})f-e^{-\tau A}u}
\end{equation*}
has the following properties\emph{:}
\begin{enumerate}[i.]
\item It is strictly concave.
\item It has first variation at $u$ \[ L_u(v) := \left\langle v,\mathbf{1}-2\mu A^{-1}(I - e^{-\tau A})f-2e^{-\tau A}u\right \rangle_\V = \left\langle v,\mathbf{1}-2\mathcal{S}_\tau u\right \rangle_\V. \] 
\end{enumerate}
\end{lem}
\begin{proof}
Let $w: = \mathbf{1}-2A^{-1}(I - e^{-\tau A})f$. We expand around $u$: \[\begin{split}J(u+tv) &= \ip{u+tv, w- e^{\tau A}(u+tv)} \\
&= \ip{u, w- e^{\tau A}u} +t \ip{v, w- e^{\tau A}u}-t\ip{u,e^{-tA}v} - t^2\ip{v,e^{-\tau A} v}.\end{split}\] 
\begin{enumerate}[i.]
\item $\frac{d^2}{dt^2}J(u+tv) = -2\ip{v,e^{-\tau A} v} <0$ for $v\neq \mathbf{0}$.
\item Since $e^{-tA}$ is self-adjoint, $J(u+tv) =  J(u) + tL_u(v) +\bigO(t^2)$.
\end{enumerate}
\vspace{-1.5\baselineskip}
\end{proof}
\begin{thm}[Cf. \text{\cite[Theorem 4.9]{Budd}}]
\label{fLyapthm}
	For $0\leq\lambda\leq 1$ we define on $\V_{[0,1]}$ the functional 
	\begin{equation}\label{Lyap}
	H(u):= J(u) + (\lambda-1)\ip{u,\mathbf{1}-u} .
	\end{equation}
	This has uniform lower bound \be H(u) \geq -2\tau||f||_\V||\mathbf{1}||_\V \ee and is a Lyapunov functional for \eqref{fSD}, i.e. $H(u_{n+1}) \leq H(u_n)$ with equality if and only if $u_{n+1}=u_n$ for $u_{n+1}$ defined by \eqref{fSD}. In particular, we have that \begin{equation}\label{Hstep}
		H(u_n)-H(u_{n+1}) \geq (1-\lambda)\left|\left|u_{n+1}-u_n\right|\right|^2_\V.
	\end{equation} 
\end{thm}
\begin{proof}
We can rewrite $H$ as:
\begin{align*}
H(u) &= \lambda\ip{u,\mathbf{1}-u} + \ip{u,u-2 A^{-1}(I - e^{-\tau A})f-e^{-\tau A}u}&&&\\
&\geq \ip{u,(I-e^{-\tau A})u} - 2\ip{u,A^{-1}(I - e^{-\tau A})f}& &\text{since $u\in\V_{[0,1]}$}&\\ 
&\geq - 2\ip{u,A^{-1}(I - e^{-\tau A})f}& &\text{since $I - e^{-\tau A}$ is positive definite }&\\
&\geq -2||f||_\V||u||_\V\left|\left|A^{-1}(I-e^{-\tau A})\right|\right|\\
&\geq -2||f||_\V||\mathbf{1}||_\V\left|\left|A^{-1}(I-e^{-\tau A})\right|\right| \geq -2\tau ||f||_\V ||\mathbf{1}||_\V
\end{align*}
where the final line follows since $A^{-1}(I-e^{-\tau A})$ is self-adjoint (since $A$ is) and has eigenvalues \[\left\{\frac{1-e^{-\tau\lambda}}{\lambda}	\;\middle|\; \lambda \in\sigma(A)\right\}\]
so we have by Proposition \ref{Aspec} that
\[ \begin{split}\left|\left|A^{-1}(I-e^{-\tau A})\right|\right| &\leq \sup_{x\in(0,||\Delta||+\mu]} \frac{1-e^{-\tau x}}{x}\\
&= \lim_{x\rightarrow 0}  \frac{1-e^{-\tau x}}{x}\quad \text{ as $x\mapsto x^{-1}(1-e^{-\tau x})$ is monotonically decreasing\footnotemark}\\
&= \tau. 
\end{split}\]\footnotetext{$\frac{d}{dx}(x^{-1}(1-e^{-\tau x}))=x^{-2}e^{-\tau x}(1+\tau x-e^{\tau x}) \leq 0$.}
	Next we show that $H$ is a Lyapunov functional. By the concavity of $J$:
	\[\begin{split}
	H(u_{n}) - H(u_{n+1}) &= J(u_{n}) - J(u_{n+1}) + (1-\lambda)\ip{u_{n+1},\mathbf{1}-u_{n+1}} - (1-\lambda)\ip{u_n,\mathbf{1}-u_n} \\
	&\geq L_{u_n}(u_{n} - u_{n+1}) + (1-\lambda)\ip{u_{n+1},\mathbf{1}-u_{n+1}} - (1-\lambda)\ip{u_n,\mathbf{1}-u_n} \:\left(*\right)\\
&= \ip{u_{n} - u_{n+1},\mathbf{1}-2\mathcal{S}_\tau u_n} + (1-\lambda)\ip{u_{n+1},\mathbf{1}-u_{n+1}} - (1-\lambda)\ip{u_n,\mathbf{1}-u_n} \\
	&= \ip{u_{n} - u_{n+1},\mathbf{1}-2\mathcal{S}_\tau u_n} + (1-\lambda)(\ip{u_{n+1}-u_n,\mathbf{1}} +\ip{u_n,u_n} -\ip{u_{n+1},u_{n+1}})\\
	&= \ip{u_{n} - u_{n+1},\lambda\mathbf{1}-2\mathcal{S}_\tau u_n + (1-\lambda)u_{n+1} + (1-\lambda)u_n}\\
	&= \ip{u_{n} - u_{n+1},2\lambda\beta_{n+1}+ (1-\lambda)(u_n-u_{n+1})} \text{ by \eqref{fSD}}\\
	&\geq (1-\lambda)\left|\left|u_{n+1}-u_n\right|\right|^2_\V\geq 0
	\end{split}
	\]
	with equality in $(*)$ if and only if $u_{n+1}=u_n$ as the concavity of $J$ is strict, and
where the last line follows since by $\beta_{n+1}\in\mathcal{B}(u_{n+1})$
\[(\beta_{n+1})_i((u_{n})_i-(u_{n+1})_i) = \left. \begin{cases} \text{``}\geq 0\text{''} \cdot (u_n)_i, &(u_{n+1})_i = 0\\
0, &(u_{n+1})_i \in (0,1)\\
\text{``}\leq 0\text{''} \cdot \left((u_n)_i-1\right), &(u_{n+1})_i = 1
\end{cases} \right\} \geq 0 \]
and so $\ip{u_{n} - u_{n+1},\beta_{n+1}}\geq 0$.
\end{proof}
\begin{cor}
For $\lambda = 1$ \emph{(}i.e. the MBO case\emph{)} the sequence $u_n$ defined by \eqref{fSD} is eventually constant.

For $0\leq \lambda\leq 1$, the sum
\[
\sum_{n=0}^\infty ||u_{n+1}-u_n||_\V^2 
\]
converges, and hence 
\[
\lim_{n\rightarrow\infty}  ||u_{n+1}-u_n||_\V = 0.
\]
\end{cor}
\begin{proof}
For the first claim, the proof follows as in \cite[Proposition 4.6]{vGGOB} \emph{mutatis mutandis}.
For the second claim, the proof follows as in \cite[Corollary 4.10]{Budd} \emph{mutatis mutandis}.
\end{proof}
\subsection{Convergence of the SDIE scheme with fidelity forcing}
Following \cite{Budd}, we first derive the $n^{\text{th}}$ term for the semi discrete sceme. 
\begin{prop}[Cf. \text{\cite[Proposition 5.1]{Budd}}]
For the sake of notation, define\emph{:}
\[w := -\frac{1}{2}\lambda\mathbf{1} + A^{-1}(I-e^{-\tau A})f.\]
Then for $\lambda\in [0,1)$ the sequence generated by \eqref{fSD} is given by\emph{:}
\be
\label{fSDsoln}
u_n =(1-\lambda)^{-n}e^{-n\tau A}u_0+\sum_{k=1}^n (1-\lambda)^{-k}e^{-(k-1)\tau A}w+\frac{\lambda}{1-\lambda}\sum_{k=1}^n (1-\lambda)^{-(n-k)}e^{-(n-k)\tau A}\beta_k.
\ee
\end{prop}
\begin{proof}
We can rewrite \eqref{fSD} as 
\[(1-\lambda)u_{n+1}= e^{-\tau A}u_n + A^{-1}(I-e^{-tA})f -\frac{1}{2}\lambda\mathbf{1}+\lambda\beta_{n+1} = e^{-\tau A}u_n +w+\lambda\beta_{n+1}.  \]
We then check \eqref{fSDsoln} inductively. The $n=0$ case is trivial, and we have that
\[\begin{split}
&(1-\lambda)^{-1}e^{-\tau A}u_n +
(1-\lambda)^{-1}w+\frac{\lambda}{1-\lambda}\beta_{n+1}\\
 &=  (1-\lambda)^{-(n+1)}e^{-(n+1)\tau A}u_0+\sum_{k=1}^n (1-\lambda)^{-(k+1)}e^{-k\tau A}w+\frac{\lambda}{1-\lambda}\sum_{k=1}^n (1-\lambda)^{-((n+1)-k)}e^{-((n+1)-k)\tau A}\beta_k\\
&\:\:+ (1-\lambda)^{-1}w+\frac{\lambda}{1-\lambda}\beta_{n+1} \\
&=(1-\lambda)^{-(n+1)}e^{-(n+1)\tau A}u_0+\sum_{k=0}^n (1-\lambda)^{-(k+1)}e^{-k\tau A}w+\frac{\lambda}{1-\lambda}\sum_{k=1}^{n+1} (1-\lambda)^{-((n+1)-k)}e^{-((n+1)-k)\tau A}\beta_k\\
&=u_{n+1}
\end{split}\]
completing the proof.
\end{proof}
Next, we consider the asymptotics of each term in \eqref{fSDsoln}.
\begin{thm}
Considering $\bigO$ relative to the limit of $\tau\downarrow 0$ and $n\rightarrow \infty$ with $n\tau-t\in [0,\tau)$ for some fixed $t\geq 0$ and for fixed $\varepsilon>0$ \emph{(}with $\varepsilon^{-1}\notin\sigma(A)$\emph{)}\footnote{More precisely, we will say for real (matrix) valued $g$, $g(\tau,n)=\bigO(\tau)$ if and only if $\operatorname{limsup} ||g(\tau,n)/\tau|| <\infty$ as $(\tau,n)\rightarrow(0,\infty)$ in $\{(\rho,m)\mid \rho > 0,\: m\rho-t\in[0,\rho)\}$ with the subspace topology induced by the standard topology on $ (0,\infty)\times\mathbb{N}$.}, and recalling that $\lambda:=\tau/\varepsilon$, $B:=A-\varepsilon^{-1}I$ and $w := -\frac{1}{2}\lambda\mathbf{1} + A^{-1}(I-e^{-\tau A})f$\emph{:}
\begin{enumerate}[i.]
\item $(1-\lambda)^{-n}e^{-n\tau A}u_0 = e^{t/\varepsilon}e^{-tA}u_0 +\bigO(\tau)= e^{-tB}u_0 +\bigO(\tau)$.
\item $\sum_{k=1}^n (1-\lambda)^{-k}e^{-(k-1)\tau A}w= B^{-1}(I-e^{-t B})f-\frac{1}{2\varepsilon}B^{-1}(I-e^{-t B})\mathbf{1} + \bigO(\tau)$.
\item $\frac{\lambda}{1-\lambda}\sum_{k=1}^n (1-\lambda)^{-(n-k)}e^{-(n-k)\tau A}\beta_k = \lambda\sum_{k=1}^n e^{-(n-k)\tau B}\beta_k + \bigO(\tau)$.
\end{enumerate}
Hence by \eqref{fSDsoln}, the SDIE term obeys 
\be
\label{fSDsoln2}
u_n = e^{-tB}u_0 + B^{-1}(I-e^{-t B})f-\frac{1}{2\varepsilon}B^{-1}(I-e^{-t B})\mathbf{1}+ \lambda\sum_{k=1}^n e^{-(n-k)\tau B}\beta_k +\bigO(\tau).
\ee
\end{thm}
\begin{proof}
Let $n\tau - t=:\eta_n = \bigO(\tau)$. Note that $e^{\eta_nX}= I + \bigO(\tau)$ for any bounded matrix $X$.  

Note that $\bigO(\tau)$ is the same as $\bigO(\lambda)$, and also that, for bounded (in $\tau$) invertible matrices $X$ and $Y$ with bounded (in $\tau$) inverses, $X = Y + \bigO(\tau)$ if and only if $X^{-1} = Y^{-1} +\bigO(\tau)$.\footnote{Suppose $X^{-1} = Y^{-1} +\bigO(\tau)$. Then $X = (Y^{-1} +\bigO(\tau))^{-1} = Y(I+\bigO(\tau))^{-1}=Y(I+\bigO(\tau)) = Y +\bigO(\tau)$.} 
\begin{enumerate}[i.]
\item $||(1-\lambda)^{-n}e^{-n\tau A}u_0 - e^{-tB}u_0||_\V \leq ||(1-\lambda)^{-n}e^{-n\tau A}- e^{-tB}||\cdot||u_0||_\V$, so it suffices to consider  $(1-\lambda)^{-n}e^{-n\tau A}- e^{-tB}$. Since $(1-\lambda)^{-n} = e^{n\lambda} +\bigO(\tau^2)$ we infer that 
\[
(1-\lambda)^{-n}e^{-n\tau A} = e^{t/\varepsilon} e^{\eta_n/\varepsilon} e^{-t A}e^{-\eta_n A} +\bigO(\tau^2) = e^{-tB} +\bigO(\tau).
\]
\item We note that \[\sum_{k=1}^n (1-\lambda)^{-k}e^{-(k-1)\tau A} = ((1-\lambda)I-e^{-\tau A})^{-1}(I-(1-\lambda)^{-n}e^{-n\tau A})= ((1-\lambda)I-e^{-\tau A})^{-1}(I-e^{-tB})+\bigO(\tau). \] We next consider each term of $w$ individually. First, we seek to show that  
\[((1-\lambda)I-e^{-\tau A})^{-1}(I-e^{-tB})A^{-1}(I-e^{-\tau A})f = B^{-1}(I-e^{-t B})f + \bigO(\tau)\]
so it suffices to show that  
\[((1-\lambda)I-e^{-\tau A})^{-1} A^{-1}(I-e^{-\tau A}) = B^{-1} + \bigO(\tau).\]
This holds if and only if 
\[ \begin{split} B &= ((1-\lambda)I-e^{-\tau A}) A(I-e^{-\tau A})^{-1} +\bigO(\tau)\\
&= A -\lambda A (I-e^{-\tau A})^{-1} +\bigO(\tau)\\
&= A -\varepsilon^{-1}\tau A \left(\tau A -\frac{1}{2}\tau^2 A^2 +-\cdots\right)^{-1}+\bigO(\tau)\\
&= A -\varepsilon^{-1}\left(I -\frac{1}{2}\tau A+-\cdots\right)^{-1}+\bigO(\tau)\\
&= A -\varepsilon^{-1}I+\bigO(\tau)
\end{split} \]
and since $B= A -\varepsilon^{-1}I$ the result follows. Next we seek to show that
\[
 ((1-\lambda)I-e^{-\tau A})^{-1}(I-e^{-tB})\frac{1}{2}\lambda\mathbf{1} = \frac{1}{2\varepsilon}B^{-1}(I-e^{-tB})\mathbf{1} +\bigO(\tau)
\]
so it suffices to show that  
\[
 ((1-\lambda)I-e^{-\tau A})^{-1}\tau = B^{-1} +\bigO(\tau)
\]
which holds if and only if 
\[
B =\tau^{-1} ((1-\lambda)I-e^{-\tau A}) +\bigO(\tau) = A -\varepsilon^{-1}I+\bigO(\tau)
\]
and since $B= A -\varepsilon^{-1}I$ the result follows.
\item We follow \cite[Proposition 5.1]{Budd} and consider the difference 
\[\begin{split}&\left |\left| \frac{\lambda}{1-\lambda}\sum_{k=1}^n (1-\lambda)^{-(n-k)}e^{-(n-k)\tau A}\beta_k - \lambda\sum_{k=1}^n e^{-(n-k)\tau B}\beta_k \right|\right|_\V\\
&=\lambda \left |\left|\sum_{k=1}^n\left( (1-\lambda)^{-(n-k+1)}-e^{(n-k)\lambda}\right)e^{-(n-k)\tau A}\beta_k \right|\right|_\V\\
&=\lambda \left |\left|\sum_{\ell=0}^{n-1}\left( (1-\lambda)^{-(\ell+1)}-e^{\ell\lambda}\right)e^{-\ell\tau A}\beta_k \right|\right|_\V\\
&\leq\lambda\sum_{\ell=0}^{n-1}\left( (1-\lambda)^{-(\ell+1)}-e^{\ell\lambda}\right)\left |\left|e^{-\ell\tau A}\beta_k \right|\right|_\V\:\:\text{as $(1-\lambda)^{-(\ell+1)}-e^{\ell\lambda}\geq 0$}\\
&\leq \frac{1}{2}\lambda||\mathbf{1}||_\V\sum_{\ell=0}^{n-1}\left( (1-\lambda)^{-(\ell+1)}-e^{\ell\lambda}\right)\:\:\text{as $||e^{-\ell\tau A}||\leq 1$ and $||\beta_k||_\V\leq \frac{1}{2}||\mathbf{1}||_\V$}\\
&=  \frac{1}{2}\lambda||\mathbf{1}||_\V\left(\frac{(1-\lambda)^{-n}-1}{1-(1-\lambda)}-\frac{e^{n\lambda}-1}{e^{\lambda}-1} \right)\\
&=\frac{1}{2}||\mathbf{1}||_\V\left((1-\lambda)^{-n} -e^{n\lambda} \right) + \bigO(\tau)\:\: \text{as $\lambda/(e^\lambda -1) = 1 + \bigO(\tau)$}\\
&=\bigO(\tau)
\end{split}\]
as desired.
\end{enumerate}
\vspace{-1.5\baselineskip}
\end{proof}
Following \cite{Budd} we define the piecewise constant function $z_\tau:[0,\infty)\rightarrow\V$
\[z_\tau(s): =\begin{cases} e^{\tau B}\beta^{[\tau]}_1, & 0\leq s \leq \tau\\
									e^{k\tau B} \beta^{[\tau]}_k,	& (k-1)\tau<s\leq k\tau \text{ for } k\in\mathbb{N}						
\end{cases}\]
and the function \[\gamma_\tau(s):= e^{-sB} z_\tau(s)=\begin{cases} e^{(\tau-s)B}\beta^{[\tau]}_1, & 0\leq s \leq \tau\\
								e^{(k\tau-s)B} \beta^{[\tau]}_k,	& (k-1)\tau<s\leq k\tau			\text{ for } k\in\mathbb{N}				
\end{cases}\]
following the bookkeeping notation of \cite{Budd} of using the superscript $[\tau]$ to keep track of the time-step governing $u_n$ and $\beta_n$.
Next, we have weak convergence of $z$, up to a subsequence, as in \cite{Budd}.
\begin{thm}
\label{convthm}
For any sequence $\tau^{(0)}_n\downarrow 0$ with $\tau^{(0)}_n<\varepsilon$ for all $n$,  there exists a function $z:[0,\infty) \rightarrow \V$ and a subsequence $\tau_n$ such that $z_{\tau_n}$ converges weakly to $z$ in $L^2_{loc}$. It follows that\emph{:}
\begin{enumerate}[\emph{(}A\emph{)}]
\item $\gamma_{\tau_n} \rightharpoonup \gamma$ in $L^2_{loc}$, where $\gamma(s) = e^{-sB}z$.
\item For all $t\geq 0$, \[\int_0^t z_{\tau_n}(s)\; ds \rightarrow \int_0^t z(s)\; ds.\]
\item Passing to a further subsequence of $\tau_n$, we have strong convergence of the Cesàro sums, i.e. for all bounded $T\subseteq[0,\infty)$ \begin{align*}&\frac{1}{N}\sum_{n=1}^N z_{\tau_n} \rightarrow z &\text{and} &&\frac{1}{N}\sum_{n=1}^N \gamma_{\tau_n} \rightarrow \gamma && \text { in } L^2(T;\V)\end{align*} as $N\rightarrow \infty$. 
\end{enumerate}
\end{thm}
\begin{proof}
Follows as in \cite[Proposition 5.2]{Budd} and \cite[Corollary 5.3]{Budd} \emph{mutatis mutandis}.
\end{proof}
We thus infer convergence of the SDIE iterates as in \cite{Budd}. Taking $\tau$ to zero along the sequence $\tau_n$, we can define for all $t\geq 0$: 
\be
\label{uhat}\hat u(t) := \lim_{n\rightarrow\infty, m =\ceil{t/\tau_n}} u^{[\tau_n]}_m.
\ee
By the above discussion, we can rewrite this as:
\[ \begin{split}\hat u(t) &=  e^{-tB}u_0 +  B^{-1}(I-e^{-t B})f-\frac{1}{2\varepsilon}B^{-1}(I-e^{-t B})\mathbf{1}+ \lim_{n\rightarrow\infty} \frac{\tau_n}{\varepsilon}\sum_{k=1}^m e^{-(m-k)\tau_n B}\beta^{[\tau_n]}_k \\
&= e^{-tB}u_0 +  B^{-1}(I-e^{-t B})f-\frac{1}{2\varepsilon}B^{-1}(I-e^{-t B})\mathbf{1}+ \frac{1}{\varepsilon} \lim_{n\rightarrow\infty}e^{-m\tau_nB}\int_0^{m\tau_n} z_{\tau_n}(s)\; ds.
\end{split}\]
Next, note that $m\tau_n = \tau_n\ceil{t/\tau_n} =: t + \eta_n$ where $\eta_n\in[0,\tau_n)$. Therefore 
\[\begin{split}
&\lim_{n\rightarrow\infty}e^{-m\tau_nB}\int_0^{m\tau_n} z_{\tau_n}(s)\; ds = \lim_{n\rightarrow\infty}e^{-\eta_nB}e^{-tB}\int_0^{t} z_{\tau_n}(s)\; ds + e^{-\eta_nB}e^{-tB}\int_t^{t+\eta_n} z_{\tau_n}(s)\; ds\\
&=\lim_{n\rightarrow\infty}e^{-tB}\int_0^{t} z_{\tau_n}(s)\; ds + e^{-tB}\int_t^{t+\eta_n} z_{\tau_n}(s)\; ds \:\: \text{ as $e^{-\eta_nB}=I +\bigO(\tau_n)$}\\
&=\lim_{n\rightarrow\infty}e^{-tB}\int_0^{t} z_{\tau_n}(s)\; ds \:\: \text{ as $z_{\tau_n}(s)$ is bounded on $[t,t+\max_{n'} \eta_{n'}]$  uniformly in $n$}\\
&=e^{-tB}\int_0^t z(s)\; ds \:\:\text{ by Theorem \ref{convthm}(B)}.
\end{split}\]
So we have that 
\be\label{uhatsoln} \hat u(t) = e^{-tB}u_0 +  B^{-1}(I-e^{-t B})f-\frac{1}{2\varepsilon}B^{-1}(I-e^{-t B})\mathbf{1}+ \frac{1}{\varepsilon}\int_0^t e^{-(t-s)B}\gamma(s) \; ds.
\ee
Note the similarity between \eqref{uhatsoln} and the explicit form for ACE solutions \eqref{fACEsoln}. Thus, to prove that $\hat u$ is a solution to \eqref{fACE2} it suffices to show that: 
\begin{enumerate}[(a)]
\item $\hat u(t) \in \V_{[0,1]}$ for all $t\geq 0$,
\item $\hat u \in H^1_{loc}([0,\infty);\V) \cap  C^0([0,\infty);\V)$, and
\item $\gamma(t) \in \mathcal{B}(\hat u(t))$ for a.e. $t\geq 0$.
\end{enumerate}
These results follow as in \cite{Budd}. Item (a) follows immediately from the fact that for all $n$, $u_m^{[\tau_n]}\in\V_{[0,1]}$.  

Towards (b), note that each term in \eqref{uhatsoln} except for the integral is $C^\infty([0,\infty);\V)$, and that $\int_0^t z(s)\; ds$ is continuous since $z$ is locally bounded as a weak limit of locally uniformly bounded  functions. Hence $\hat u$ is continuous. By (a), $\hat u$ is bounded so is locally $L^2$. Finally, it is easy to check that $\hat u$ has weak derivative 
\[
\frac{d\hat u}{dt} = -Be^{-tB}u_0 + e^{-tB}\left(f -\frac{1}{2\varepsilon}\mathbf{1}\right) + \frac{1}{\varepsilon}e^{-tB}z(t) -\frac{1}{\varepsilon}Be^{-tB}\int_0^t z(s)\; ds.
\]
This is locally $L^2$ since (for $T$ a bounded interval) $B$ and $e^{-tB}$ are bounded operators from $L^2(T;\V)$ to $L^2(T;\V)$, $z$ is a weak limit of locally $L^2$ functions so is locally $L^2$, and  $\int_0^t z(s)\; ds$ is continuous so is locally bounded.

Towards (c), by Theorem \ref{convthm}(C) and \cite[p. 25]{Budd} \emph{mutatis mutandis} there is a sequence $N_k \rightarrow \infty$, independent of $t$, with 
\[
\gamma(t) = \lim_{k\rightarrow\infty}  \frac{1}{N_k} \sum_{n=1}^{N_k} \beta_m^{[\tau_n]}
\] 
for a.e. $t\geq 0$. Then, at each such $t$, $\gamma(t) \in \mathcal{B}(\hat u(t))$  follows from  $u^{[\tau_n]}_m\rightarrow\hat u(t)$ and $\beta_m^{[\tau_n]}\in \mathcal{B}( u^{[\tau_n]}_m)$ as in \cite[p. 25]{Budd}. Hence we can infer the following convergence theorem.
\begin{thm}[\text{Cf. \cite[Theorem 5.4]{Budd}}]
\label{SDlimit}
For any given $u_0\in\V_{[0,1]}$, $\varepsilon>0$ \emph{(}with $\varepsilon^{-1}\notin\sigma(A)$\emph{)} and $\tau_n\downarrow 0$, there exists a subsequence $\tau'_n$ of $\tau_n$ with $\tau'_n<\varepsilon$ for all $n$, along which the SDIE iterates $(u^{[\tau'_n]}_m,\beta^{[\tau'_n]}_m)$ given by \eqref{fSD} with initial state $u_0$ converge to the ACE solution with initial condition $u_0$ in the following sense\emph{:} For each $t\geq 0$, as $n\rightarrow\infty$ and $m=\ceil{t/\tau'_n}$, $u^{[\tau'_n]}_m\rightarrow \hat u(t)$, and there is a sequence $N_k\rightarrow\infty$ such that for almost every $t\geq 0$, $\frac{1}{N_k}\sum_{n=1}^{N_k}\beta^{[\tau'_n]}_m\rightarrow \gamma(t)$, where $(\hat u,\gamma)$ is the solution to \eqref{fACE2} with $\hat u(0) = u_0$. 
\end{thm}
\begin{cor}\label{SDlimitcor}
Let $u_0\in\V_{[0,1]}$, $\varepsilon>0$, $\varepsilon^{-1}\notin\sigma(A)$ and $\tau_n\downarrow 0$ with $\tau_n<\varepsilon$ for all $n$. Then for each $t\geq 0$, as $n\rightarrow\infty$, $u^{[\tau_n]}_{\ceil{t/\tau_n}}\rightarrow \hat u(t)$.
\end{cor}
\begin{proof}
Let $x_n:t\mapsto u^{[\tau_n]}_{\ceil{t/\tau_n}}$ and let $\tau_{n_k}$ be any subsequence of $\tau_n$. Then by the theorem there is a subsubsequence $\tau_{n_{k_l}}$ such that $x_{n_{k_l}}\rightarrow \tilde u$ pointwise where $\tilde u$ is a solution to \eqref{fACE2} with initial condition $\tilde u(0) = u_0$. By Theorem \ref{fACEthm}(c) such solutions are unique, so $\tilde u = \hat u$. Thus there exists $x$ (in particular, $x=\hat u$) such that every subsequence of $x_n$ has a convergent subsubsequence with limit $x$. It follows by a standard fact of topological spaces that $x_n \rightarrow \hat u$ pointwise as $n\rightarrow\infty$.\footnote{Suppose $x_n\nrightarrow x$. Then there exists $U$ which is open in the topology of pointwise convergence such that $x\in U$ and infinitely many $x_n\notin U$. Choose $x_{n_k}$ such that for all $k$, $x_{n_k}\notin U$. This subsequence has no further subsubsequence converging to $x$.} 
\end{proof}
Finally, we follow \cite{Budd} to use Theorem \ref{SDlimit} to deduce that the Ginzburg--Landau energy monotonically decreases along the ACE trajectories by considering the Lyapunov functional  $H$ defined in \eqref{Lyap}. We also deduce well-posedness of the ACE.  
\begin{prop}[\text{Cf. \cite[Proposition 5.6]{Budd}}]
\label{Htauprop}
Let $H_\tau(u):=\frac{1}{2\tau} H(u)$. Then for $u\in\V_{[0,1]}$
\[
H_\tau(u) = \fGL(u) -\frac{1}{2}\ip{\tilde f, M\tilde f} +\frac{1}{2}\tau\ip{u,Q_\tau (u -2 A^{-1}f)}
\]
where $Q_\tau :=\tau^{-2}(I-\tau A -e^{-\tau A})$. Hence $H_\tau +\frac{1}{2}\ip{\tilde f, M\tilde f}\rightarrow \fGL$ uniformly on $\V_{[0,1]}$ as $\tau\rightarrow 0$, and furthermore if $u_\tau \rightarrow u$ in $\V_{[0,1]}$ then $H_\tau(u_\tau) +\frac{1}{2}\ip{\tilde f, M\tilde f}\rightarrow \fGL(u)$.
\end{prop}
\begin{proof}
Expanding and collecting terms in \eqref{fGL}, we find that for $u\in\V_{[0,1]}$
\[
\fGL(u) = \frac{1}{2\varepsilon}\ip{u,\mathbf{1}-u} +\frac{1}{2}\ip{u,Au -2f} +\frac{1}{2}\ip{\tilde f, M\tilde f}. 
\]
Then by \eqref{Lyap} and recalling that $\lambda:=\tau/\varepsilon$
\[\begin{split}
H_\tau(u) &= \frac{1}{2\varepsilon}\ip{u,\mathbf{1}-u} +\frac{1}{2\tau}\ip{u,(I-e^{-\tau A})u -2A^{-1}(I-e^{-\tau A})f}\\
&=\frac{1}{2\varepsilon}\ip{u,\mathbf{1}-u} +\frac{1}{2\tau}\ip{u,(\tau A+\tau^2Q_\tau)u -2A^{-1}(\tau A+\tau^2Q_\tau)f}\\
&=\frac{1}{2\varepsilon}\ip{u,\mathbf{1}-u} +\frac{1}{2}\ip{u, Au-2 f} +\frac{1}{2}\tau\ip{u,Q_\tau (u -2A^{-1}f)}.
\end{split}\]
To show the uniform convergence, note that $||u||_\V$ and $||u-2A^{-1}f||_\V$ are uniformly bounded in $u$ for $u\in\V_{[0,1]}$. Thus it suffices to prove that $||Q_\tau||$ is uniformly bounded in $\tau$. But $Q_\tau$ is self-adjoint, and if $\xi_k$ is an eigenvalue of $A$ then $Q_\tau$ has corresponding eigenvalue $\tau^{-2}(1-\tau\xi_k-e^{-\tau\xi_k})\in [-\frac{1}{2}\xi^2_k,0] $, so $||Q_\tau||\leq \frac{1}{2}||A||^2$.  
Finally, it suffices to show that $H_\tau(u_\tau)-H_\tau(u)\rightarrow 0$, since
\[H_\tau(u_\tau) + \frac{1}{2}\ip{\tilde f, M\tilde f}-\fGL(u) = H_\tau(u_\tau) -H_\tau(u) +H_\tau(u) + \frac{1}{2}\ip{\tilde f, M\tilde f} -\fGL(u).\]
Then by the above expression for $H_\tau$ \[H_\tau(u_\tau)-H_\tau(u) = \frac{1}{2}\left\langle u_\tau -u, \frac{1}{\varepsilon}\left(1-u_\tau-u\right)+(A +\tau Q_\tau)(u_\tau +u)-2(I+\tau Q_\tau A^{-1})f \right\rangle_\V\rightarrow 0\] since the right-hand entry in the inner product is bounded uniformly in $\tau$.  
\end{proof}
\begin{thm}[\text{Cf. \cite[Theorem 5.7, Remark 5.8]{Budd}}]\label{fGLthm}
Suppose and $\varepsilon^{-1}\notin\sigma(A)$. Then the ACE trajectory $u$ defined by Definition \ref{ACEdef} has $\fGL(u(t))$ monotonically decreasing in $t$. More precisely\emph{:} for all $t > s \geq 0 $, \be \label{GLstep} 
 \fGL(u(s)) -  \fGL( u(t)) \geq \frac{1}{2(t-s)} \left|\left| u(s) - u(t) \right|\right|_\V^2.
\ee 
Furthermore, this entails an explicit $C^{0,1/2}$ condition for $ u$
\be \label{C0half}
 \left|\left| u(s) - u(t) \right|\right|_\V\leq \sqrt{|t-s|}\sqrt{2\fGL( u(0))}.
\ee
\end{thm}
\begin{proof} The proof is identical to that in \cite[Theorem 5.7]{Budd} and \cite[Remark 5.8]{Budd}.
\end{proof}
\begin{thm}[\text{Cf. \cite[Theorem 3.11]{Budd}}]\label{fACEstab}
Let $u_0,v_0\in\V_{[0,1]}$ define ACE trajectories $ u, v$ by Definition \ref{ACEdef}, and suppose $\varepsilon^{-1}\notin\sigma(A)$. Then, if $\xi_1:=\min\sigma(A)$, then
\be\label{fACEstab1}
|| u(t)- v(t)||_\V \leq e^{-\xi_1 t}e^{t/\varepsilon}||u_0-v_0||_\V.
\ee
\end{thm}
\begin{proof}
Fix $t\geq 0$ and let $m:= \ceil{t/\tau_n}$. By Corollary \ref{SDlimitcor}, we take $\tau_n\downarrow 0$ such that $u_m^{[\tau_n]}\rightarrow  u(t)$ and $v_m^{[\tau_n]}\rightarrow  v(t)$ as $n\rightarrow\infty$. Then by \eqref{fSDstab}:
\[
||u_m^{[\tau_n]}-v_m^{[\tau_n]}||_\V \leq e^{-m\xi_1\tau_n}(1-\tau_n/\varepsilon)^{-m}||u_0-v_0||_\V
\]
and taking $n\rightarrow\infty$ gives \eqref{fACEstab1}.
\end{proof}

\section{The SDIE scheme as a classification algorithm}\label{classificationsec}

As was noted in the introduction, trajectories of the ACE and of the MBO scheme on graphs can be deployed as classification algorithms, as was originally done in work by Bertozzi and co-authors in \cite{MKB,BF}. In the above, we have shown that the SDIE scheme \eqref{fSD} introduced in \cite{Budd} generalises the MBO scheme into a family of schemes, all of the same computational cost as the MBO scheme, and which as $\tau\downarrow 0$ become increasingly more accurate approximations of the ACE trajectories. In the remainder of this paper, we will investigate whether the use of these schemes can significantly improve on the use of the MBO scheme to segment the ``two cows'' images from \cite{MKB,BF}. We will also discuss other potential improvements to the method of \cite{MKB}. 

\subsection{Groundwork }  

In this section, we will describe the framework for applying graph dynamics to classification problems, following \cite{MKB,BF}. 

The individuals that we seek to classify we will denote as a set $V$, upon which we have some information $x : V\rightarrow \mathbb{R}^q$. For example, in image segmentation $V$ is the pixels of the image, and $x$ is the greyscale/RGB/etc. values at each pixel. Furthermore, we have \emph{labelled reference data} which we shall denote as $Z\subseteq V$ and binary reference labels $\tilde f$ supported on $Z$. Supported on $Z$ we have our fidelity parameter $\mu\in\V_{[0,\infty)}\setminus\{\mathbf{0}\}$, and we recall the notation $M:=\operatorname{diag}(\mu)$ and $f:=M\tilde f$  (recall that the operator diag sends a vector to the diagonal matrix with that vector as diagonal, and also \emph{vice versa}).  

To build our graph, we first construct \emph{feature vectors}  $z:V\rightarrow\mathbb{R}^\ell$. The philosophy behind these is that we want vertices which are ``similar'' (and hence should be similarly labelled) to have feature vectors that are ``close together''. What this means in practice will depend on the application, e.g. \cite{VZ} incorporates texture into the features and \cite{BF,birdspot} give other options.

Next, we construct the weights on the graph by deciding on the edge set $E$ (e.g. $E = V^2 \setminus \{(i,i)\mid i \in V\}$) and for each $ij\in E$ computing $\omega_{ij} := \Omega({z}_i,{z}_j)$ (and for $ij\notin E$, $\omega_{ij}:=0$). There are a number of standard choices for the similarity function $\Omega$, see for example \cite{BF,Yarov,BCM,ZMP}. The similarity function we will use in this paper is the Gaussian function:
\[
\Omega(z,w) := e^{-\frac{||z - w||^2}{\ell\sigma^2}}.
\]

Finally, from these weights we compute the graph Laplacian so that we can employ the graph ODEs discussed in the previous sections. In particular, we compute the \emph{normalised} (a.k.a. random walk) graph Laplacian, i.e. we will henceforth take $r=1$ and so $\Delta = I - D^{-1}\omega$. We will also consider the \emph{symmetric normalised} Laplacian $\Delta_s := I - D^{-1/2}\omega D^{-1/2}$, though this does not fit into the above framework.   This normalisation matters because, as discussed in \cite{BF}, the segmentation properties of diffusion-based graph dynamics are linked to the segmentation properties of the eigenvectors of the corresponding Laplacian. As shown in \cite[Fig. 2.1]{BF}, normalisation vastly improves these segmentation properties. As that figure looked at the symmetric normalised Laplacian, we include Fig. \ref{rwVSsym} to show the difference between the symmetric normalised and the random walk Laplacian. 
\begin{figure}[h]
\centering
\includegraphics[width = 0.35\textwidth]{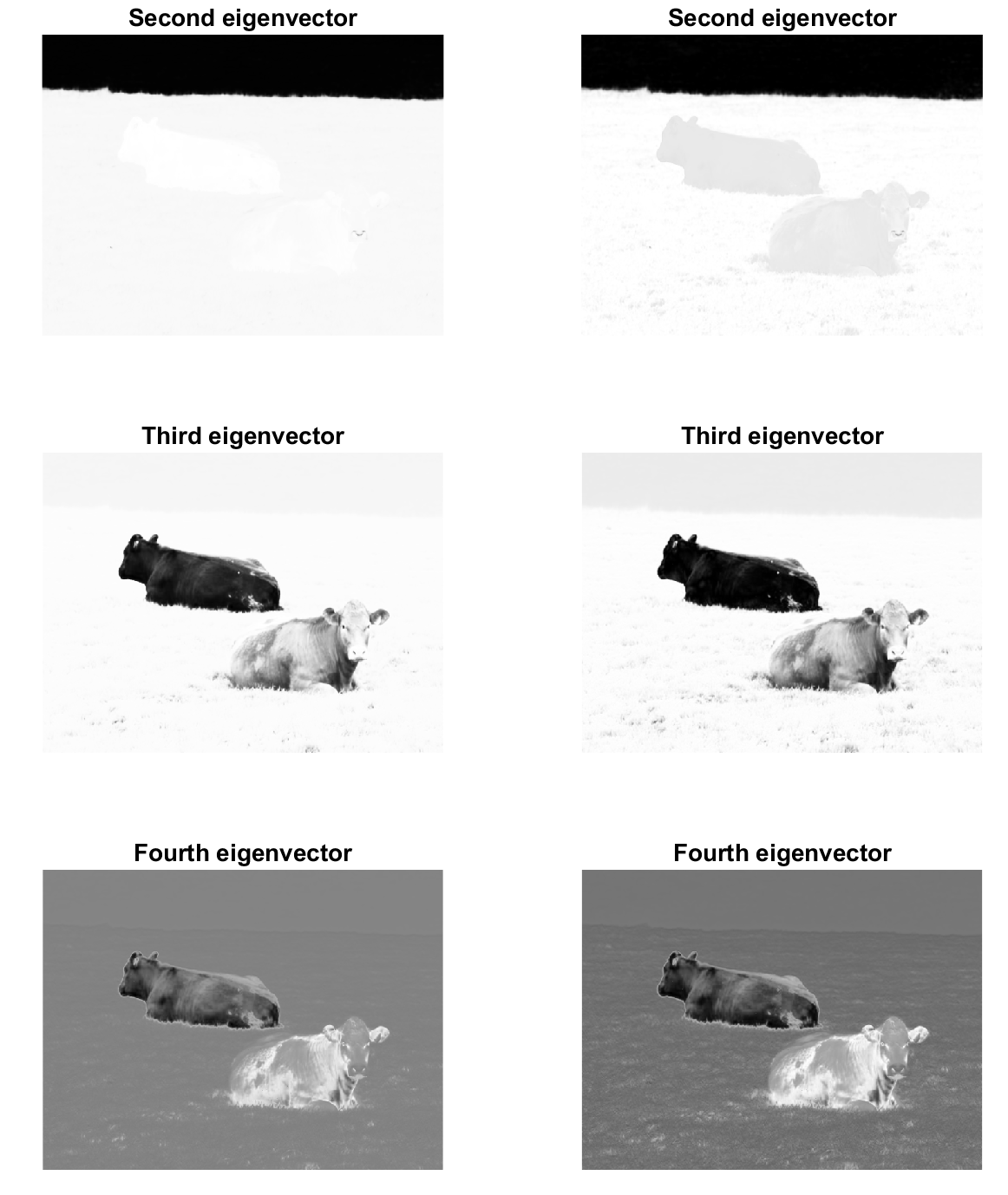}
\caption{Second, third, and fourth eigenvectors of the random walk Laplacian (left) and symmetric normalised Laplacian (right) for the graph built on one of the ``two cows'' images from Section \ref{applicationsec}, computed using Algorithm~\ref{nysQR}.}
\label{rwVSsym}
\end{figure} 
\subsection{The basic classification algorithm}
 For some time step $0<\tau\leq\varepsilon $ note that 
\[
\mathcal{S}_\tau u = e^{-\tau A}u + b
\]
where $b:= A^{-1}(I - e^{-\tau A})f$ is independent of $u$.

\begin{enumerate}
\item \textbf{Input:} Vector $x:V\rightarrow\mathbb{R}^q$, reference data $Z$, and labels $f$ supported on $Z$. 
\item Convert $x$ into feature vectors $z:V\rightarrow\mathbb{R}^\ell$. 
\item Build a weight matrix $\omega = (\omega_{ij})$ on $V^2$ via $\omega_{ij} :=  \Omega({z}_i,{z}_j)$.
\item Compute $A$ and therefore $b$.
\item From some initial condition $u_0$, compute the SDIE sequence $u_n$ until it meets a stopping condition at some $n = N$.
\item\textbf{Output:} $u_N$. 
\end{enumerate}

Unfortunately, as written this algorithm cannot be feasibly run. The chief obstacle is that in many applications $A$ is too large to store in memory, yet we need to quickly compute $e^{-\tau A} u$, potentially a large number of times. We also need to compute $b$ accurately. Moreover, in general $A$ does not have low numerical rank, so it cannot be well approximated by a low-rank matrix.
In the rest of this section we describe our  modifications to this basic algorithm that make it computationally efficient. 
\subsection{Matrix compression and approximate SVDs}

We will need to compress $\Delta$ into something we can store in memory. 
Following \cite{MKB,BF}, we employ the  Nystr\"om extension \cite{Nys,FBCM}. We choose $K\ll |V|$ to be the rank to which we will compress $\Delta$, and choose nonempty Nystr\"om interpolation sets $X_1\subseteq V\setminus Z$ and $X_2 \subseteq Z$ at random such that $|X|=K$ where $X:= X_1\cup X_2$. Then using the function $ij\mapsto \omega_{ij}$ we compute $\omega_{VX}:=\omega(V,X)$ (i.e. $\omega_{VX}:=(\omega_{ij})_{i\in V, j\in X}$) and $\omega_{XX}:=\omega(X,X)$ 
and then the Nystr\"om extension is the approximation:
\[
\omega \approx  \omega_{VX} \omega_{XX}^{-1} \omega_{VX}^T.
\]
Note that this avoids having to compute the full matrix $\omega$ which in many applications is too large to store in memory. We next compute an approximation for the degree vector $d$ and degree matrix $D:=\operatorname{diag}(d)$ of our graph
\begin{align*} 
&d\approx \hat d:=\omega_{V{X}} \omega_{{XX}}^{-1} \omega_{V{X}}^T\mathbf{1}, & D\approx \hat D := \operatorname{diag}\left(\hat d\right).
\end{align*}
We thus approximately normalise $\omega$
\begin{align*}
\tilde \omega &:=  D^{-1/2}\omega D^{-1/2} \approx \hat D^{-1/2}   \omega_{V{X}} \omega_{{XX}}^{-1}  \omega_{V{X}}^T \hat D^{-1/2} =   \tilde  \omega_{V{X}} \omega_{{XX}}^{-1} \tilde  \omega_{V{X}}^T
\end{align*}
where $\tilde  \omega_{V{X}}:= \hat D^{-1/2} \omega_{V{X}}$.

Following \cite{MKB,BF}, we next compute an approximate eigendecomposition of $\tilde\omega$. We here diverge from the method of \cite{MKB,BF}. The method used in those papers requires taking the matrix square root of $\omega_{XX}$, but unless $\omega_{XX}$ is the zero matrix it will not be positive semi-definite.\footnote{It is easy to see that non-zero $\omega_{XX}$ has negative eigenvalues, as it has zero trace.} Whilst this clearly does not prevent the method of \cite{MKB,BF} from working in practice, it is a potential source of error and we found it conceptually troubling. We here present an improved method, adapted from the method from \cite{BK} for computing a singular value decomposition (SVD) from an adaptive cross approximation (ACA) (see \cite{BK} for a definition of ACA).

First, we compute the thin QR decomposition (see \cite[Theorem 5.2.2]{GVL})
\[
\tilde \omega_{V{X}} = QR
\] 
where $Q\in\mathbb{R}^{|V|\times K}$ is orthonormal and $R\in\mathbb{R}^{K\times K}$ is upper triangular.
Next, we compute the eigendecomposition 
\[
R\omega_{{XX}}^{-1}R^T = \Phi \Sigma \Phi^T
\]
where $\Phi\in\mathbb{R}^{K\times K}$ is orthogonal and $\Sigma\in\mathbb{R}^{K\times K}$ is diagonal. It follows that $\tilde \omega$ has approximate eigendecomposition:
\[
\tilde \omega \approx Q\Phi\Sigma \Phi^T Q^T = U_s \Sigma U_s^T 
\]
where $U_s:=Q\Phi$ is orthonormal. This gives an approximate eigendecomposition of the symmetric normalised Laplacian
\[
\Delta_s = I - \tilde\omega \approx U_s(I_{K}-\Sigma)U_s^T = U_s\Lambda U_s^T
\]
where $I_K$ is the $K\times K$ identity matrix and $\Lambda := I_{K}-\Sigma$, and so we get an approximate SVD of the random walk Laplacian
\[
\Delta = D^{-1/2} \Delta_s D^{1/2} \approx U_1\Lambda U_2^T
\]
where $U_1:= \hat D^{-1/2}U_s$ and $U_2:=\hat D^{1/2}U_s$. As in \cite{BK}, it is easy to see that this approach is $\bigO(K|V|)$ in space and $\bigO(K^2|V|
+K^3)$ in time. We summarise this all as Algorithm \ref{nysQR}.
\begin{algorithm}[h]
 \caption{QR decomposition-based Nystr\"om method for computing an approximate SVD of $\Delta$, inspired by \cite{BK}.\label{nysQR}}
\begin{algorithmic}[1]
\Function{Nystr\"omQR}{$ij\mapsto\omega_{ij}, V, Z, K$}\Comment{Computes $U_1,\Lambda$, and $U_2$; $\Delta\approx U_1\Lambda U_2^T$ is an approximate SVD of rank $K$  }
\State $ X_1 = \texttt{random\_subset}(V\setminus Z,K/2) $ \Comment{A random subset of $V\setminus Z$ of size $K/2$}
\State $ X_2 = \texttt{random\_subset}(Z,K/2) $ \Comment{A random subset of $Z$ of size $K/2$}
\State $ {X} = X_1 \cup X_2$ 
\State $\omega_{{XX}}=\omega({X},{X})$ 
\State $\omega_{V{X}}=\omega(V,{X})$ 
\State $\hat d = \omega_{V{X}} \left(\omega_{{XX}}^{-1} \left(\omega_{V{X}}^T\mathbf{1}\right)\right) $ 
\State $ \tilde\omega_{VX} =\hat d^{-1/2} \odot \omega_{VX}$ \Comment{Applying $\odot$ columnwise, i.e. $(\tilde\omega_{V{X}})_{ij} = \hat d_i^{-1/2}(\omega_{V{X}})_{ij} $ }
\State $ [Q,R] = \texttt{thin\_qr}(\tilde \omega_{V{X}}) $\Comment{Computes thin QR decomposition $\tilde \omega_{V{X}}=QR$}
\State $ S = R\omega_{{XX}}^{-1}R^T$ 
\State $ S = (S + S^T)/2 $ \Comment{Corrects symmetry-breaking computational errors}
\State $ [\Phi,\Sigma] = \texttt{eig}(S)$\Comment{Computes eigendecomposition $S = \Phi\Sigma\Phi^T$}
\State $ \Lambda = I_K -\Sigma$ 
\State $ U_s = Q\Phi$ \Comment{Note that $\Delta_s\approx U_s \Lambda U_s^T$}
\State $ U_1 = \hat d^{-1/2} \odot U_s $ \Comment{I.e. $(U_1)_{ij} =  \hat d_i^{-1/2} (U_s)_{ij} $}
\State $  U_2 = \hat d^{1/2} \odot U_s $ \Comment{I.e. $(U_2)_{ij} =  \hat d_i^{1/2} (U_s)_{ij} $}
\State \textbf{return} $U_1,\Lambda,U_2$
\EndFunction
\end{algorithmic}
\end{algorithm}
\subsubsection{Numerical assessment of the matrix compression method}
We consider the accuracy of our Nystr\"om-QR approach for the compression of the symmetric normalised Laplacian $\Delta_s$ built on the simple image in Fig.~\ref{Fig_Image_Error_LR}, containing $|V| = 6400$ pixels, which is sufficiently small that we can compute the true value of $\Delta_s$ to high accuracy.
For $K \in \{50, 100, \ldots, 500\}$, we compare the rank $K$ approximation $U_s\Lambda U_s^T$ with the true $\Delta_s$ in terms of the relative Frobenius distance, i.e. $||U_s\Lambda U_s^T-\Delta_s||_F/||\Delta_s||_F$. 
Moreover, we compare these errors to the errors incurred by other low-rank approximations of $\Delta_s$, namely the  Nystr\"om method suggested by \cite{MKB,BF}, the Halko--Martinsson--Tropp (HMT) method\footnote{The HMT results serve only to give an additional benchmark for the Nystr\"om methods: HMT requires matrix-vector-products with $\Delta_s$, which was infeasible for us in applications. However, as we were finalising this paper we were made aware of the recent work of \cite{BSV}, which may make computing the HMT-SVD of $\Delta_s$ feasible.} \cite{HMT} (a randomised algorithm), and the rank $K$ approximation of $\Delta_s$ obtained by setting all but its $K$ leading singular values to $0$. By the Eckart--Young theorem \cite{EY} (see also \cite[Theorem~2.4.8]{GVL}) the latter is the optimal rank $K$ approximation of $\Delta_s$ with respect to the Frobenius distance.
In addition to the methods' accuracy, we measure their complexity in terms of the elapsed time used for their execution, obtained with an implementation in \textsc{Matlab}R2019a on the set-up described in Section~\ref{appsetup}.

\begin{figure}[ht]
\centering
\includegraphics[width = 0.2\textwidth]{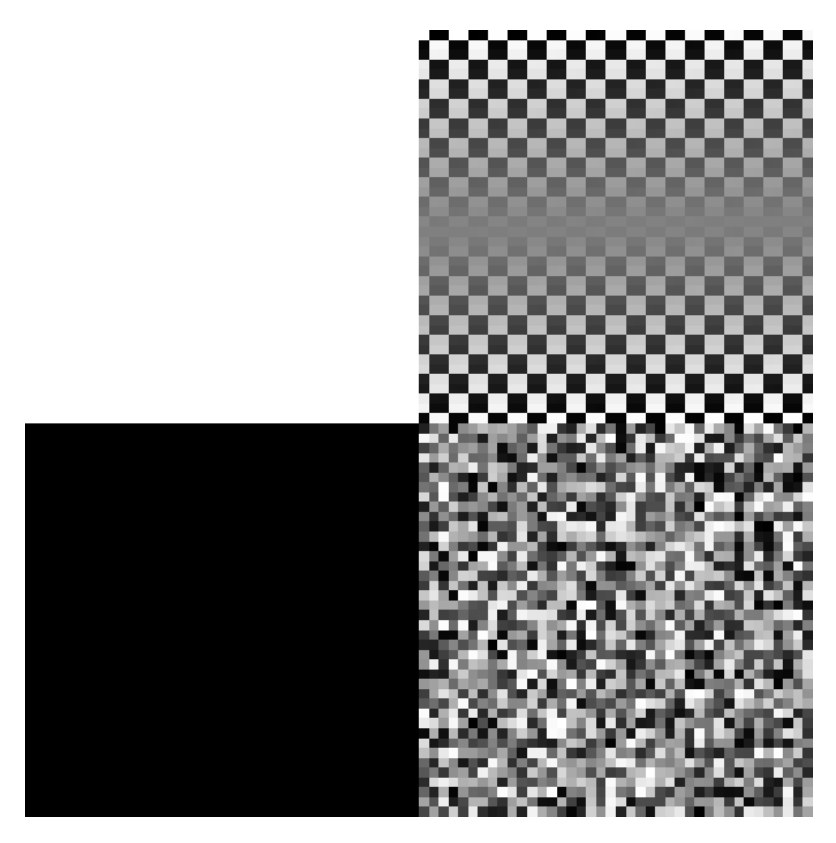}
\caption{The $80\times 80$ image on which the Laplacian $\Delta_s$ is constructed to test the low-rank approximations. }
\label{Fig_Image_Error_LR}
\end{figure}

We report the relative Frobenius distance in the left of Fig.~ \ref{Fig_Timings_Error_LR}. As the Nystr\"om (and HMT) methods are randomised, we repeat the experiments $100$ times and plot the mean error as solid lines and the mean error $\pm$ standard deviation as dotted lines. To expose the difference between the methods for $K \geq 100$, we subtract the SVD error from the other errors and show this difference in the central figure. In the right figure, we compare the complexity of the algorithms in terms of their average runtime. The SVD timing is constant in $K$ as we always computed a full SVD and kept the largest $K$ singular values.

\begin{figure}
\centering
\includegraphics[height=110pt]{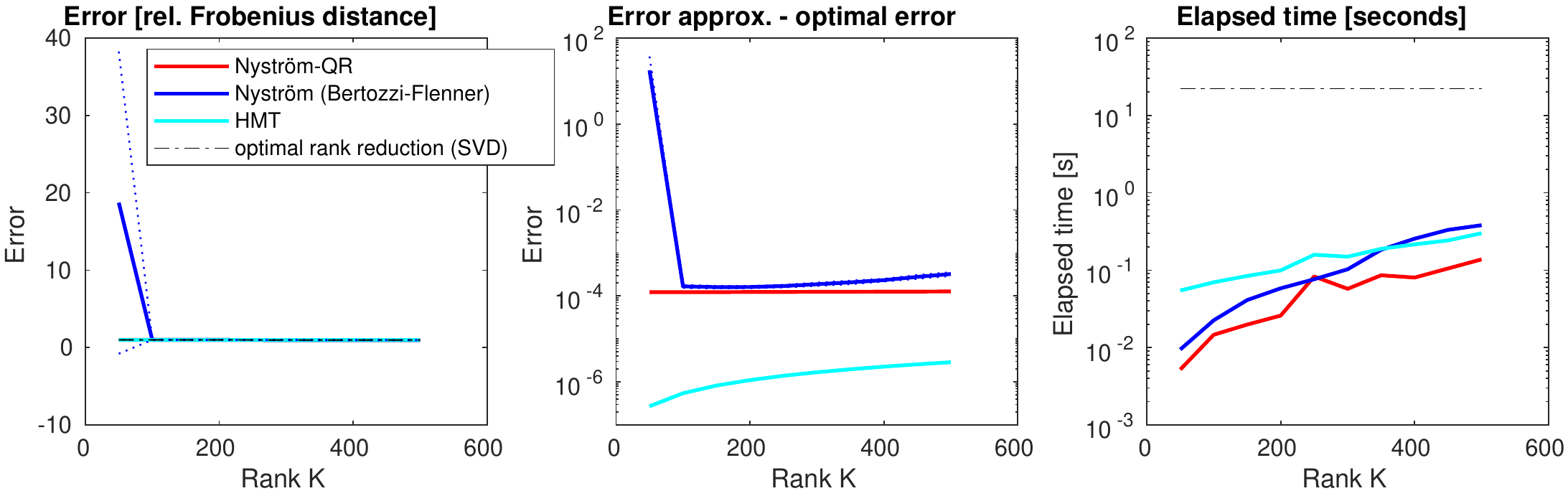}
\caption{Error of the low-rank approximations of $\Delta_s$ in terms of their relative Frobenius distance to the true $\Delta_s$ (left) and translated by the negative optimal error reached with the SVD (centre). Timings of the methods in seconds (right). In the left and the centre figure, the solid lines refer to the mean values and the dotted lines to the means $\pm$ the standard deviations. The lower blue dotted line in the centre figure is not plotted for low $K$ as the mean $-$ standard deviation is negative and cannot be plotted on a log scale; all other solid and dotted lines are plotted but some cannot be distinguished from the others by the eye.}
\label{Fig_Timings_Error_LR}
\end{figure}

We observe that the Nystr\"om-QR method outperforms the Nystr\"om method from \cite{MKB,BF}: it is faster, more accurate, and is stable for small $K$. In terms of accuracy, the Nystr\"om-QR error is equal to only $4.6 \times 10^{-5}$ plus the optimal low-rank error. Notably, this additional error is (almost) constant, indicating that the Nystr\"om-QR method and the SVD converge at a similar rate.
The Nystr\"om-QR randomness has hardly any effect on the error; the standard deviation of the relative error ranges from $8.87 \times 10^{-7}$ $(K=50)$ to $5.00 \times 10^{-8}$ $(K=500)$. By contrast, for the Nystr\"om method from \cite{MKB, BF} we see much more random variation.
\subsection{Implementing the SDIE scheme: a Strang formula method}

To compute the iterates of our SDIE scheme, we will need to compute an approximation for $\mathcal{S}_\tau u_n$. In \cite{MKB}, at each iteration $\mathcal{S}_\tau u_n$ was approximated via a semi-implicit Euler method, which therefore incurred a linear (in the time step of the Euler method, i.e. $\delta t$ in the below notation) error in both the $e^{-\tau A}u_n$ and $b$ terms (plus a spectrum truncation error). In Appendix \ref{MKBapp} we show that the method from \cite{MKB} works by approximating a Lie product formula approximation (see \cite[Theorem 2.11]{Hall}) of $e^{-\tau A}$, therefore we propose as an improvement a scheme that directly employs the superior Strang formula\footnote{We owe the suggestion to use this formula to Arieh Iserles, who also suggested to us the Yoshida method that we consider below.} to approximate $e^{-\tau A}u_n$---with quadratic error (plus a spectrum truncation error). We also consider potential improvements of the accuracy of computing $b$: by expressing $b$ as an integral and using quadrature methods;\footnote{We again thank Arieh Iserles for also making this suggestion.} by expressing $b$ as a solution to the ODE from \eqref{fdiffuse} with initial condition $\mathbf{0}$, and using the Euler method from \cite{MKB} with a very small time step (or a higher-order ODE solver);\footnote{We can afford to do this for $b$, but not generally for the $\mathcal{S}_\tau u_n$, because $b$ only needs to be computed once rather than at each $n$.} or by computing the closed form solution for $b$ directly using the Woodbury identity. We therefore improve on the accuracy of computing $\mathcal{S}_\tau u$ at low cost.

The Strang formula for matrix exponentials \cite{Strang} is given, for $k >0$ a parameter and $\bigO$ relative to the limit $k\rightarrow \infty$, by
\[
e^{X + Y} = (e^{\frac{1}{2}Y/k}e^{X/k}e^{\frac{1}{2}Y/k})^k +\bigO(k^{-2}). 
\]
Given $\Delta \approx  U_1\Lambda U_2^T $ as in Algorithm \ref{nysQR} (the case for $\Delta_s $ is likewise) for any $u\in \V$ we compute (writing $\delta t:= \tau/k$)
\be
\begin{split}
e^{-\tau A}u &= \left(e^{-\frac{1}{2}\tau/k M} e^{-\tau/k \Delta} e^{-\frac{1}{2}\tau/k M}\right)^k u+ \bigO(k^{-2})\\
	           &=  \left(e^{-\frac{1}{2}\delta t M}\left(I + U_1(e^{-\delta t \Lambda}-I_K)U_2^T\right) e^{-\frac{1}{2}\delta t M}\right)^k u+E+ \bigO(\delta t^2)\\
		&= v_k + E+ \bigO(\delta t^2)
\end{split}
\ee
where $E$ is a spectrum truncation error and $v_k$ is defined by $v_0 := u$, and $v_r$ for $r \in\{1,...,k\}$ is defined iteratively by 
\be
\label{SFscheme}
\begin{split}
v_{r+1} &= e^{-\delta t M}v_r + e^{-\frac{1}{2}\delta t M} U_1(e^{-\delta t \Lambda}-I_K)U_2^T e^{-\frac{1}{2}\delta t M}v_r\\
&= a_1(\delta t) \odot v_r + a_3(\delta t) \odot \left( U_1\left( a_2(\delta t) \odot \left(U_2^T\left(a_3(\delta t) \odot v_r\right) \right)\right) \right)\\
&=: S(\delta t)v_r
\end{split}
\ee
where $\odot$ is the Hadamard (i.e. elementwise) product, $a_1(\delta t):= \operatorname{exp}(-\delta t\mu)$, $a_2(\delta t) := \operatorname{exp}(-\delta t\operatorname{diag}(\Lambda))-\mathbf{1}_K$, and $a_3(\delta t):= \operatorname{exp}(-\frac{1}{2}\delta t\mu)$ is the elementwise square root of $a_1(\delta t)$ (where $\exp$ is applied elementwise, and $\mathbf{1}_K$ is the vector of $K$ ones). In Fig. \ref{LPFfig}, we verify on a simple image that this method has quadratic error (plus a spectrum truncation error) and outperforms the \cite{MKB} Euler method. Furthermore
\eqref{SFscheme} is as fast as \eqref{Eulersoln2} (i.e. a step of the \cite{MKB} Euler method). This is because by defining $\tilde a_1 := \mathbf{1}-\delta t \mu$ and $\tilde a_2 := (\mathbf{1}_K +\delta t \operatorname{diag}(\Lambda))^{-1}$ (applying the reciprocation elementwise), we can rewrite \eqref{Eulersoln2} as 
\[
v_{r+1} = U_1 \left( \tilde a_2 \odot \left( U_2^T \left(  \tilde a_1 \odot v_r + \mu\delta t f \right) \right)\right)
\] and so both \eqref{SFscheme} and \eqref{Eulersoln2} involve two $\bigO(NK)$ matrix multiplications and the vectors in \eqref{SFscheme} and \eqref{Eulersoln2} and are at most $N$-dimensional, hence the Hadamard products in \eqref{SFscheme} and \eqref{Eulersoln2} are all at most $\bigO(N)$ and so are not rate-limiting.

At the cost of extra $\bigO(NK)$ matrix multiplications, one can employ the method of Yoshida \cite{Yoshida} to increase the order of the (non-spectrum-truncation) error by 2.
If we set $\alpha_0 := -\sqrt[3]{2}/(2 - \sqrt[3]{2})$ and $\alpha_1:= 1/(2 - \sqrt[3]{2})$ then we can define the map 
\[
Y(\delta t) := S(\alpha_1 \delta t) \circ S(\alpha_0 \delta t) \circ S(\alpha_1 \delta t)
\]
which gives $Y^k(\delta t)u = e^{-\tau A} u +\bigO(\delta t ^4)$ plus a spectrum truncation error.\footnote{This method can be extended to give higher-order formulae of any even order, but consideration of those formulae is beyond the scope of this paper.} However, as can be seen in Fig. \ref{LPFfig}(c,d), the spectrum truncation error can make negligible any gain from using the Yoshida method over the Strang formula.

It remains to compute an approximation for 
$
b : = A^{-1}(I-e^{-\tau A}) f. 
$ 
It is easy to show that $b$ can be rewritten as the integral
\[
b = \int_0^\tau e^{-tA} f \; dt
\]
which we can approximate via a quadrature, e.g. applying respectively the trapezium, midpoint, or Simpson's rules we get 
\be\label{eq_quad}
b = \frac{1}{2}\tau(I + e^{-\tau A})f +\bigO(\tau^3) =\tau e^{-\frac{1}{2}\tau A} f + \bigO(\tau^3) = \frac{1}{6}\tau(I + 4e^{-\frac{1}{2}\tau A} + e^{-\tau A})f + \bigO(\tau^5) 
\ee
any of which we can approximate efficiently via the above methods. Furthermore, as we only need to compute $b$ once, we can take a large value, $k_b$, for $k$ in those methods. As is standard for quadrature methods, the accuracy can often be improved by subdividing the interval. For example, using Simpson's rule and subdividing into intervals of length $h:=\tau/(2m)$ we get 
\begin{align*}
 b &= \frac{1}{3}h \left(I+2\sum_{j=1}^{m-1} e^{-2jhA} + 4\sum_{j=1}^{m}e^{-(2j-1)hA} + e^{-\tau A} \right)f + \bigO(\tau h^4)\\
\intertext{which if $k_b = 2m$, i.e. the Simpson subdivision equals the Strang/Yoshida step number, can be approximated efficiently by}
b& = \frac{1}{3} h \left(f+2\sum_{j=1}^{m-1} v_{2j} + 4\sum_{j=1}^{m} v_{2j-1} + v_{2m} \right) + E_1 + E_2 + \bigO(\tau h^4)
\end{align*} where $v_r:=S^r(h) f$ with $E_1 = \bigO(\tau^2 h)$ or $v_r := Y^r(h) f$ with $E_1 = \bigO(\tau^2 h^3)$, and $E_2$ is the spectrum truncation error. Finally, we can also let \textsc{Matlab} compute its preferred quadrature using the in-built \texttt{integrate} function, using either the Strang formula or Yoshida method to compute the integrand. However, we found this to be very slow. 

Another method to compute $b$ is to solve an ODE. We note that, by \eqref{fdiffusesoln}, $b$ is the fidelity forced diffusion of $\mathbf{0}$ at time $\tau$, i.e.
	\begin{align*}&\frac{dv}{dt}(t) = -\Delta v(t) -Mv(t) + f,  &v(0)=\mathbf{0},\\
\intertext{has $v(\tau) = b$. Hence we can approximate $b$ by solving}
	&\frac{d\hat v}{dt}(t) = -U_1\Lambda (U_2^T \hat v(t)) -\mu \odot \hat v(t) + f,  &\hat v(0)=\mathbf{0},\end{align*}
via the semi-implicit Euler method from \cite{MKB}.
Since we only need to compute $b$ once we can choose a small time step, i.e. a time step of $\tau/k_b$ for $k_b$ large, for this Euler method. One could also choose a higher-order ODE solver for this same reason, however as \cite{MKB} notes this ODE is stiff, which we found causes standard solvers such as \texttt{ode45} (i.e. Dormand--Prince-(4, 5) \cite{DP}) to be inaccurate, and we ran into the issue of the \textsc{Matlab} stiff solvers requiring matrices too large to fit in memory.  

Finally, we can try to compute the formula for $b$ directly. By the Strang formula or Yoshida method, we can efficiently compute $g:=f - e^{-\tau A}f$. It remains to compute $b = A^{-1}g$. Given our approximation $\Delta \approx  U_1\Lambda U_2^T$, by the Woodbury identity \cite{Woodbury}
\[
A^{-1} = (\Delta + M)^{-1}  \approx (I - U_1\Sigma U_2^T + M)^{-1} = Y^{-1} - Y^{-1} U_1\left( -\Sigma^{+} + U_2^T Y^{-1} U_1 \right)^{-1} U_2^T Y^{-1}
\]
where $Y := I+ M$, superscript $+$ denotes the pseudoinverse, and recall that $\Sigma = I_K-\Lambda$. Then
\[
b = y \odot \left( g - U_1 h \right) 
\]
where $y:=(1+\mu)^{-1}$, reciprocation applied elementwise, and $h$ is given by solving
\[
\left( -\Sigma^{+} + U_2^T (y \odot U_1) \right)h = U_2^T (y \odot g)
\]
where we define $y \odot U_1$ as columnwise Hadamard multiplication, i.e. $(y \odot U_1)_{ij} := y_i(U_1)_{ij}$. 
We compare the accuracy of these approximations of $b$ in Table \ref{btable}, and observe that no method is hands-down superior. Table \ref{btable} also indicates that the likely reason for methods like Simpson's rule not performing as well as expected is that the spectrum truncation error is dominating.

Given these ingredients, it is then straightforward to compute the SDIE scheme sequence via Algorithm~\ref{SDalg}.
\begin{algorithm}[htp]
 \caption{The SDIE scheme via a Strang formula method.\label{SDalg}}
\algblock[SWITCH]{switch}{endswitch}
\algblockdefx[NAME]{START}{END}%
[1]{\textbf{switch} #1}{\textbf{end switch}}
\begin{algorithmic}[1]
\Function{SDIE}{$\mu$, $f$, $U_1,\Lambda,U_2$, $\tau$, $\varepsilon$, $u_0$, $k$, $k_b$, $K$, $\delta$} \Comment{Computes the terminus of the SDIE scheme.  }
\If{using the quadrature method}
\State $F_1: t \mapsto e^{-tA} f$ \Comment{Computed using Strang formula or Yoshida method,  with parameter $k_b$}
\State $ b = \texttt{quadrature}(F_1,[0,\tau]) $ \Comment{Approximates $\int_0^\tau F_1(t) \; dt$ by some quadrature method} 
\ElsIf {using the ODE method}
\State $ F_2: x \mapsto  \left(-U_1\Lambda (U_2^T x) -\mu \odot x + f\right)$ 
\State $\hat v = \texttt{ode\_solver}(F_2,[0,\tau],\mathbf{0})$ \Comment{Solves $\hat v'(t)=F_2(\hat v)$ on $[0,\tau]$ with $\hat v(0) = \mathbf{0}$} 
\State $ b = \hat v(\tau)$
\ElsIf {using the Woodbury identity method}
\State $g = f - e^{-\tau A} f$  \Comment{Computed using Strang formula or Yoshida method, with parameter $k_b$}
\State $ y = (1+ \mu)^{-1} $
\State $\Sigma = I_K-\Lambda$
\State $(-\Sigma^+ + U_2^T (y \odot U_1) )h = U_2^T(y \odot g)$ \Comment{Solving the linear system for $h$}
\State $b = y \odot (g - U_1 h)$ 
\EndIf
\State $a_1 = \operatorname{exp}(-\tau/k \mu)$ 
\State $a_2 = \operatorname{exp}(-\tau/k \operatorname{diag}(\Lambda))-\mathbf{1}_K$ 
\State $a_3 = \operatorname{sqrt}(a_1)$
\State $ n=0$ 
\While{$||u_n-u_{n-1}||_2^2/||u_n||_2^2\geq\delta$}
\State $v = u_n$ 
\For{$r = 1$ to $k$}
	\State $v =a_1 \odot v + a_3 \odot \left( U_1\left( a_2 \odot \left(U_2^T\left(a_3 \odot v\right) \right)\right) \right)$ \Comment{Strang formula iteration}
\EndFor
\State $v = v + b$ \Comment{Approximates $v = \mathcal{S}_\tau u_n$}
\For{$i \in V$}
\If{$v_i < \tau/2\varepsilon$}
   \State $(u_{n+1})_i = 0$ 
\ElsIf{$v_i \geq 1 - \tau/2\varepsilon$}
    \State $(u_{n+1})_i = 1$ 
  \Else
   \State $(u_{n+1})_i= \frac{1}{2} + \frac{v_i - 1/2}{1-\tau/\varepsilon}$ 
 \EndIf
\EndFor
\State $n = n +1$ 
 \EndWhile
\State \textbf{return} $u_n$
\EndFunction
\end{algorithmic}
\end{algorithm}
\subsubsection{Numerical assessment of methods}
In this section, we will build our graphs on the image in Fig. \ref{toyimage}, which has sufficiently few pixels that we can compute the true values of $e^{-\tau A} u$ (with $A$ here given by $\Delta_s+M$) and $b$ to high accuracy. 
\begin{figure}[ht]
\centering
         \includegraphics[width=0.20\textwidth]{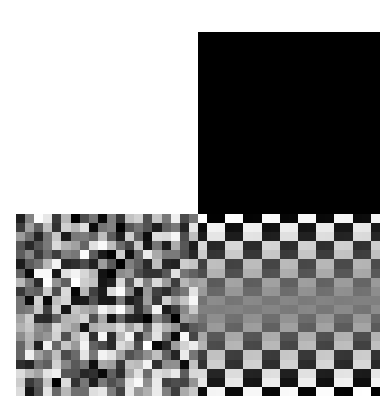}
\caption{The $40 \times 40$ or $80\times 80$ image that we build our graphs on, using feature vectors as described in Section \ref{appsetup}. } 
\label{toyimage}
\end{figure}
First, in Fig. \ref{LPFfig} we investigate the accuracy of the Strang formula and Yoshida method vs. the \cite{MKB} Euler method. We take $|V|=1600$, $\tau = 0.5$, $u$ a random vector given by \textsc{Matlab}'s \texttt{rand}(1600,1), and $\mu$ as the characteristic function of the left two quadrants of the image. We consider two cases: one where $K= |V|$ (i.e. full-rank) and one where $K=\sqrt{|V|}$. We observe that the Strang formula and Yoshida method are more accurate than the Euler method in both cases, and that the Yoshida method is more accurate than the Strang formula, but only barely in the rank-reduced case. Furthermore, the log-log gradients of the Strang formula error and the Yoshida method error (excluding the outliers for small $k$ values and the outliers caused by errors from reaching machine precision) in Fig. \ref{LPFfig}(b) are respectively 2.000 and 3.997 (computed using \texttt{polyfit}), confirming that these methods achieve their theoretical orders of error in the full-rank case.
\begin{figure}[ht]
\centering
\begin{subfigure}{0.48\textwidth}
\centering
         \includegraphics[width=\textwidth]{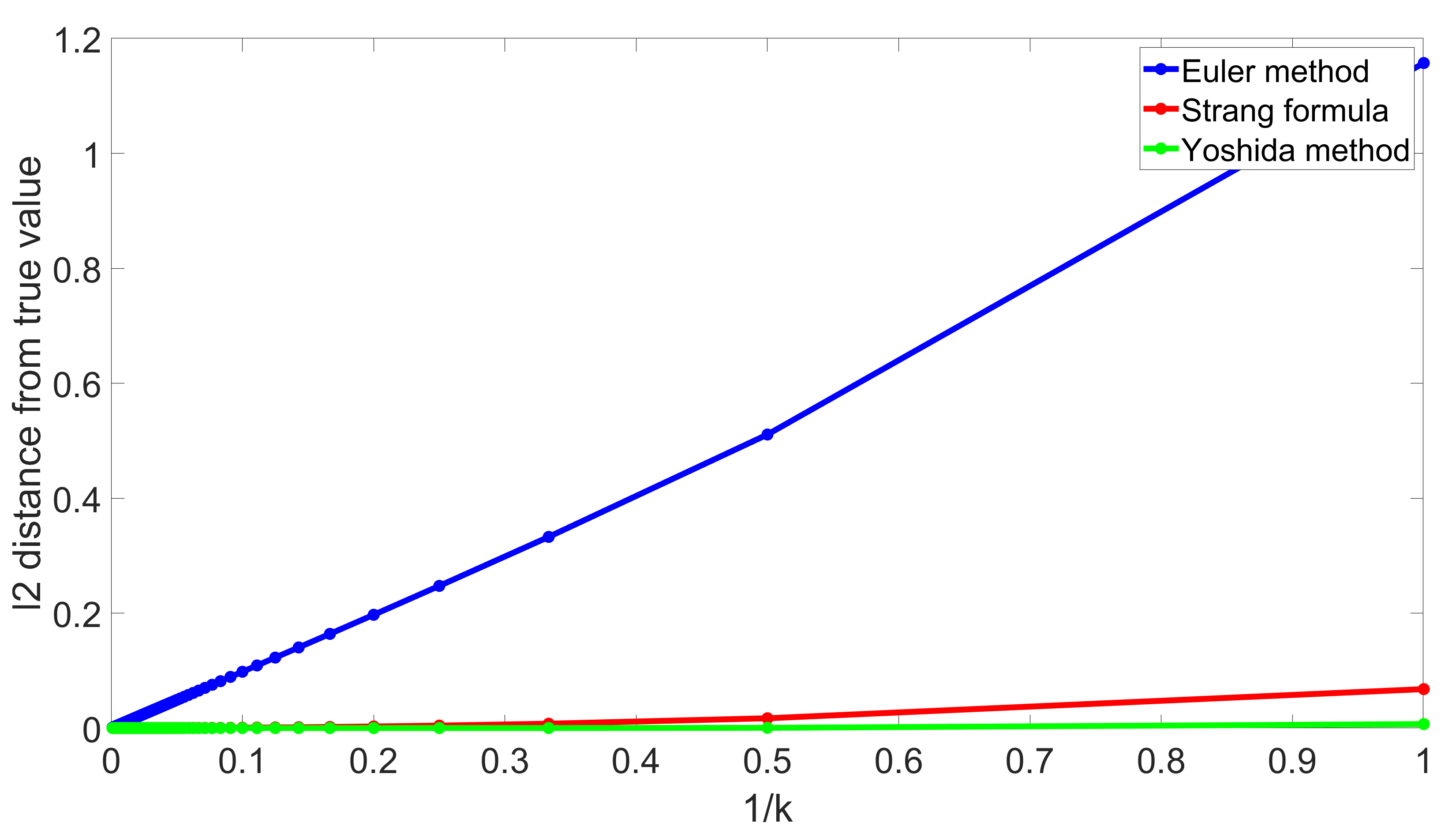}
\caption{Plot of $\ell^2$ errors vs. $1/k$.}
\end{subfigure}
\begin{subfigure}{0.48\textwidth}
\centering
         \includegraphics[width=\textwidth]{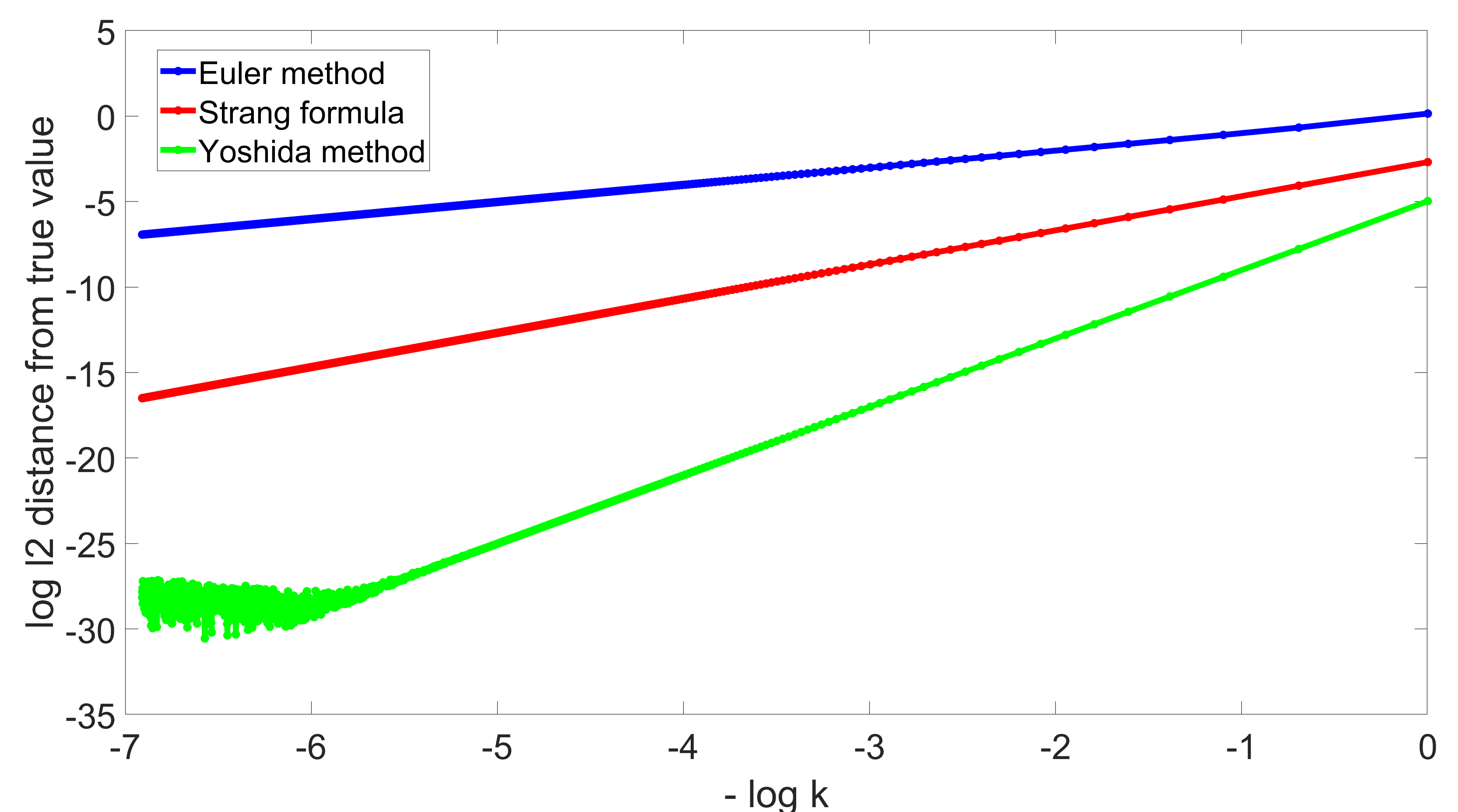}
\caption{Log-log plot, i.e. plot of log $\ell^2$ errors vs. $-\log k$.}
\end{subfigure}\\
\vspace{0.8em}
\begin{subfigure}{0.48\textwidth}
\centering
         \includegraphics[width=\textwidth]{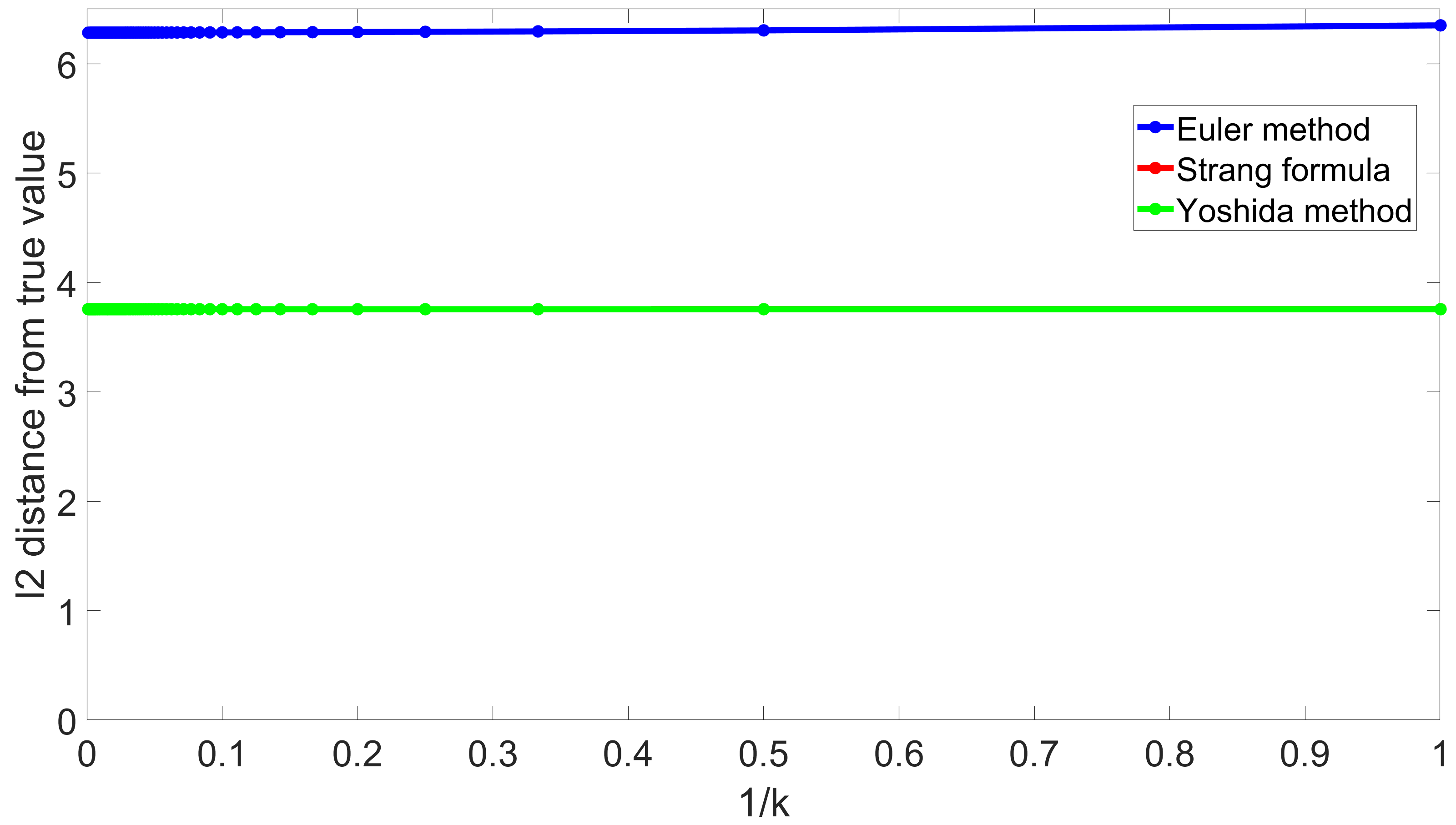}
\caption{Plot of $\ell^2$ errors vs. $1/k$ in the rank-reduced case.}
\end{subfigure}
\begin{subfigure}{0.48\textwidth}
\centering
         \includegraphics[width=\textwidth]{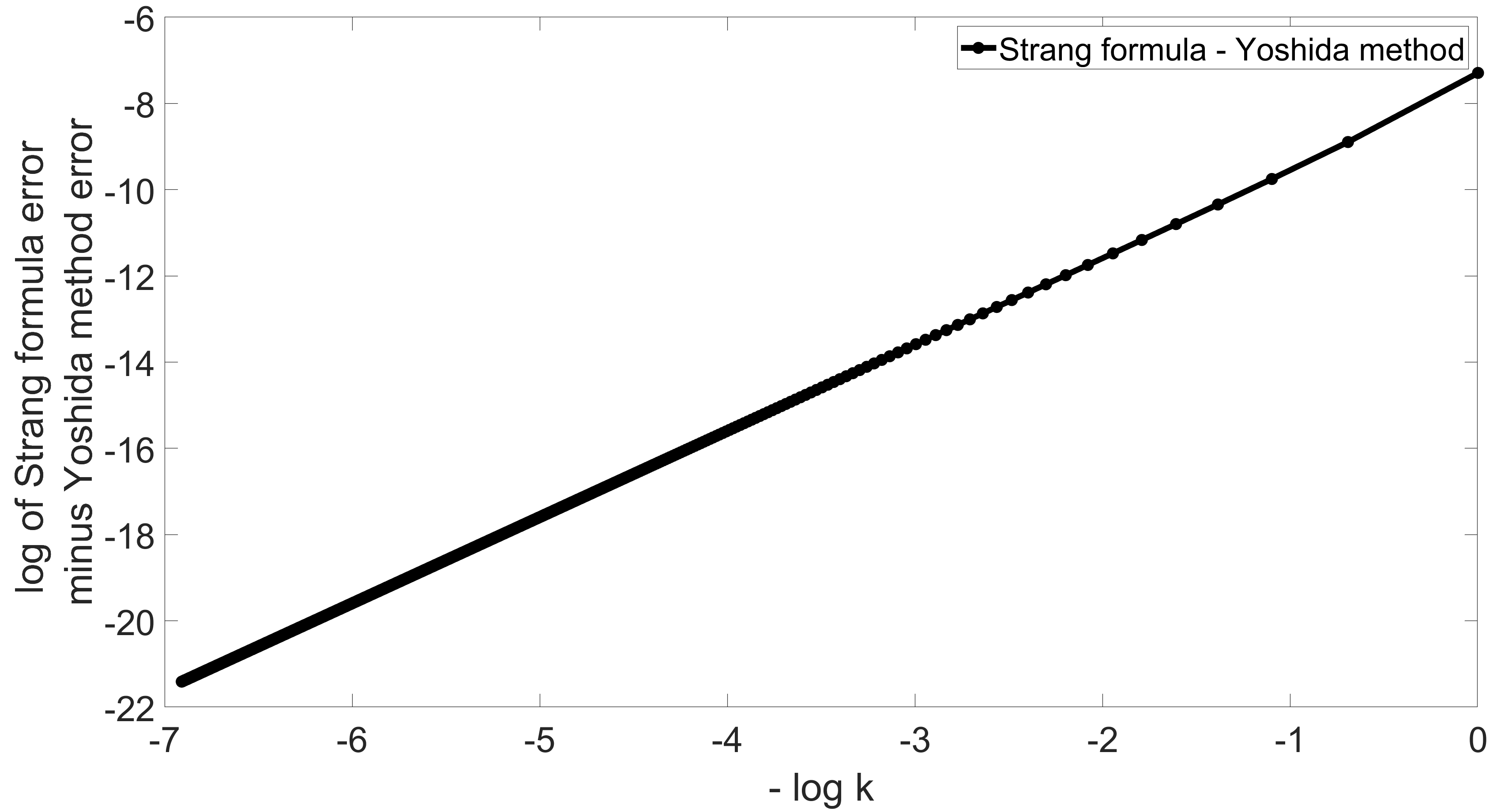}
\caption{Log-log plot of Strang formula error minus Yoshida formula error vs. $1/k$ in the rank-reduced case.}
\end{subfigure}
\caption{Comparison of the $\ell^2$ error from approximating $e^{-\tau A}u$ via the semi-implicit Euler method (blue), Strang formula (red), and Yoshida method (green) approximations for $e^{-\tau A}u$ on the graph built on the image in Fig. \ref{toyimage}. For (a) and (b), $K=|V|$. For (c) and (d), $K=\sqrt{|V|}$. The gradient of the line in (d), excluding outliers for small $k$, is 2.000. In (c), the green and red lines are both plotted but cannot be distinguished by the eye.   }
\label{LPFfig} 
\end{figure}

Next, in Table \ref{btable} we compare the accuracy of the different methods for computing $b$. We take $Z$ as the left two quadrants of the image, $\mu=\chi_Z$, $\tilde f$ as equal to the image on $Z$, and $k_b=1000$ in the Strang formula/Yoshida method approximations for $ e^{-t A}f$ and in the \cite{MKB} Euler scheme. We observe that the rank reduction plays a significant role in the errors incurred, and that no method is hands-down superior. In the ``two cows'' application (Example~\ref{ex_twocows}), we have observed that (interestingly) the \cite{MKB} Euler method yields the best segmentation. A topic for future research can be whether this is generally true for large matrices. 

\begin{table}[ht]
\hspace{-5.5em}
\begin{tabular}{|c|c|c|c|c|c|c|}
\hline
\multirow{2}{*}{\textbf{Method}}                                                                 & \multicolumn{3}{c|}{\textbf{Relative $\ell^2$ error for $\tau = 0.5$}}                                                                                                                       & \multicolumn{3}{c|}{\textbf{Relative $\ell^2$ error for $\tau=4$}}                                                                                                                           \\ \cline{2-7} 
                                                                                                 & \begin{tabular}[c]{@{}c@{}}$|V|=1600$,\\ $K=1600$\end{tabular} & \begin{tabular}[c]{@{}c@{}}$|V|=1600$,\\ $K=40$\end{tabular} & \begin{tabular}[c]{@{}c@{}}$|V|=6400$,\\ $K=40$\end{tabular} & \begin{tabular}[c]{@{}c@{}}$|V|=1600$,\\ $K=1600$\end{tabular} & \begin{tabular}[c]{@{}c@{}}$|V|=1600$,\\ $K=40$\end{tabular} & \begin{tabular}[c]{@{}c@{}}$|V|=6400$,\\ $K=40$\end{tabular} \\ \hline
Semi-implicit Euler  \cite{MKB}                                                                            & $1.434\times 10^{-4}$                                                         & 0.4951                                                       & 0.4111                                                       & $2.882\times 10^{-4}$                                                         & 0.2071                                                       & 0.1721                                                       \\ \hline
Woodbury identity                                                                                 & $2.292\times 10^{-8}$                                         & 0.5751                                                       & 0.4607                                                       & $1.038\times 10^{-7}$                                         & \textbf{0.1973}                                                       & \textbf{0.1537}                                                       \\ \hline
Midpoint rule \eqref{eq_quad}                                                                   & 0.0279                                                         & \textbf{0.1290}                                                       & \textbf{0.1110}                                                       & 0.4279                                                         & 0.6083                                                       & 0.6113                                                       \\ \hline
\begin{tabular}[c]{@{}c@{}}Simpson's rule ($m=500$) \\ via Strang formula\end{tabular}      & $2.592\times 10^{-9}$                                         & 0.1335                                                       & 0.1136                                                       & $5.845\times 10^{-7}$                                         & 0.5124                                                       & 0.4827                                                       \\ \hline
\begin{tabular}[c]{@{}c@{}}\textsc{Matlab} \texttt{integrate} \\ via Strang formula\end{tabular} & n/a                                                            & 0.1335                                                       & 0.1136                                                       & n/a                                                            & 0.5124                                                       & 0.4827                                                       \\ \hline
\begin{tabular}[c]{@{}c@{}}Simpson's rule ($m = 500$) \\ via Yoshida method\end{tabular}    &$\mathbf{8.381\times 10^{-14}}$                                        & 0.1335                                                       & 0.1136                                                       & $\mathbf{9.335\times 10^{-12}}$                                        & 0.5124                                                       & 0.4827                                                       \\ \hline
\begin{tabular}[c]{@{}c@{}}\textsc{Matlab} \texttt{integrate} \\ via Yoshida method\end{tabular} & n/a                                                            & 0.1335                                                       & 0.1136                                                       & n/a                                                            & 0.5124                                                       & 0.4827                                                       \\ \hline
\end{tabular}
\caption{Comparison of the relative $\ell^2$ errors from the methods for approximating $b$ on the image from Fig. \ref{toyimage}. We did not compute \texttt{integrate} for $K=1600$ as it ran too slowly. Bold entries indicate the smallest error in that column. }
\label{btable}
\end{table}
\newpage
\section{Applications in image processing}\label{applicationsec}
\subsection{Examples}
We consider three examples, all using images of cows from the Microsoft Research Cambridge Object Recognition Image Database\footnote{Available at \texttt{https://www.microsoft.com/en-us/research/project/image-understanding/} accessed 20 October 2020.}). Some of these images have been used before by \cite{MKB,BF} to illustrate and test graph-based segmentation algorithms.

\begin{example}[Two cows]\label{ex_twocows} We first introduce the \emph{two cows} example familiar from \cite{MKB,BF}. 
We take the image of two cows in the top left of Fig.~\ref{Fig_twocows} as the reference data $Z$ and the segmentation in the bottom left as the reference labels $\tilde f$, which separate the cows from the background. 
We apply the SDIE scheme to segment the image of two cows shown in the top right of Fig.~\ref{Fig_twocows}, aiming to separate the cows from the background, and compare to the ground truth in the bottom right.
Both images are RGB images of size $480 \times 640$ pixels, i.e. the reference data and the image are tensors of size $480 \times 640 \times 3$.
\end{example}

\begin{figure}[ht]
\centering
\includegraphics[scale=0.45]{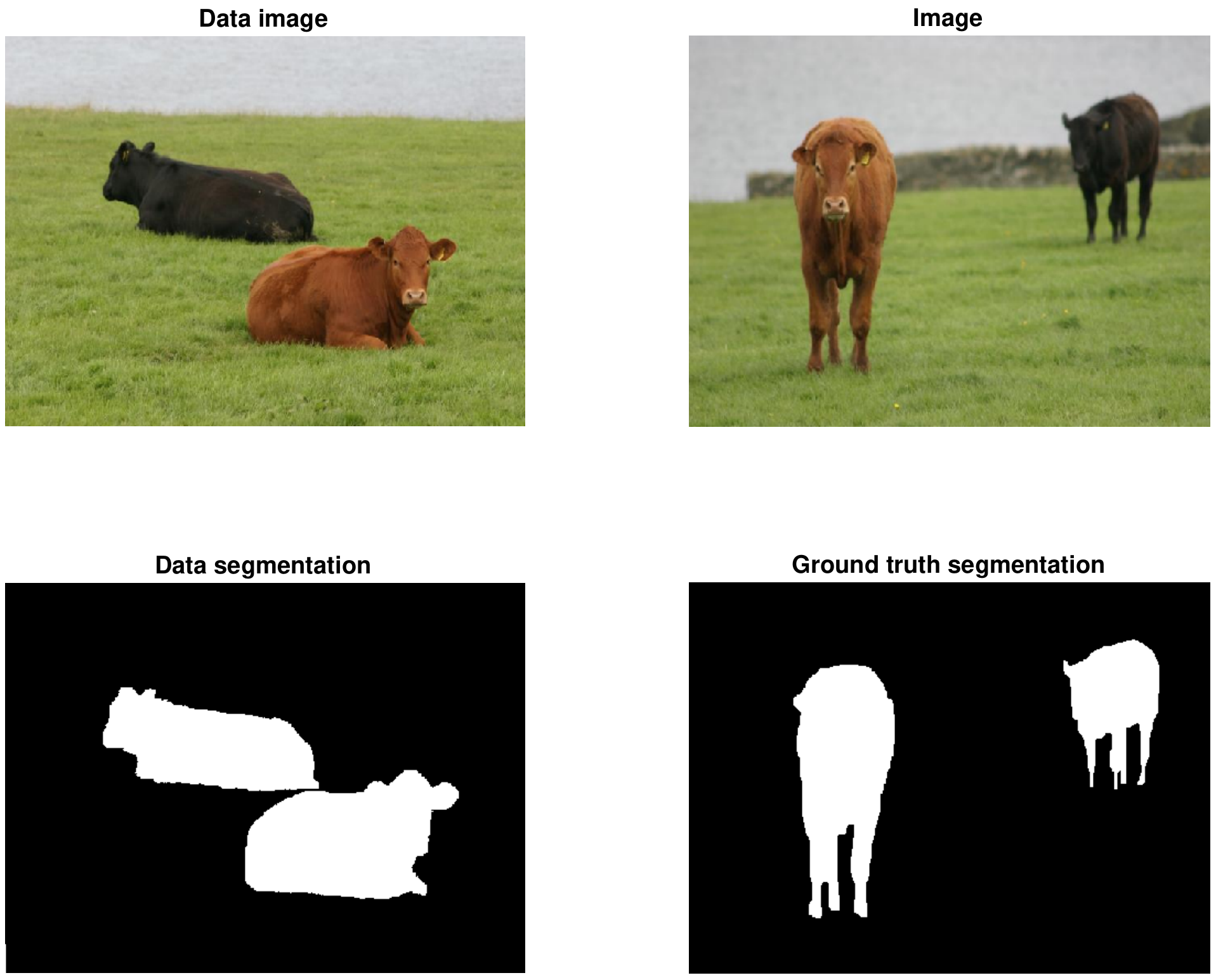}
\caption{Two cows: the reference data image, the reference $\tilde f$, the image to be segmented, and the ground truth segmentation associated to Example~\ref{ex_twocows}. We drew both segmentations by hand.} \label{Fig_twocows}
\end{figure}
We will use Example~\ref{ex_twocows} to illustrate the application of the SDIE scheme. Moreover, we will run several numerical experiments on this example. Namely, we will:
\begin{itemize}
\item study the influence of the parameters $\varepsilon$ and $\tau$, comparing non-MBO SDIE ($\tau<\varepsilon$) and MBO SDIE ($\tau = \varepsilon$);
\item compare different normalisations of the graph Laplacian, i.e. the symmetric vs. random walk normalisation;
\item investigate the influence of the Nystr\"om-QR approximation of the graph Laplacian in terms of the rank $K$; and
\item quantify the inherent uncertainty in the computational strategy induced by the randomised Nystr\"om approximation.
\end{itemize}

\begin{example}[Greyscale] \label{ex_twocows_gs}
This example is the greyscale version of Example~\ref{ex_twocows}. Hence, we map the images in Fig.~\ref{Fig_twocows} to greyscale using {\rm\texttt{rgb2gray}}. We show the greyscale images in Fig.~\ref{Fig_twocows_gs}. We use the same segmentation of the reference data image as in Example~\ref{ex_twocows}.
The greyscale images are matrices of size $480 \times 640$.
\end{example}
\begin{figure}[h]
\centering
\includegraphics[scale=0.65]{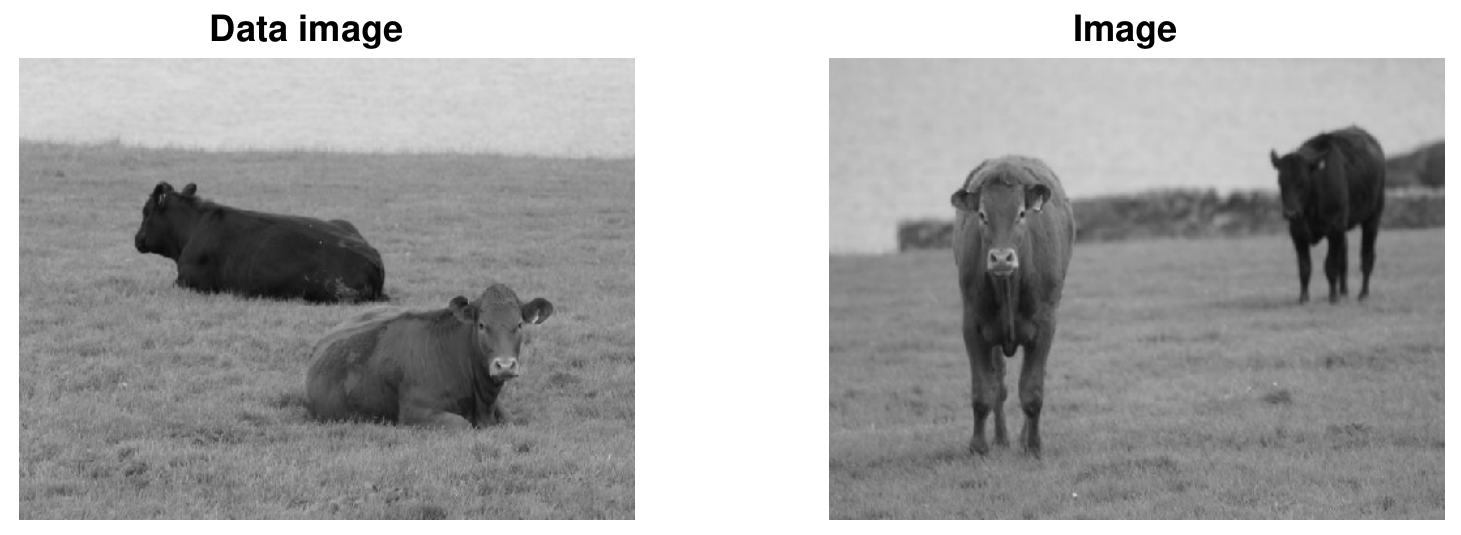}

\caption{Two cows greyscale: the reference data image and the image to be segmented associated to Example~\ref{ex_twocows_gs}. Note that the reference $\tilde f$ and the ground truth segmentation are identical to those in Fig.~\ref{Fig_twocows}.} \label{Fig_twocows_gs}
\end{figure}

The greyscale image is much harder to segment than the RGB image, as there is no clear colour separation. With Example~\ref{ex_twocows_gs}, we aim to illustrate the performance of the SDIE scheme in a harder segmentation task.

\begin{example}[Many cows] \label{ex_manycows}
In this example, we have concatenated four images of cows that we aim to segment as a whole. We show the concatenated image in Fig.~\ref{Fig_twocows_large}.
Again, we shall separate the cows from the background. As reference data, we use the reference data image and labels as in Example~\ref{ex_twocows}. Hence, the reference data is a tensor of size $480 \times 640 \times 3$. The image consists of approximately $1.23$ megapixels. It is represented by a tensor of size $480 \times 2560 \times 3$.
\end{example}
\begin{figure}[h]
\centering
\includegraphics[width=\textwidth]{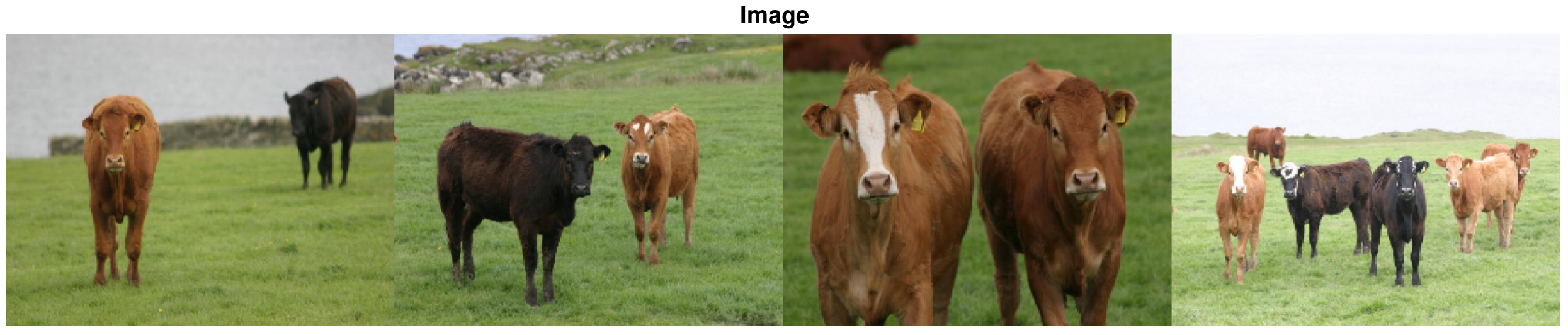}

\caption{Many cows: the image to be segmented associated to Example~\ref{ex_manycows}. Note that the reference data image and labels are identical to those in Fig.~\ref{Fig_twocows} (left).} \label{Fig_twocows_large}
\end{figure}
With Example~\ref{ex_manycows}, we will illustrate the application of the SDIE scheme to large scale images, as well as the case where the image and reference data are of different sizes.
\begin{nb}
In each of these examples we took as reference data a separate reference data image. However, our algorithm does not require this, and one could take a subset of the pixels of a single image to be the reference data, and thereby investigate the impact of the relative size of the reference data on the segmentation, which is beyond the scope of this paper but is explored for the \cite{MKB} MBO segmentation algorithm and related methods in \cite[Fig. 4]{BayesianGraphs}.
\end{nb}
\subsection{Set-up}\label{appsetup}
\paragraph{Algorithms}
We here use the Nystr\"om-QR method to compute the rank $K$ approximation to the Laplacian, and we use the \cite{MKB} semi-implicit Euler method (with time step $\tau/k_b$) to compute $b$ (as we found that in practice this worked best for the above examples).
\paragraph{Feature vectors}
Let $\mathcal{N}(i)$ denote the $3\times 3$ neighbourhood of pixel $i\in V$ in the image (with replication padding at borders performed via \texttt{padarray}) and let $\mathcal{K}$ be a $3\times 3$ Gaussian kernel with standard deviation 1 (computed via \texttt{fspecial}(\texttt{`gaussian'},3,1)). Thus $x|_{\mathcal{N}(i)}$ can be viewed as a triple of $3\times 3$ matrices $x^J|_{\mathcal{N}(i)}$ for $J\in\{R,G,B\}$ (i.e. one in each of the R, G, and B channels). Then in each channel we define \begin{align*}
&z_i^R := 9\mathcal{K}\odot x^R|_{\mathcal{N}(i)}, &z_i^G := 9\mathcal{K}\odot x^G|_{\mathcal{N}(i)},& &z_i^B := 9\mathcal{K}\odot x^B|_{\mathcal{N}(i)},
\end{align*}
and thus define $z_i := (z_i^R,z_i^G,z_i^B)\in\mathbb{R}^{3\times 3\times 3}$, which we reshaped (using \texttt{reshape}) so that $z\in\mathbb{R}^{|V|\times 27}$. 
\paragraph{Interpolation sets}
For the interpolation sets $X_1,X_2$ in Nystr\"om, we took $K/2$ vertices from the reference data image and $K/2$ vertices from the image to be segmented, chosen at random using \texttt{randperm}. We experimented with choosing interpolation sets using ACA (see \cite{BK}), but this showed no improvement over choosing random sets, and ran much slower. 
\paragraph{Initial condition}
We took the initial condition, i.e. $u_0$, to equal the reference $\tilde f$ on the reference data vertices and to equal 0.49 on the vertices of the image to be segmented (where $\tilde f$ labels `cow' with 1 and `not cow' with 0). We used 0.49 rather than the more natural 0.5 because the latter led to much more of the background (e.g. the grass) getting labelled as `cow'. This choice can be viewed as incorporating the slight extra \emph{a priori} information that the image to be segmented has more non-cow than cow.
\paragraph{Fidelity parameter}
We followed \cite{MKB} and took $\mu = \hat \mu \chi_Z$, for $\hat\mu>0$ a parameter. 
\paragraph{Computational set-up}
All programming was done in \textsc{Matlab}R2019a with relevant toolboxes the Computer Vision Toolbox Version 9.0, Image Processing Toolbox Version 10.4, and Signal Processing Toolbox Version 8.2.
All reported runtimes are of implementations executed serially on a machine with an Intel\textsuperscript{\textregistered}  Core\textsuperscript{\texttrademark} i7-9800X @ 3.80 GHz [16 cores] CPU and 32 GB RAM of memory.

\subsection{Two cows}
We begin with some examples of segmentations obtained from the SDIE scheme. Based on these, we illustrate the progression of the algorithm and discuss the segmentation output qualitatively. We will give a quantitative analysis in Section~\ref{subsubse_error_timing_2cows}. Note that we give here merely \emph{typical} realisations of the random output of the algorithm---the output is random due to the random choice of interpolation sets in the Nystr\"om approximation. We investigate the stochasticity of the algorithm in Section~\ref{subsubse_uncertain}.

\begin{figure}[hpt]
\centering
\begin{subfigure}{0.49\textwidth}
\centering
         \includegraphics[width=\textwidth]{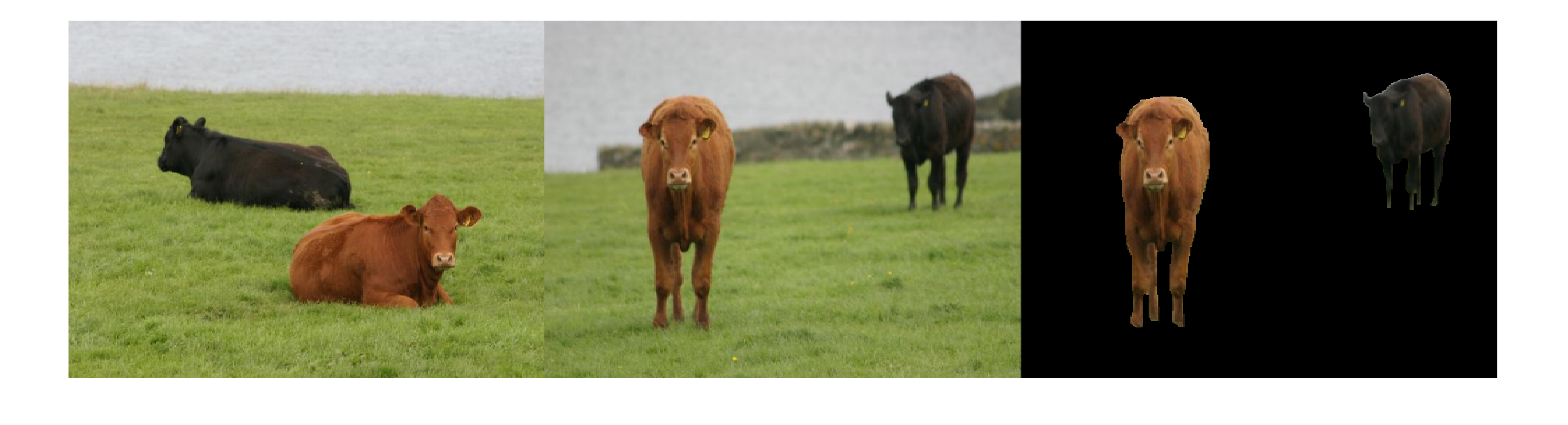}\vspace{-0.35cm}
         \caption{Reference data image, image to be segmented, and \break image  masked with ground truth segmentation.}
\end{subfigure} 
\begin{subfigure}{0.49\textwidth}
\centering
         \includegraphics[width=\textwidth]{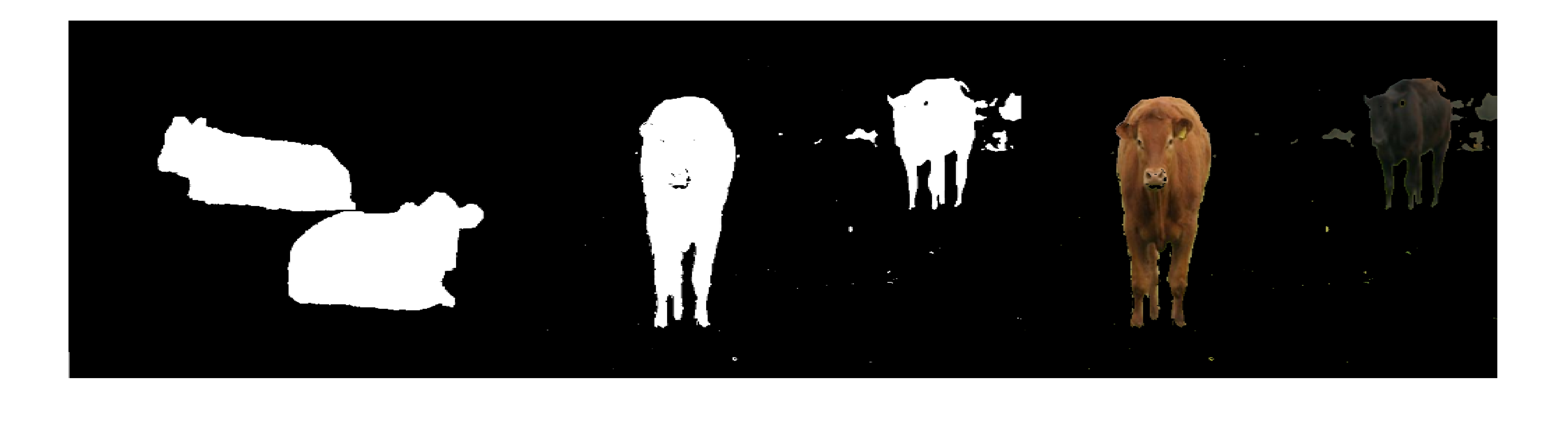}\vspace{-0.35cm}
\caption{$\tau = 0.001 \ll 0.003 = \varepsilon$. Relative segmentation \break error: 1.6416\%, elapsed time: 13.0 sec.}
\end{subfigure} \\
 \begin{subfigure}{0.49\textwidth}
\centering
         \includegraphics[width=\textwidth]{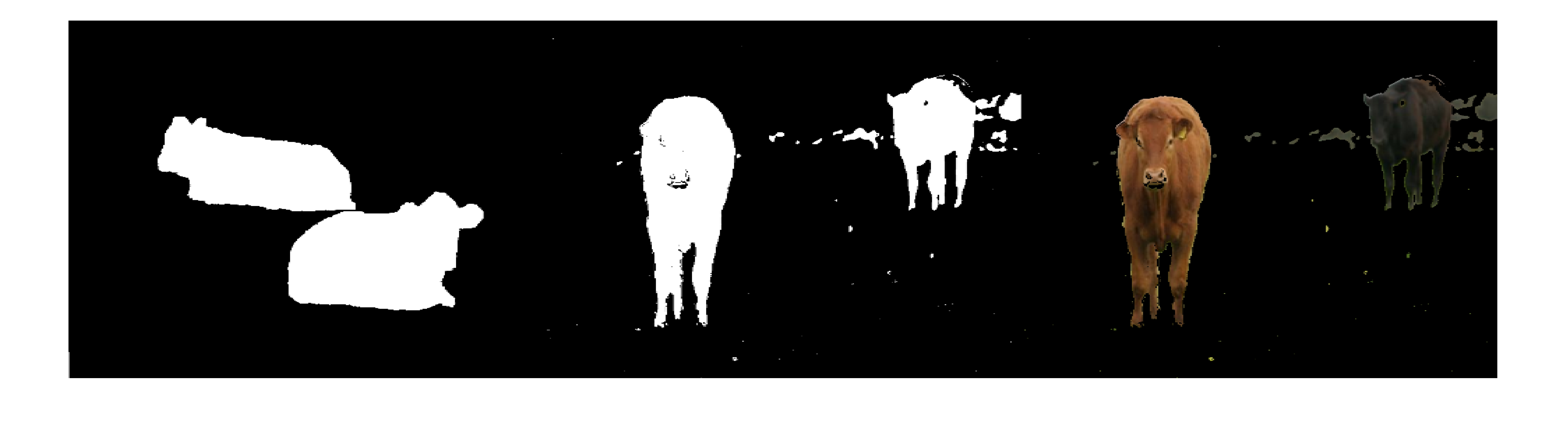}\vspace{-0.35cm}
\caption{$\tau = 0.0025 < 0.003 =\varepsilon$. Relative segmentation \break error: 1.9424\%, elapsed time: 4.9 sec.}
\end{subfigure}
 \begin{subfigure}{0.49\textwidth}
\centering
         \includegraphics[width=\textwidth]{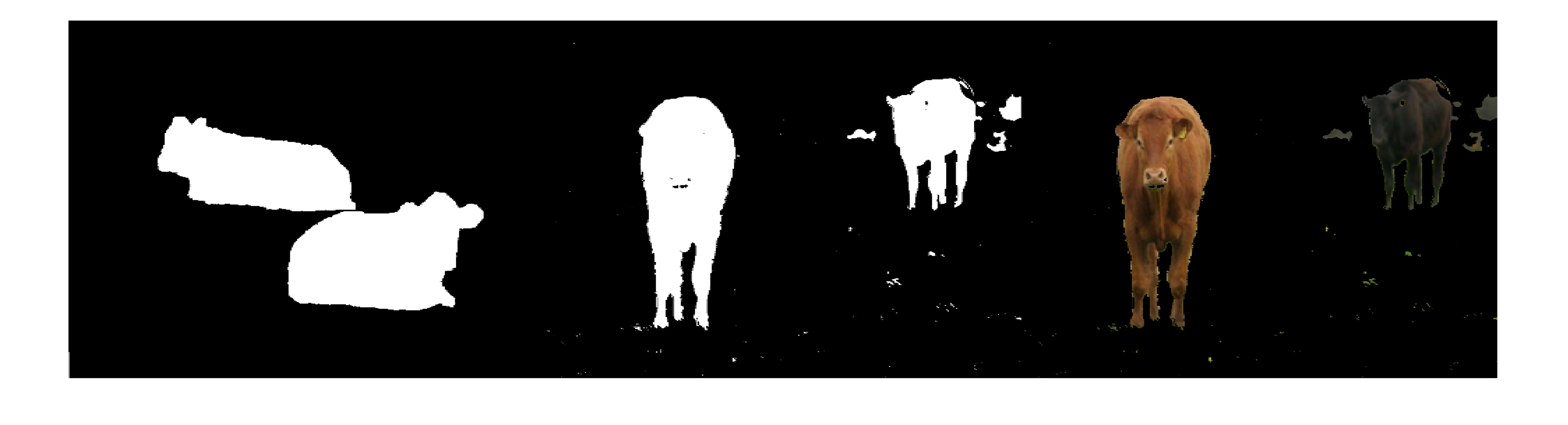}\vspace{-0.35cm}
\caption{$\tau = \varepsilon = 0.003$ (following \cite{MKB}). Relative segmentation error: 1.5378\%, elapsed time: 2.4 sec.}
\end{subfigure}
\caption{MBO SDIE and non-MBO SDIE segmentations for the \emph{two cows} segmentation task. In the top left figure, we show the reference data image, the image to be segmented, and the image masked with the segmentation we consider the ground truth, see also Fig.~\ref{Fig_twocows}. The other figures (b)-(d) show the labels on the reference data, the segmentation returned by the respective algorithm, and the original images masked with the segmentation.}
\label{Fig_Results_MBO+SDIE} 
\end{figure}

We consider three different cases: the MBO case $\tau = \varepsilon$ and two non-MBO cases, where $\tau \ll \varepsilon$, and $\tau < \varepsilon$. We show the resulting reconstructions from these methods in Fig.~\ref{Fig_Results_MBO+SDIE}. Moreover, we show the progression of the algorithms in Fig. \ref{Fig_ProgressionMBO}. The parameters not given in the captions of Fig.~\ref{Fig_Results_MBO+SDIE}  are $\hat\mu = 30$, $\sigma = 35$, $k_b = 1$, $k = 1$, $\delta = 10^{-10}$, and $K = 70$.    

\begin{nb} The regime $\tau > \varepsilon$ is not of much interest since, by \cite[Remark 4.8]{Budd} \emph{mutatis mutandis}, in this regime the SDIE scheme has non-unique solution for the update, of which one is just the MBO solution.
\end{nb}

Comparing the results in Fig.~\ref{Fig_Results_MBO+SDIE}, we see roughly equivalent segmentations and segmentation errors. Indeed, the cows are generally nicely segmented in each of the cases. However, the segmentation also labels as `cow' a part of the wall in the background and small clumps of grass, while a small part of the left cow's snout is cut out. This may be because the reference data image does not contain these features and so the scheme is not able to handle them correctly. 

In Fig.~\ref{fig:usvsBFMKB} we compare the result of Fig.~\ref{Fig_Results_MBO+SDIE}(d) (our best segmentation) with the results of the analogous experiments in \cite{MKB,BF}. We observe  significant qualitative improvement. In particular, we achieve a much more complete identification of the left cow's snout, complete identification of the left cow's eyes and ear tag, and a slightly more complete identification of the right cow's hind.   
\begin{figure}[ht]
    \centering
    \begin{subfigure}[t]{0.32\textwidth}
    \centering
    \includegraphics[width=\textwidth]{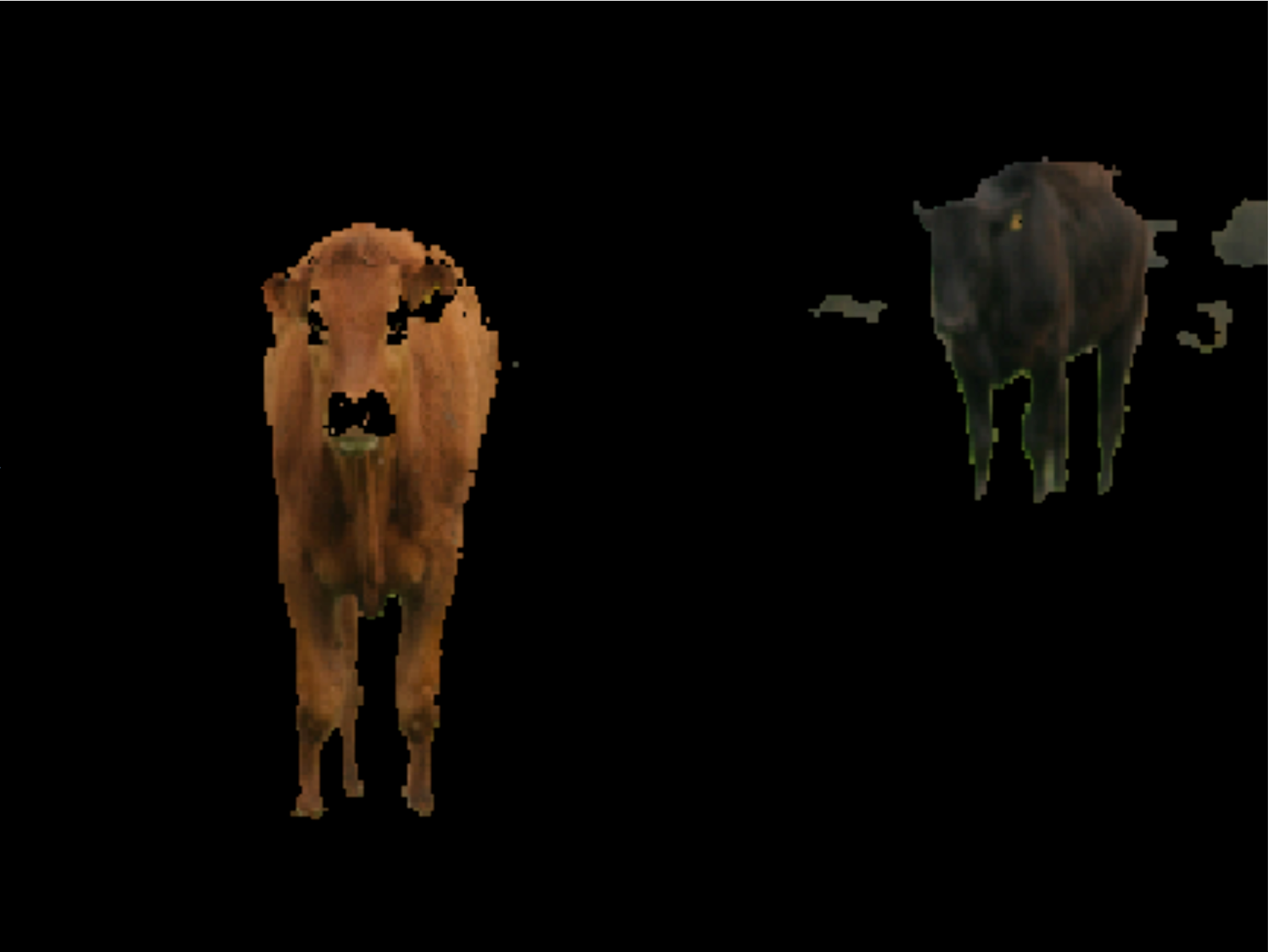}
    \caption{Segmentation from \cite[Fig. 4.6]{BF}}
    \end{subfigure}
    \begin{subfigure}[t]{0.32\textwidth}
    \centering
    \includegraphics[width=\textwidth]{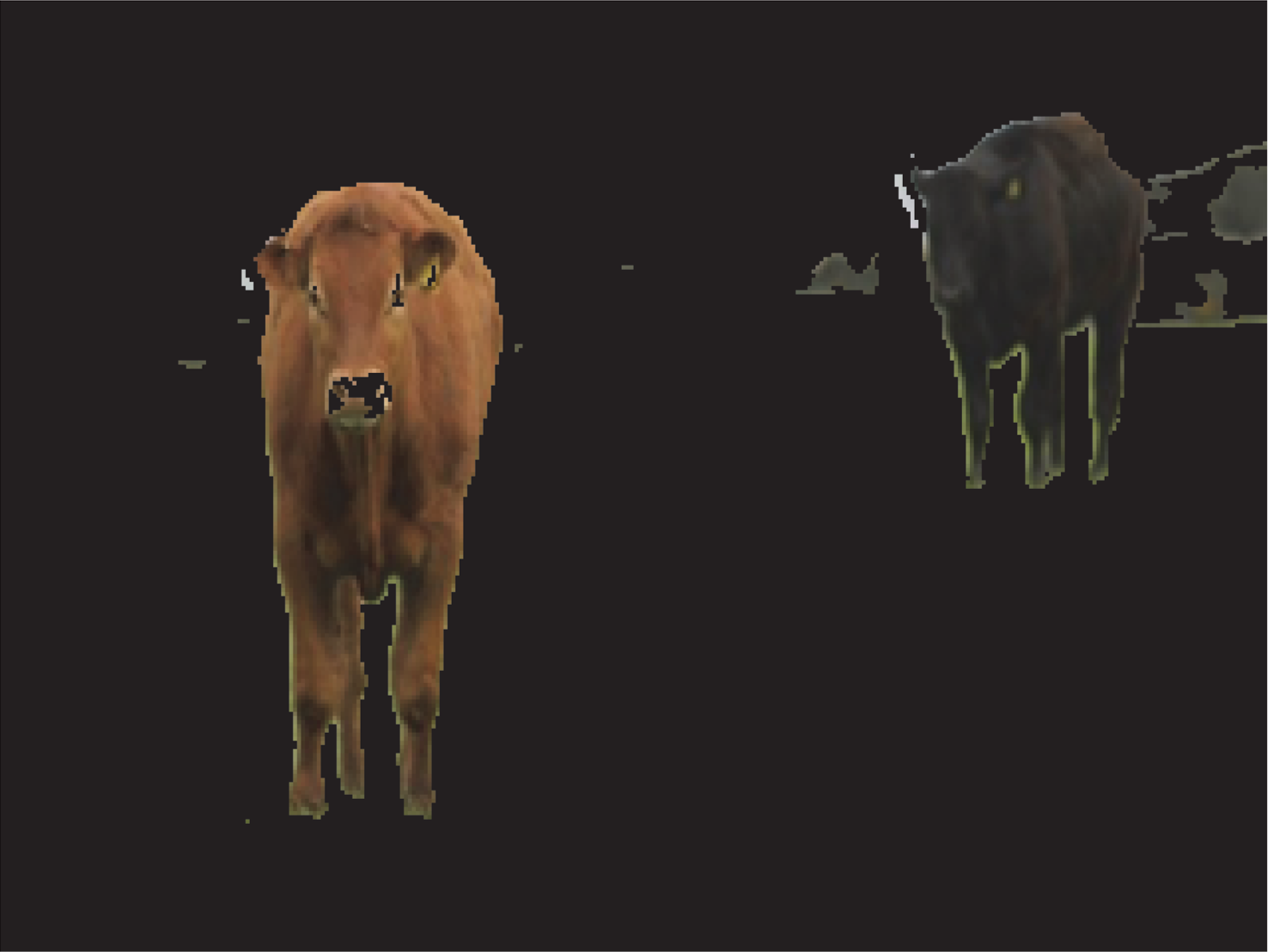}
    \caption{Segmentation from \cite[Fig. 2(f)]{MKB}}
    \end{subfigure}
    \begin{subfigure}[t]{0.32\textwidth}
    \centering
    \includegraphics[width=\textwidth]{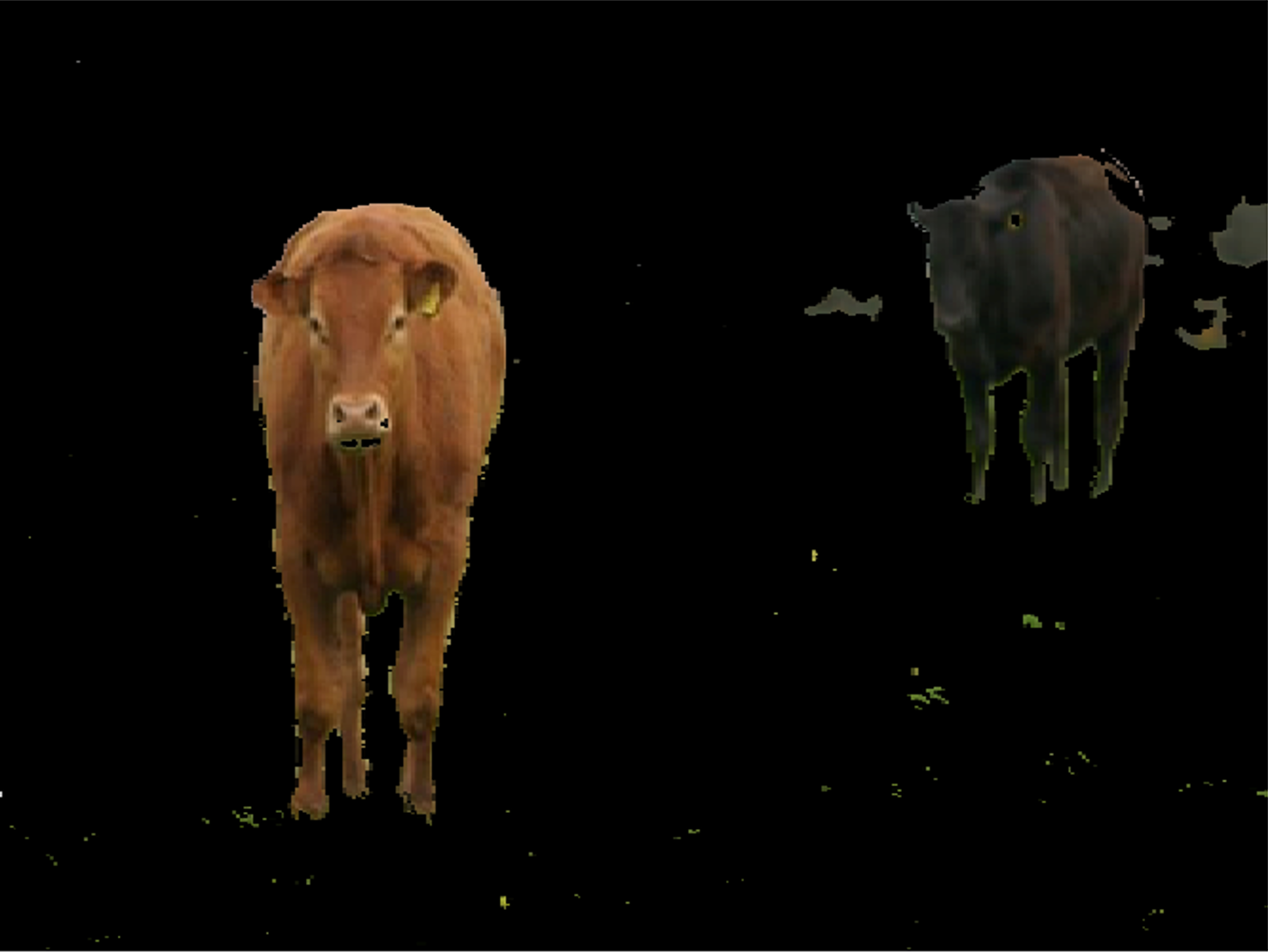}
    \caption{Segmentation from Fig.~\ref{Fig_Results_MBO+SDIE}(d)}
    \end{subfigure}
    \caption{Comparison of our segmentation (using the set-up in Fig.~\ref{Fig_Results_MBO+SDIE}(d)) with the analogous segmentations from the previous literature \cite{MKB,BF}, both reproduced with permission from SIAM and the authors. Note that unfortunately in reproduction the colour balances and aspect ratios have become slightly inconsistent, but we can still make qualitative comparisons.  }
    \label{fig:usvsBFMKB}
\end{figure}
\begin{figure}[hpt]
\centering
\begin{subfigure}[t]{0.25\textwidth}
\centering
         \includegraphics[width=\textwidth,height=4.8\textwidth]{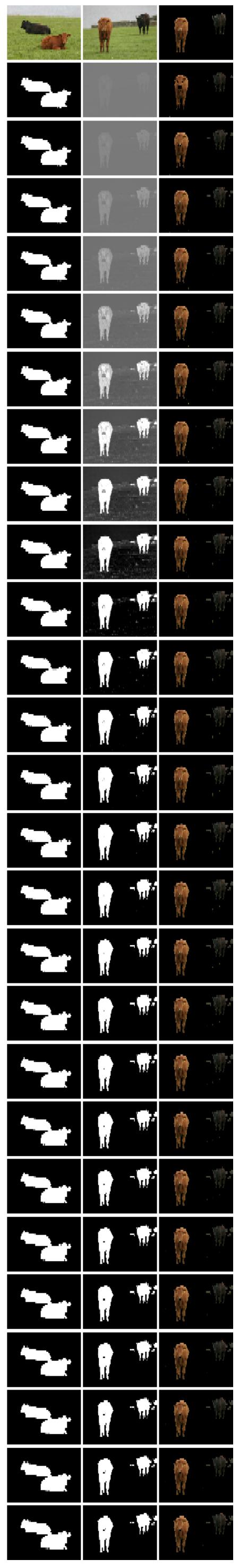}
\caption{$\tau = 0.001 \ll 0.003 = \varepsilon$.}
\end{subfigure}
\hspace{0.1\textwidth}
 \begin{subfigure}[t]{0.25\textwidth}
\centering
         \includegraphics[width=\textwidth]{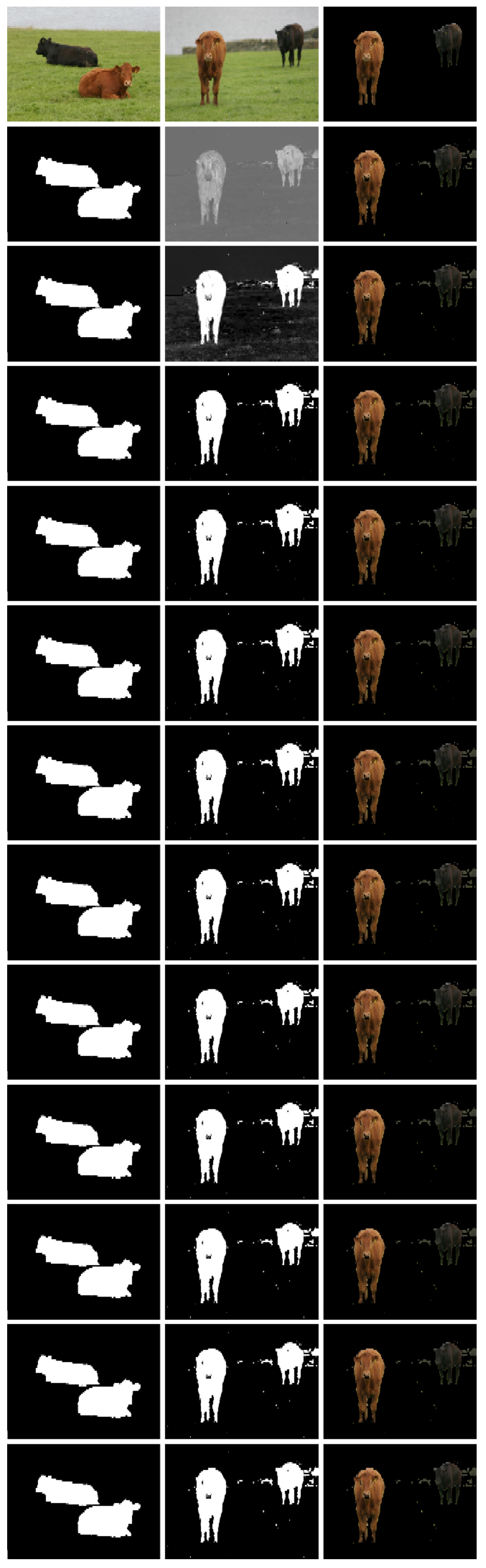}
\caption{$\tau = 0.0025 < 0.003 =\varepsilon$}
\end{subfigure}
\hspace{0.1\textwidth}
 \begin{subfigure}[t]{0.25\textwidth}
\centering
         \includegraphics[width=\textwidth]{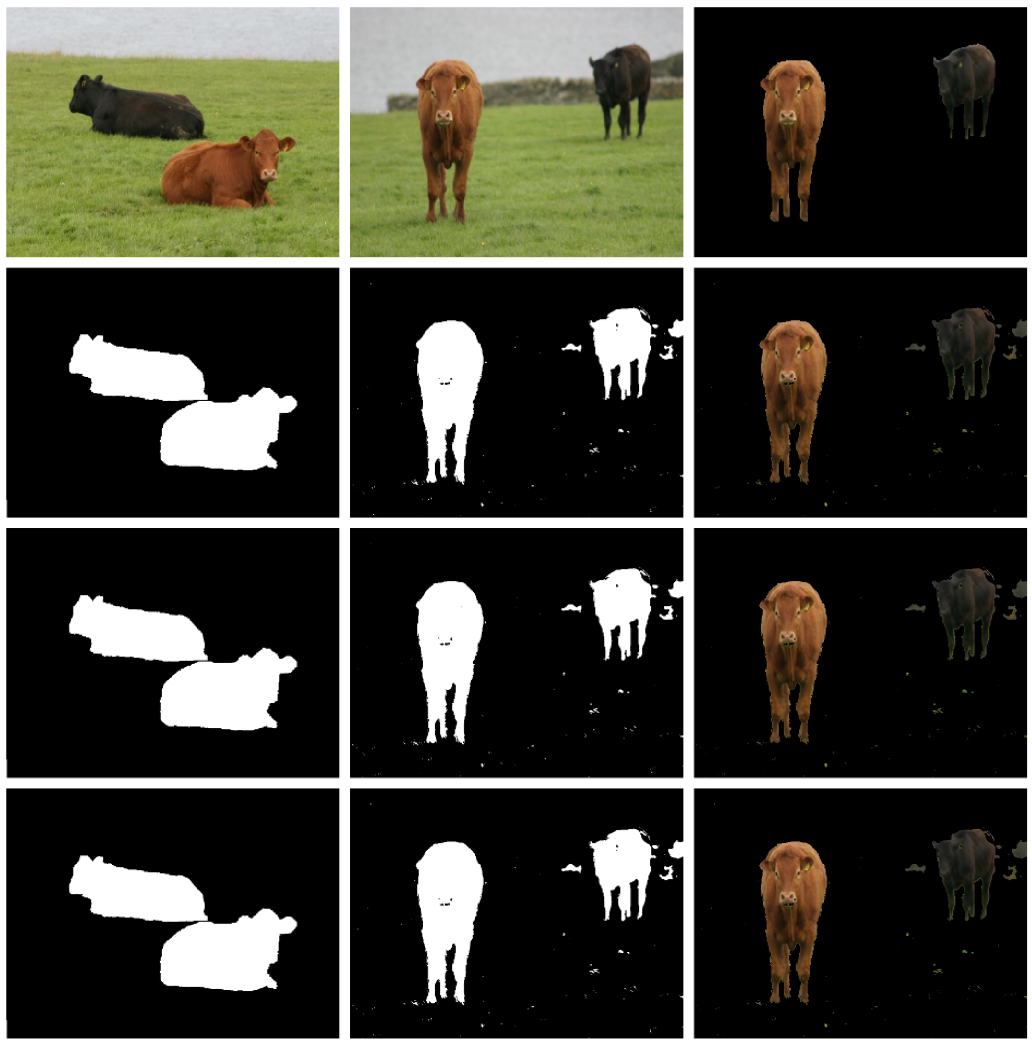}
\caption{$\tau = \varepsilon = 0.003$ (as in \cite{MKB})}
\end{subfigure} 
\caption{Progression of MBO and non-MBO SDIE for the \emph{two cows} example. In each subfigure: The first row shows the reference data, image, and ground truth, as in Fig.~\ref{Fig_Results_MBO+SDIE}. The middle rows, showing the reshaped label vector $u_n$ and the image masked by the thresholded label, each represent one iteration of the considered algorithm, to be read from top to bottom. The last row gives the state returned by the scheme, i.e. the state satisfying the termination criterion, which correspond to the subfigures in Fig.~\ref{Fig_Results_MBO+SDIE}. For layout reasons, we have squashed the figure in (a).}
\label{Fig_ProgressionMBO} 
\end{figure}

We measure the computational cost of the SDIE scheme through the measured runtime of the respective algorithm. We note from Fig.~\ref{Fig_Results_MBO+SDIE} that the MBO scheme ($\tau = \varepsilon$) outperforms the non-MBO schemes ($\tau < \varepsilon$); the SDIE relaxation of the MBO scheme merely slows down the convergence of the algorithm, without improving the segmentation. This can especially be seen in Fig.~\ref{Fig_ProgressionMBO}, where the SDIE scheme needs many more steps to satisfy the termination criterion.
At least for this example, the non-MBO SDIE scheme is less efficient than the MBO scheme. Thus, in the following sections we focus on the MBO case.
\subsubsection{Errors and timings} \label{subsubse_error_timing_2cows}
We now quantify the influence, on the accuracy and computational cost of the segmentation, of the Nystr\"om rank $K$, the number of discretisation steps $k_b$ and $k$ in the Euler method and the Strang formula respectively, and the choice of normalisation of the graph Laplacian. 
To this end, we segment the \emph{two cows} image using the following parameters: $\varepsilon = \tau = 0.003$,  $\hat\mu = 30$, $\sigma = 35$, and $\delta = 10^{-10}$. We take $K \in \{10, 25, 70, 100, 250\}$, $(k_b, k) \in \{(1,1), (10,5)\}$, and use the random walk Laplacian $\Delta$ and the symmetric normalised Laplacian $\Delta_s$. 

We plot runtimes and relative segmentation errors in Fig.~\ref{Fig_Error_timing}. As our method has randomness from the Nystr\"om extension, we repeat every experiment 100 times and show means and standard deviations. We make several observations. Starting with the runtimes, we indeed see that these are roughly linear in $K$, verifying numerically the expected complexity. The runtime also increases when increasing $k_b$ and $k$. That is, increasing the accuracy of the Euler method and Strang formula does not lead to faster convergence, and also does not increase the accuracy of the overall segmentation. Finally, we see that the symmetric normalised Laplacian incurs consistently low relative segmentation error for small values of $K$. This property is of the utmost importance to scale up our algorithm for very large images. Interestingly, the segmentations using the symmetric normalised Laplacian seem to deteriorate for a large $K$. The random walk Laplacian has diametric properties in this regard: the segmentations are only reliably accurate when $K$ is reasonably large.

\begin{figure}[hpt]
\centering
\begin{subfigure}{0.49\textwidth}
\centering
         \includegraphics[width=\textwidth]{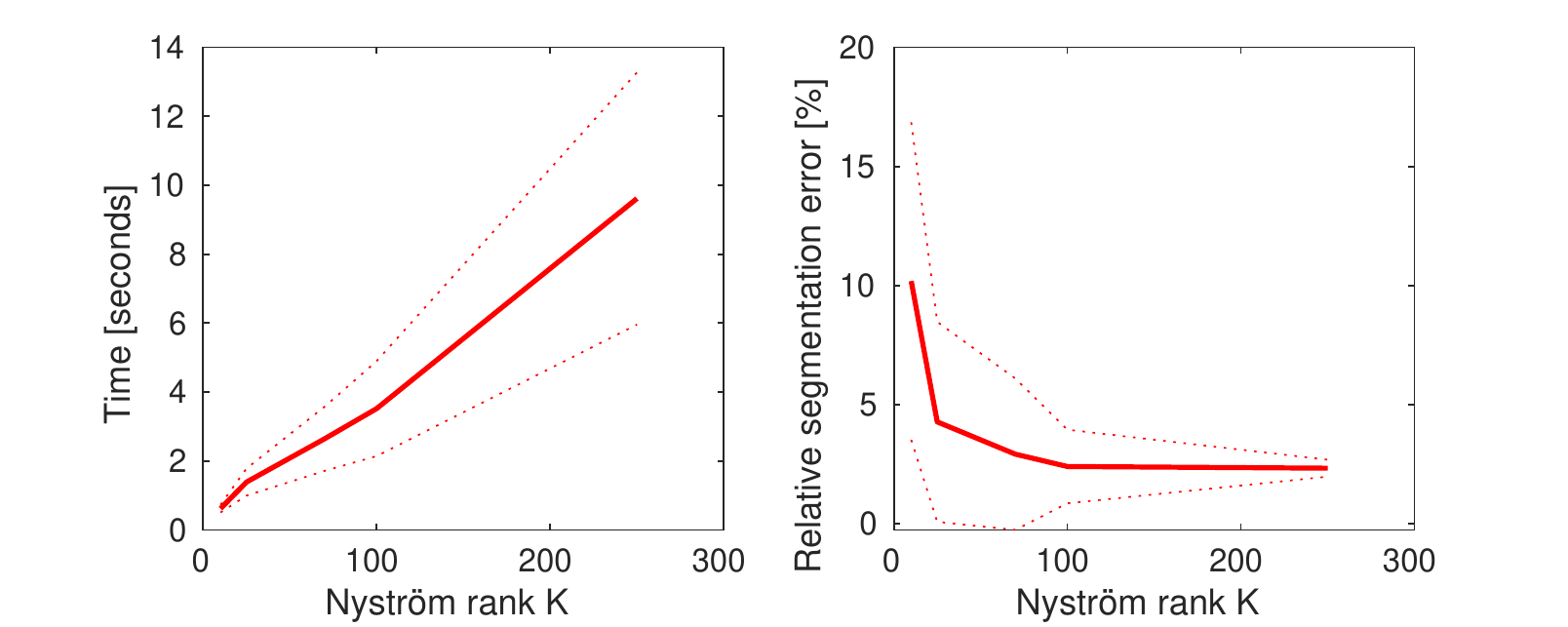}
\caption{$k_b=1$, $k = 1$, random walk Laplacian.}
\end{subfigure}
\begin{subfigure}{0.49\textwidth}
\centering
         \includegraphics[width=\textwidth]{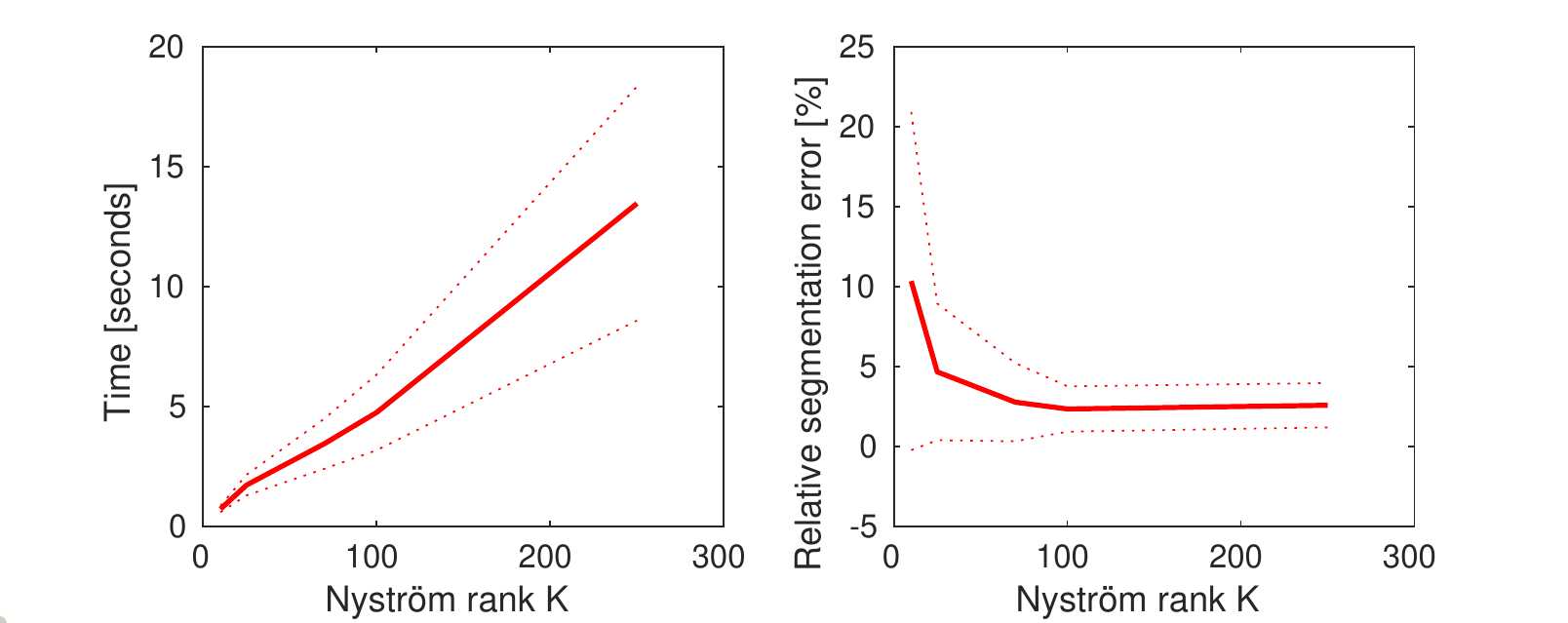}
\caption{$k_b=10$, $k = 5$, random walk Laplacian.}
\end{subfigure}\\
\begin{subfigure}{0.49\textwidth}
\centering
         \includegraphics[width=\textwidth]{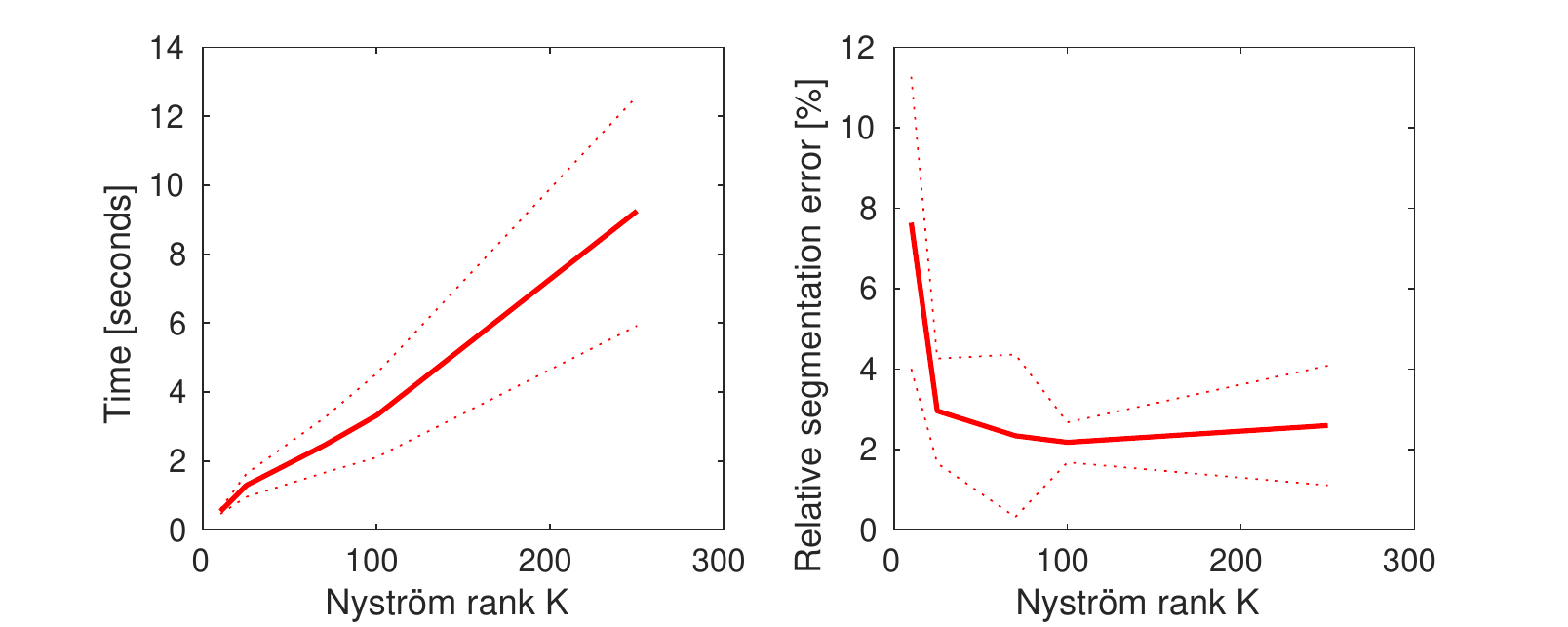}
\caption{$k_b=1$, $k = 1$, symmetric normalised Laplacian.}
\end{subfigure}
\begin{subfigure}{0.49\textwidth}
\centering
         \includegraphics[width=\textwidth]{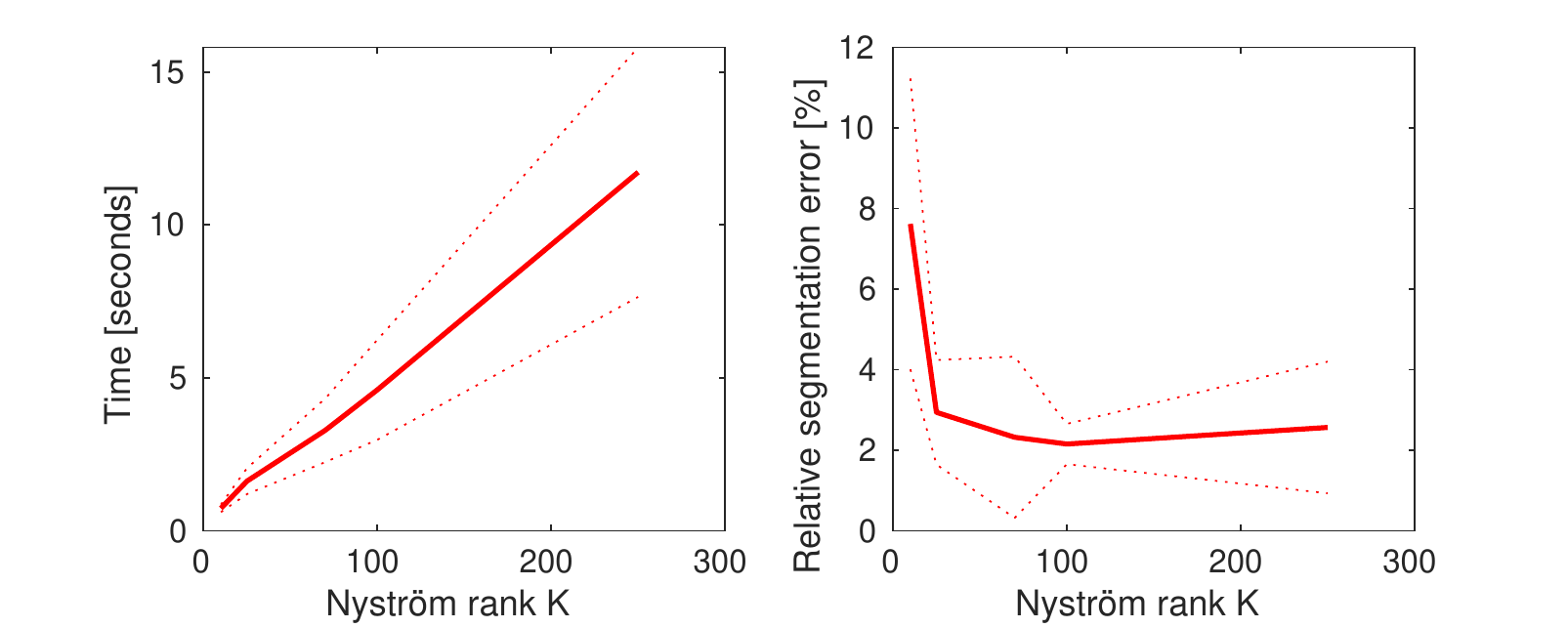}
\caption{$k_b=10$, $k=5$, symmetric normalised Laplacian.}
\end{subfigure}
\caption{Error and timing of $100$ independent segmentations of the two cows image (Example~\ref{ex_twocows}) with the MBO SDIE scheme. The solid lines represent the means averaged over 100 runs, the dotted lines show means $\pm$ standard deviations.}
\label{Fig_Error_timing} 
\end{figure}

\subsubsection{Uncertainty in the segmentation} \label{subsubse_uncertain}
Due to the randomised Nystr\"om approximation, our approximation of the SDIE scheme is inherently stochastic. 
Therefore, the segmentations that the algorithm returns are realisations of random variables. 
We now briefly study these random variables, especially with regard to $K$. We show pointwise mean and standard deviations of the binary labels in each of the left two columns of the four subfigures of Fig.~\ref{Fig_MonteCarlo}. In the remaining figures, we weight the original \emph{two cows} image with these means (varying continuously between label 1 for `cow' and label 0 for `not cow') and standard deviations.  For these experiments we use the same parameter set-up as Section~\ref{subsubse_error_timing_2cows}.

\begin{figure}[hpt]
\centering
\begin{subfigure}{0.49\textwidth}
\centering
         \includegraphics[width=\textwidth]{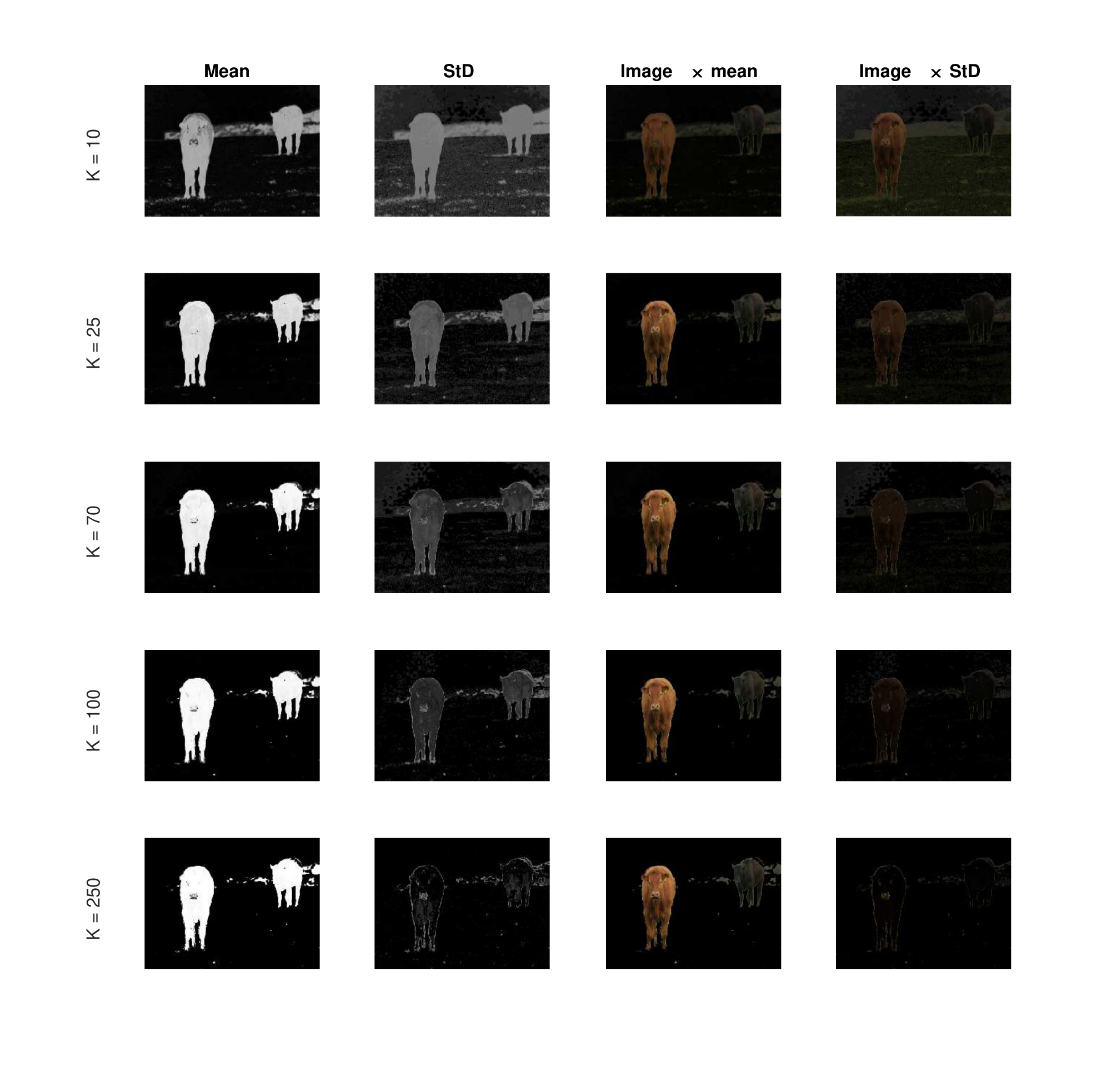}\vspace{-0.8cm}
\caption{$k_b=1$, $k = 1$, random walk Laplacian. }
\end{subfigure}
\begin{subfigure}{0.49\textwidth}
\centering
         \includegraphics[width=\textwidth]{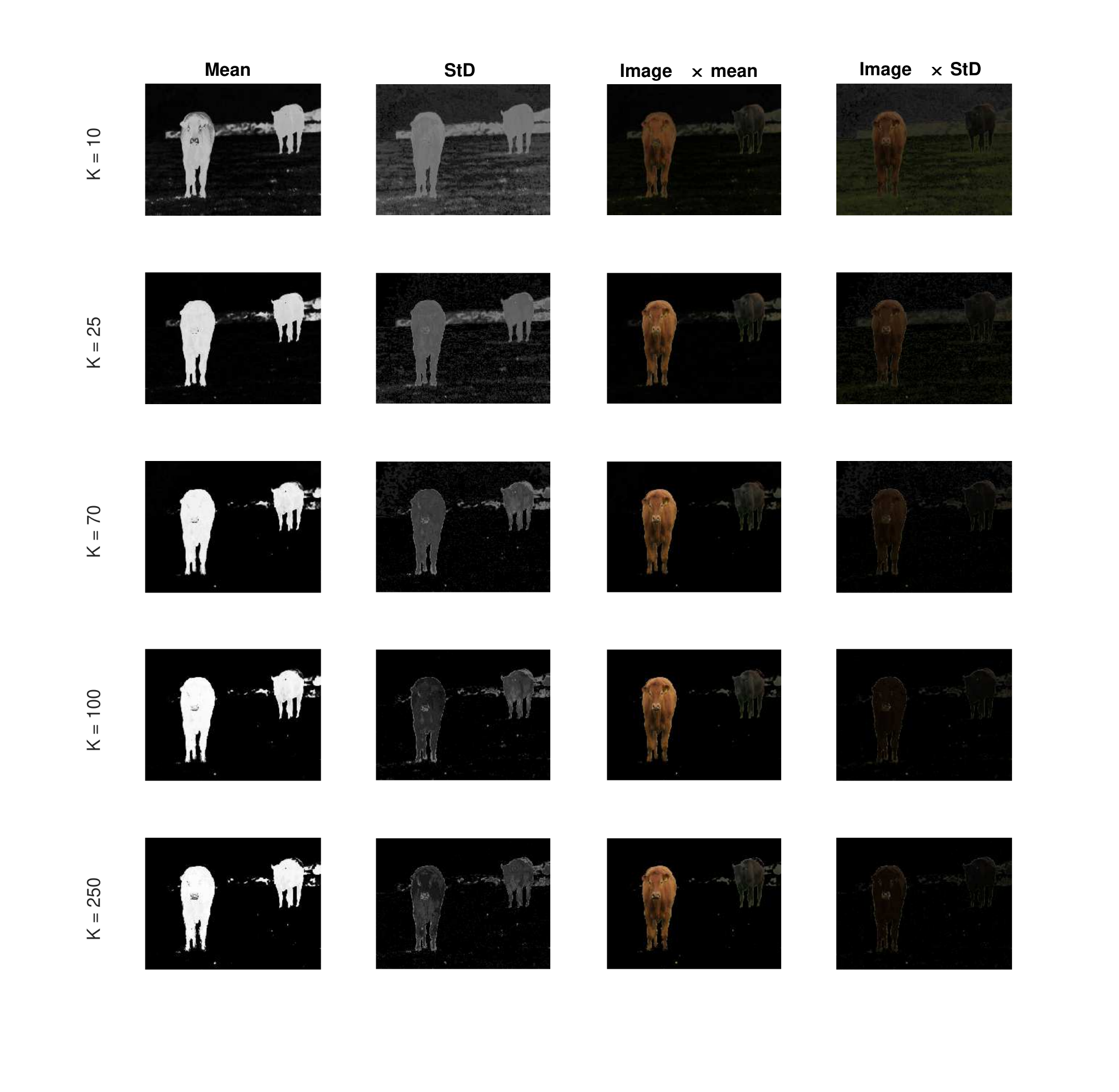}\vspace{-0.8cm}
\caption{$k_b=10$, $k = 5$, random walk Laplacian.
}
\end{subfigure}\\
\begin{subfigure}{0.49\textwidth}
\centering
         \includegraphics[width=\textwidth]{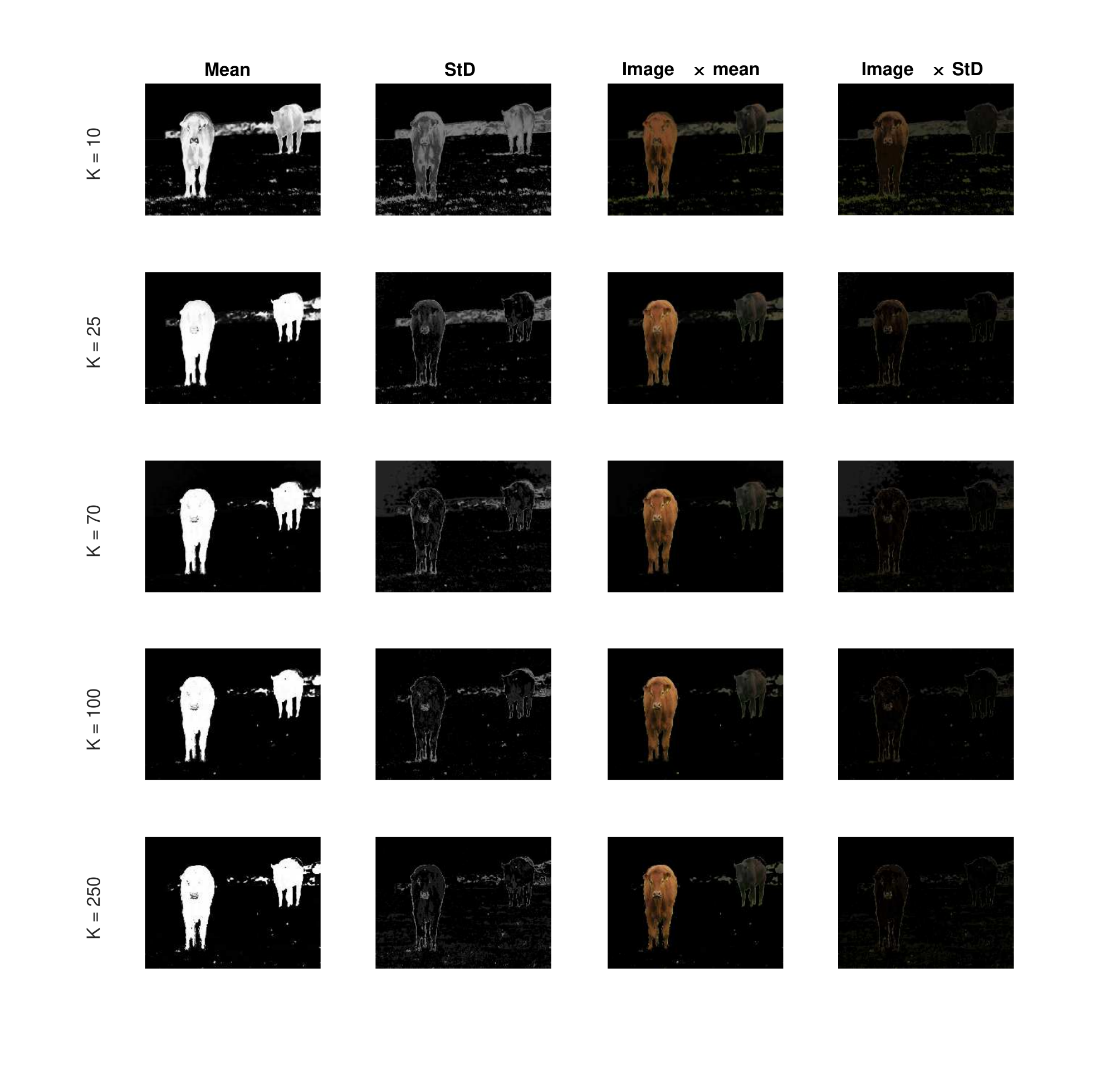}\vspace{-0.8cm}
\caption{$k_b=1$, $k = 1$, symmetric normalised Laplacian.}
\end{subfigure}
\begin{subfigure}{0.49\textwidth}
\centering
         \includegraphics[width=\textwidth]{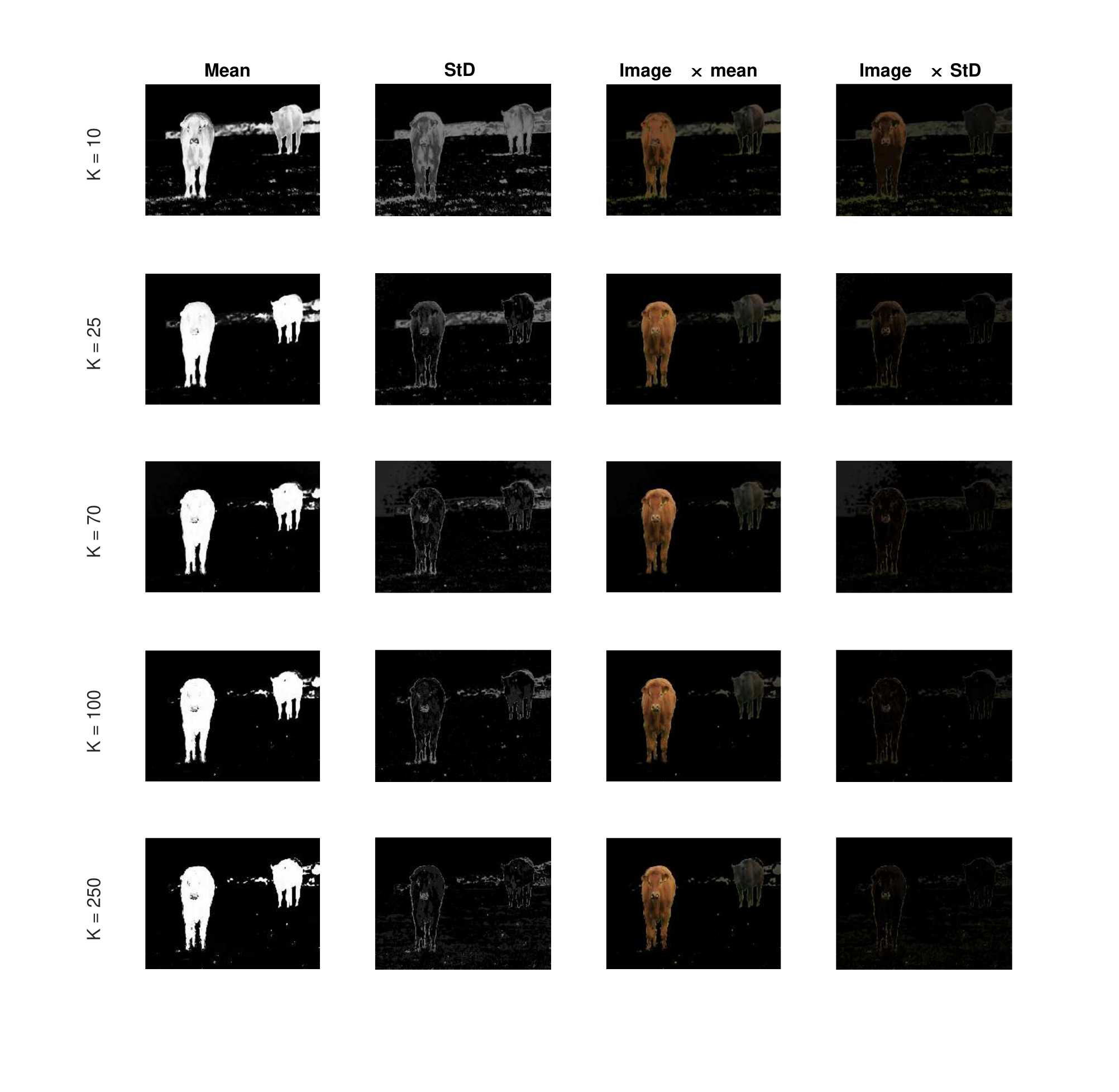}\vspace{-0.8cm}
\caption{$k_b=10$, $k=5$, symmetric normalised Laplacian.}
\end{subfigure}
\caption{Mean, standard deviation, and images weighted by mean and standard deviation of $100$ independent segmentations of Example~\ref{ex_twocows} with the MBO SDIE scheme, with set-up as in Section~\ref{subsubse_error_timing_2cows}.}
\label{Fig_MonteCarlo} 
\end{figure}

We make several observations. First, as $K$ increases we see less variation. This is what we expect, as when $K=|V|$ the method is deterministic so has no variation. Second, the type of normalisation of the Laplacian has an effect: the symmetric normalised Laplacian leads to less variation than the random walk Laplacian. Third, the parameters $k_b$ and $k$  appear to have no major effect within the range tested. Finally, looking at the figures with rather large $K$, we observe that the standard deviation of the labels is high in the areas of the figure in which there is indeed uncertainty in the segmentation, namely the boundaries of the cows and the parts of the wall with similar colour to the dark cow. Determining the exact position of the boundary of a cow on a pixel-by-pixel level is indeed also difficult for a human observer. Moreover, the SDIE scheme usually confuses the wall in the background for a cow. Hence, a large standard deviation here reflects that the estimate is uncertain. 

This is of course not a rigorous Bayesian uncertainty quantification, as for instance is given in \cite{GraphUQ,BayesianGraphs} for graph-based learning. However the use of stochastic algorithms for inference tasks, and the use of their output as a method of uncertainty quantification, has for instance been motivated by \cite{SGDBayes}.

\subsection{Greyscale}
We now move on to Example~\ref{ex_twocows_gs}, the \emph{greyscale} problem. We will especially use this example to study the influence of the parameters $\hat\mu$ and $\sigma$. The parameter $\hat\mu > 0$ determines the strength of the fidelity term in the ACE. From a statistical point of view, we can understand a choice of $\hat\mu$ as an assumption on the statistical precision (i.e. the inverse of the variance of the noise) of the reference~$\tilde f$ (see \cite[Section 3.3]{GraphUQ}
for details). Thus, a small $\mu$ should lead to a stronger regularisation coming from the Ginzburg--Landau functional, and a large $\mu$ leads to more adherence to the reference. The parameter $\sigma > 0$ is the `standard deviation' in the Gaussian kernel $\Omega$ used to build the weight matrix $\omega$. For our methods we must not choose too small a $\sigma$, as otherwise the weight matrix becomes sparse (up to some precision), and so the Nystr\"om submatrix $\omega_{XX}$ has a high probability of being ill-conditioned.\footnote{However, in such a case this sparsity can be exploited using Raleigh–Chebyshev \cite{RC} methods as in \cite{MKB}, or \textsc{Matlab} sparse matrix algorithms as in \cite{BF}. This lies beyond the scope of this paper.} If $\sigma$ is too large then the graph structure no longer reflects the features of the image.

\begin{figure}[hpt]
\centering
\begin{subfigure}{0.49\textwidth}
\centering
         \includegraphics[width=\textwidth]{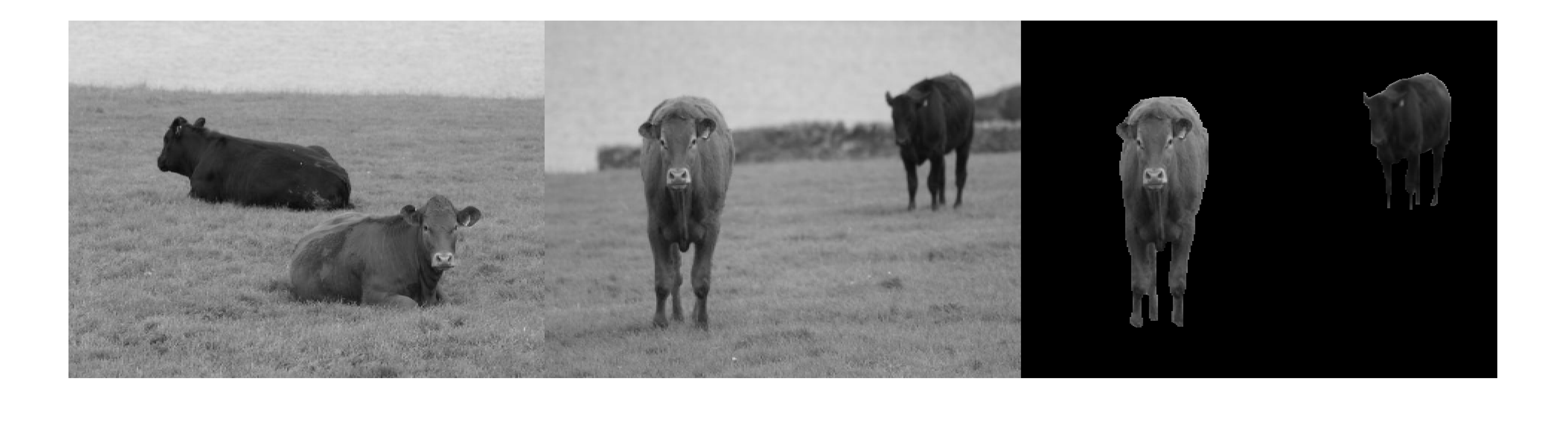}
\end{subfigure} \\
\begin{subfigure}{0.49\textwidth}
\centering
        \includegraphics[width=\textwidth]{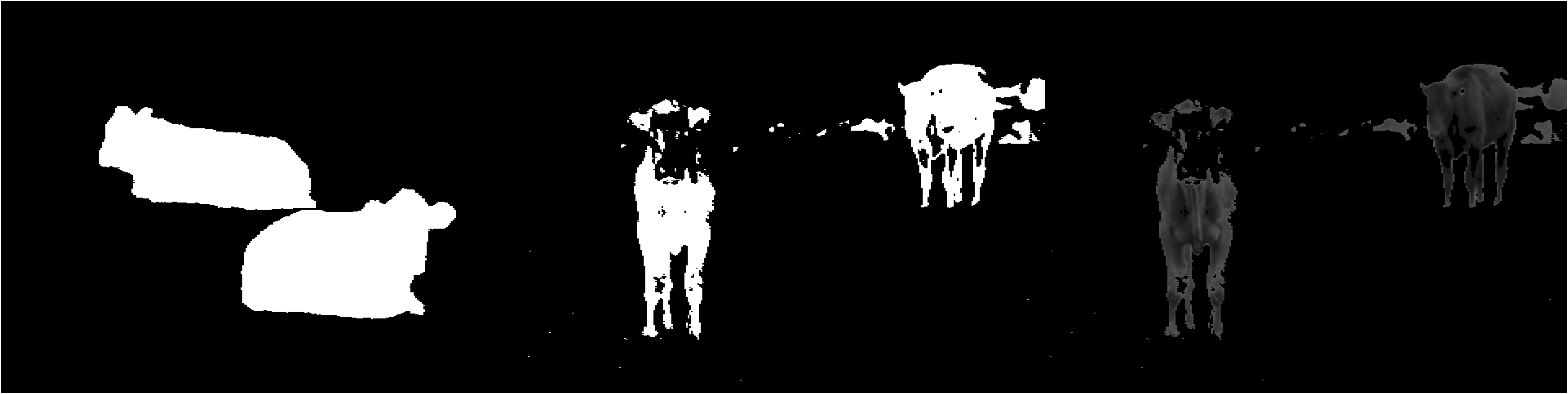} 
 \caption{$\hat\mu = 100, \sigma = 50$, error = 5.7868 \%, time = 6.9 sec.}
\end{subfigure}
 \begin{subfigure}{0.49\textwidth}
\centering
         \includegraphics[width=\textwidth]{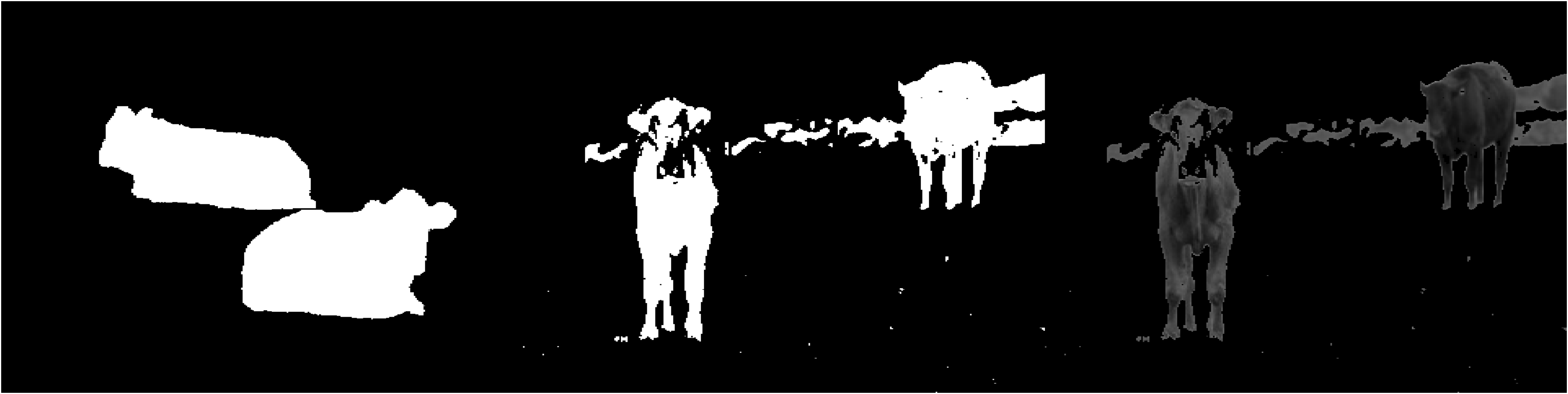} 
\caption{$\hat\mu = 150, \sigma = 35$, error = 5.8245 \%, time = 7.0 sec.}
\end{subfigure}
\caption{MBO SDIE segmentations for the \emph{greyscale} segmentation task. In the centred top figure, we show the reference data image, the image to be segmented, and the image masked with the ground truth segmentation, cf. Fig.~\ref{Fig_twocows}. The other figures show the reference $\tilde f$, the segmentation returned by the algorithm, and the original images masked with the segmentation.}
\label{Fig_Results_greysc} 
\end{figure}

In the following, we set $\varepsilon = \tau = 0.00024$, $k_b = 10$, $k = 5$, and $\delta = 10^{-10}$. To get reliable results we choose a rather large $K$, $K=200$, and therefore (by the discussion in Section \ref{subsubse_uncertain}) we use the random walk Laplacian. We will qualitatively study single realisations of the inherently stochastic SDIE algorithm. We vary $\hat\mu \in \{50, 100, 150, 200\}$ and $\sigma \in \{20, 35, 50\}$. We show the best results from these tests in  Fig.~\ref{Fig_Results_greysc}. Moreover, we give a comprehensive overview of all tests and the progression of the SDIE scheme in Fig.~\ref{Fig_Progression_GS}.
\begin{figure}[htp]
    \centering
    \begin{subfigure}{0.24\textwidth}
    \includegraphics[width=\textwidth]{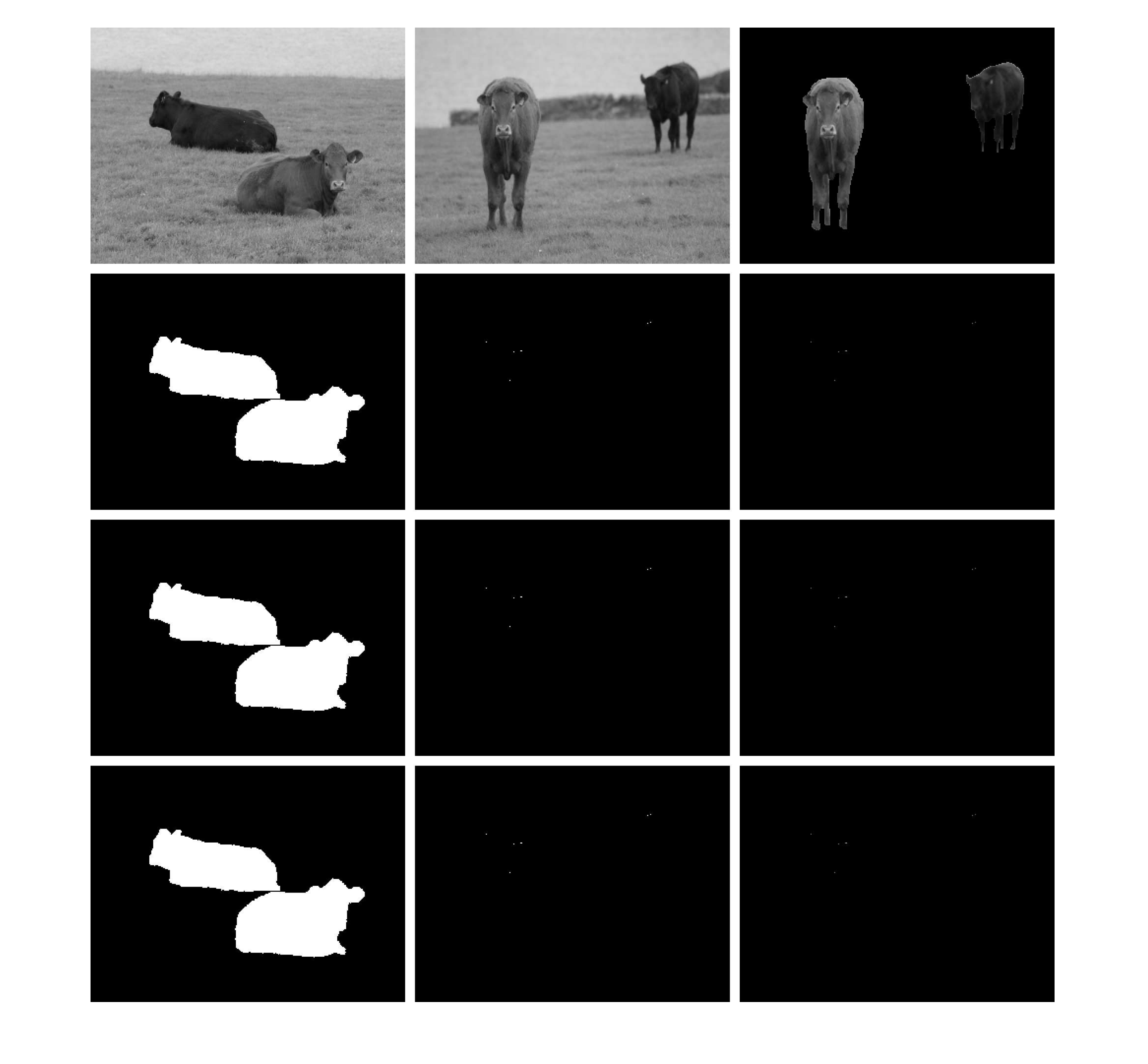}\vspace{-0.25cm}
    \caption{$\hat\mu = 50, \sigma = 20$, error = 12.7887 \%, time = 7.0 sec.}\vspace{0.25cm}
    \end{subfigure}
       \begin{subfigure}{0.24\textwidth}
    \includegraphics[width=\textwidth]{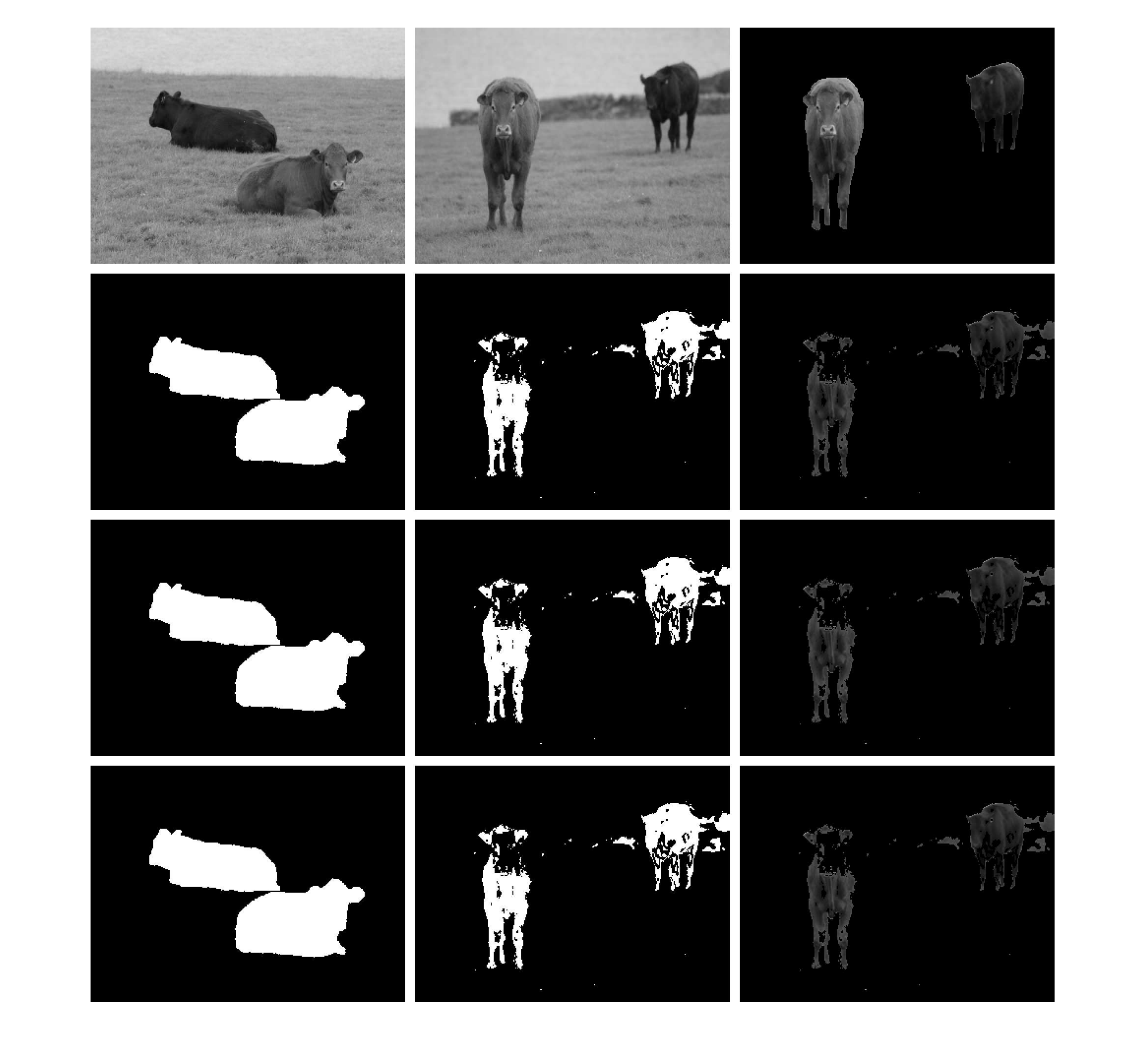}\vspace{-0.25cm}
    \caption{$\hat\mu = 100, \sigma = 20$, error = 6.0492 \%, time = 7.0 sec.}\vspace{0.25cm}
    \end{subfigure}
       \begin{subfigure}{0.24\textwidth}
    \includegraphics[width=\textwidth]{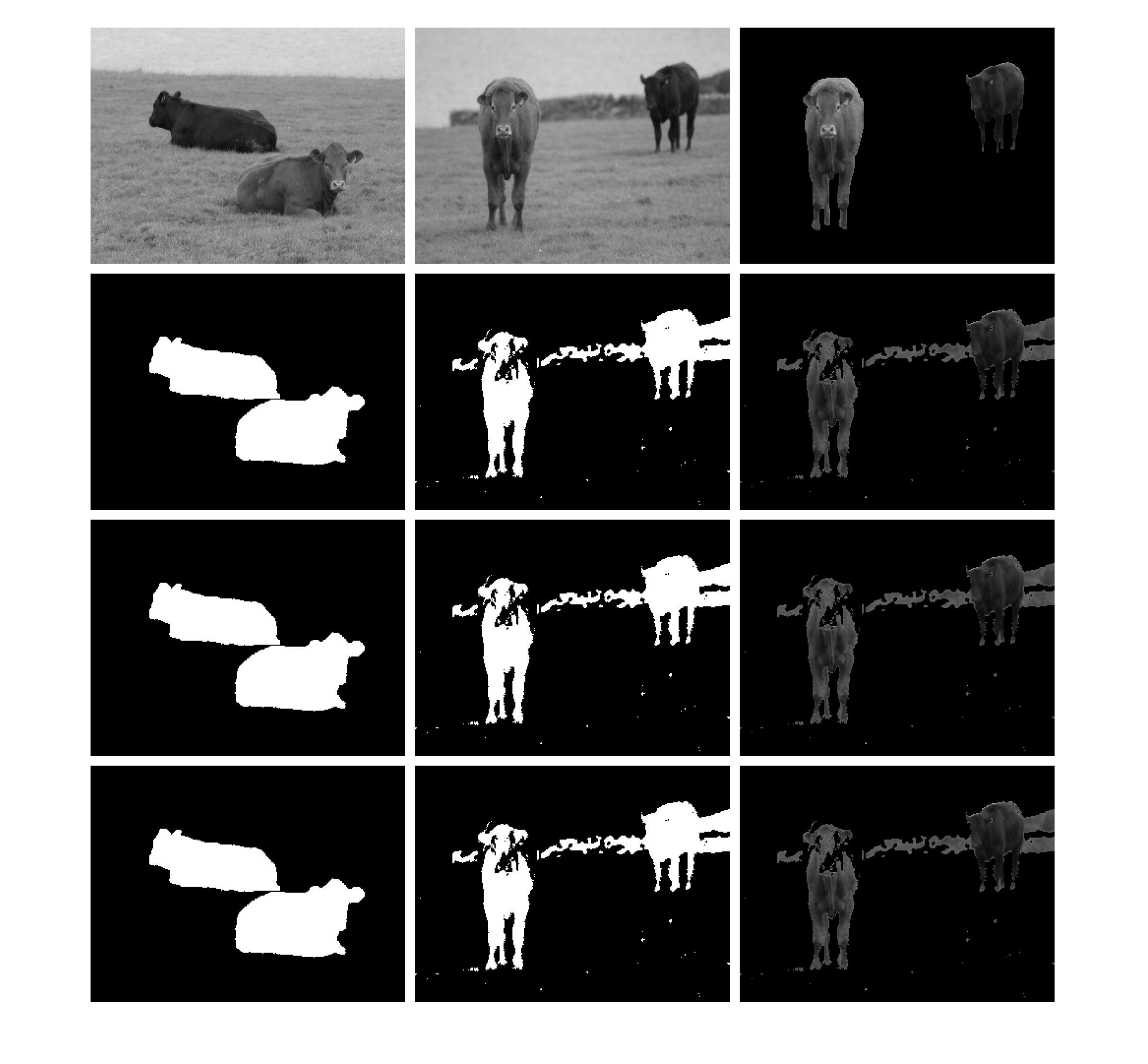}\vspace{-0.25cm}
    \caption{$\hat\mu = 150, \sigma = 20$, error = 5.9326 \%, time = 7.1 sec.}\vspace{0.25cm}
    \end{subfigure}
       \begin{subfigure}{0.24\textwidth}
    \includegraphics[width=\textwidth]{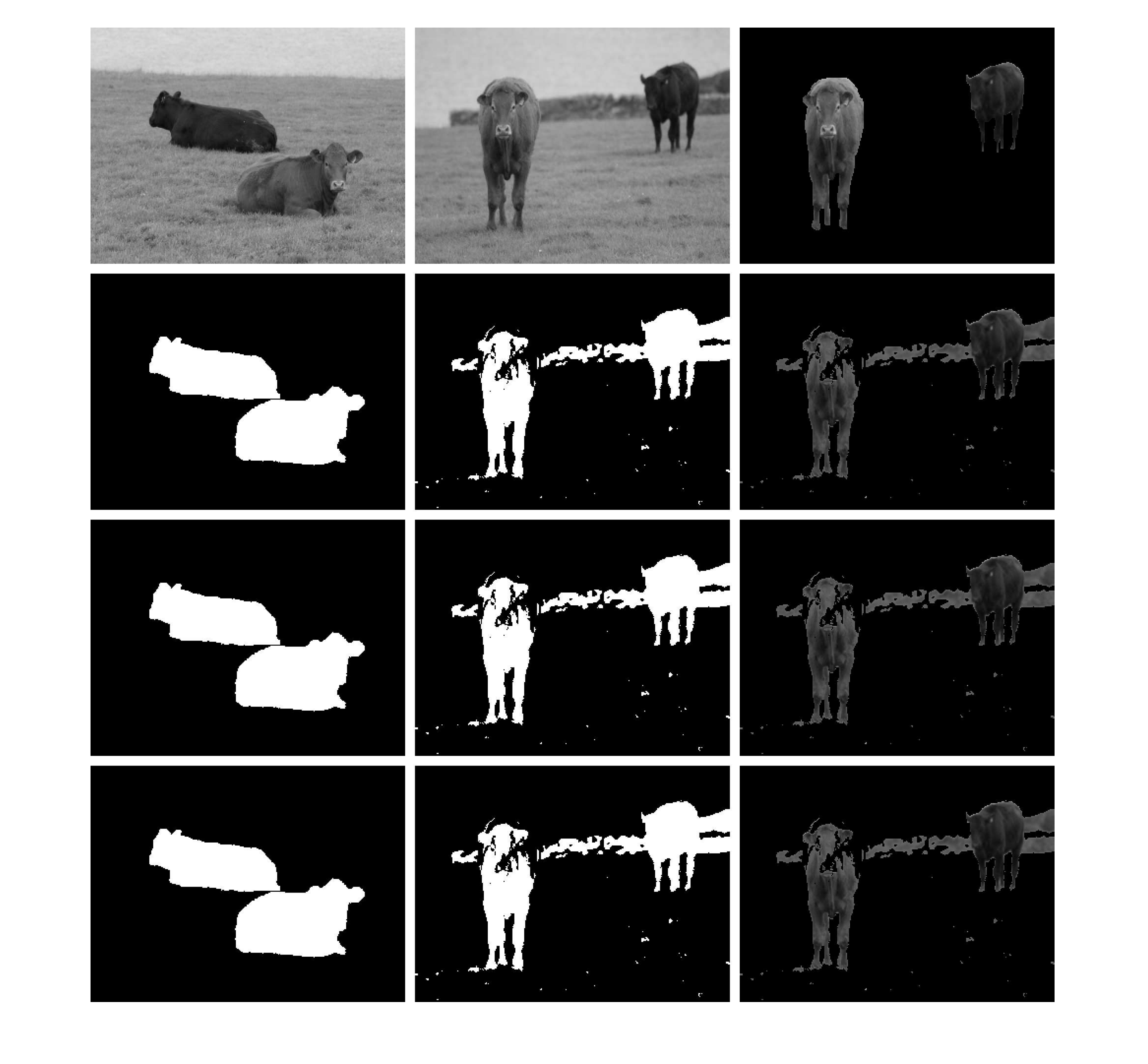}\vspace{-0.25cm}
    \caption{$\hat\mu = 200, \sigma = 20$, error = 6.293 \%, time = 13.3 sec.}\vspace{0.25cm}
    \end{subfigure} \\
        \begin{subfigure}{0.24\textwidth}
    \includegraphics[width=\textwidth]{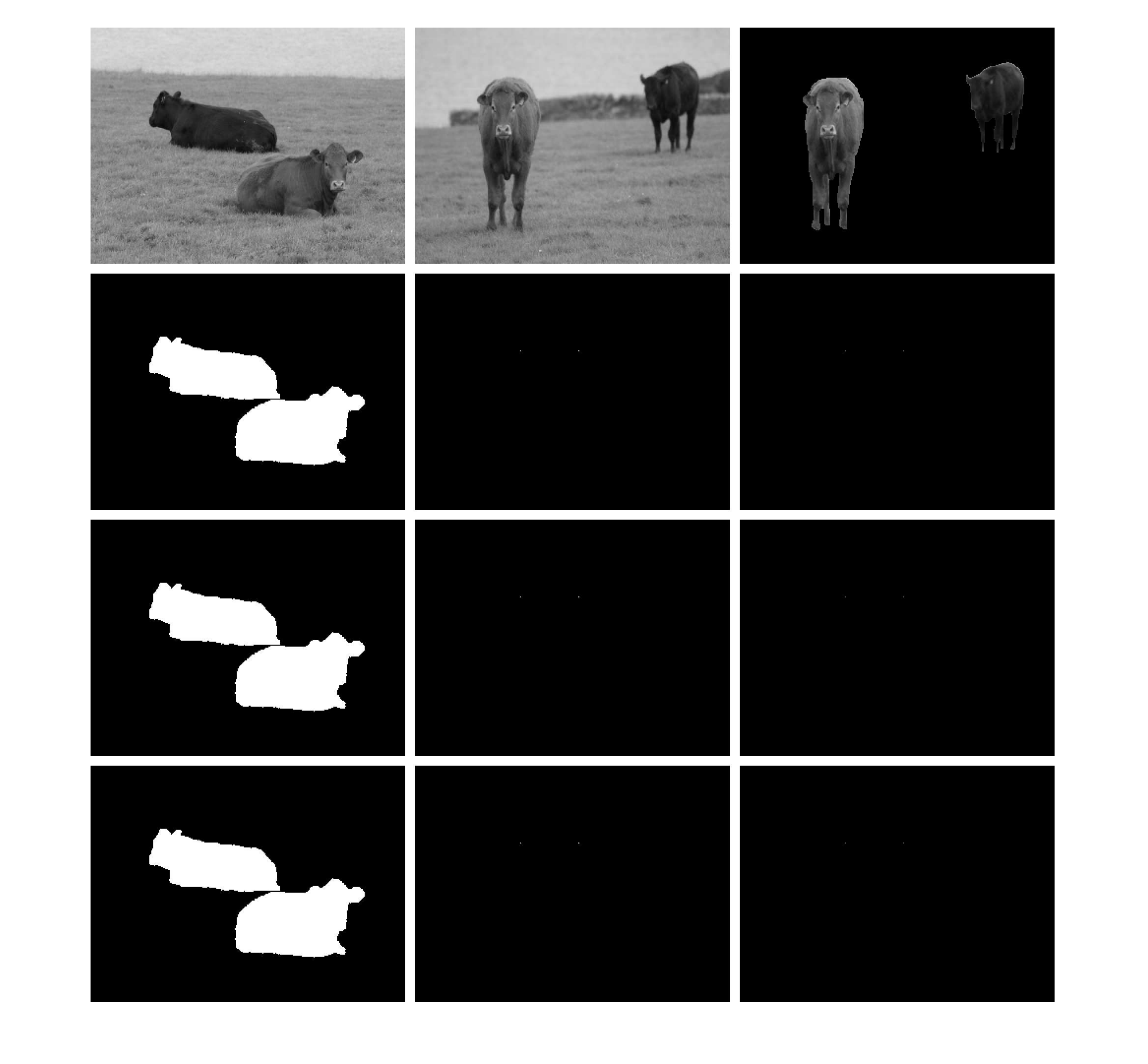}\vspace{-0.25cm}
    \caption{$\hat\mu = 50, \sigma = 35$, error = 12.7969 \%, time = 7.1 sec.}\vspace{0.25cm}
    \end{subfigure}
       \begin{subfigure}{0.24\textwidth}
    \includegraphics[width=\textwidth]{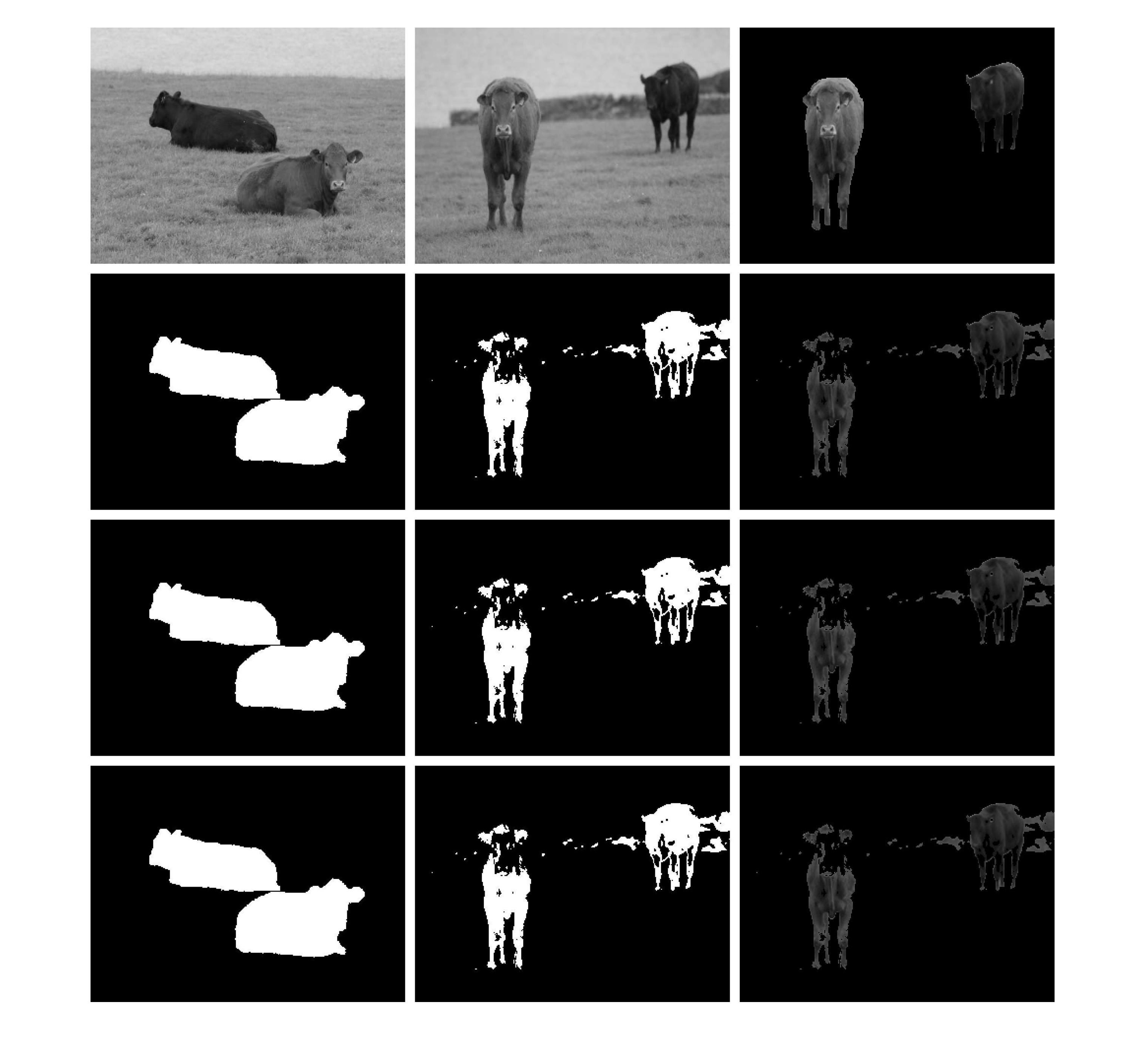}\vspace{-0.25cm}
    \caption{$\hat\mu = 100, \sigma = 35$, error = 5.9043 \%, time = 7.2 sec.}\vspace{0.25cm}
    \end{subfigure}
       \begin{subfigure}{0.24\textwidth}
    \includegraphics[width=\textwidth]{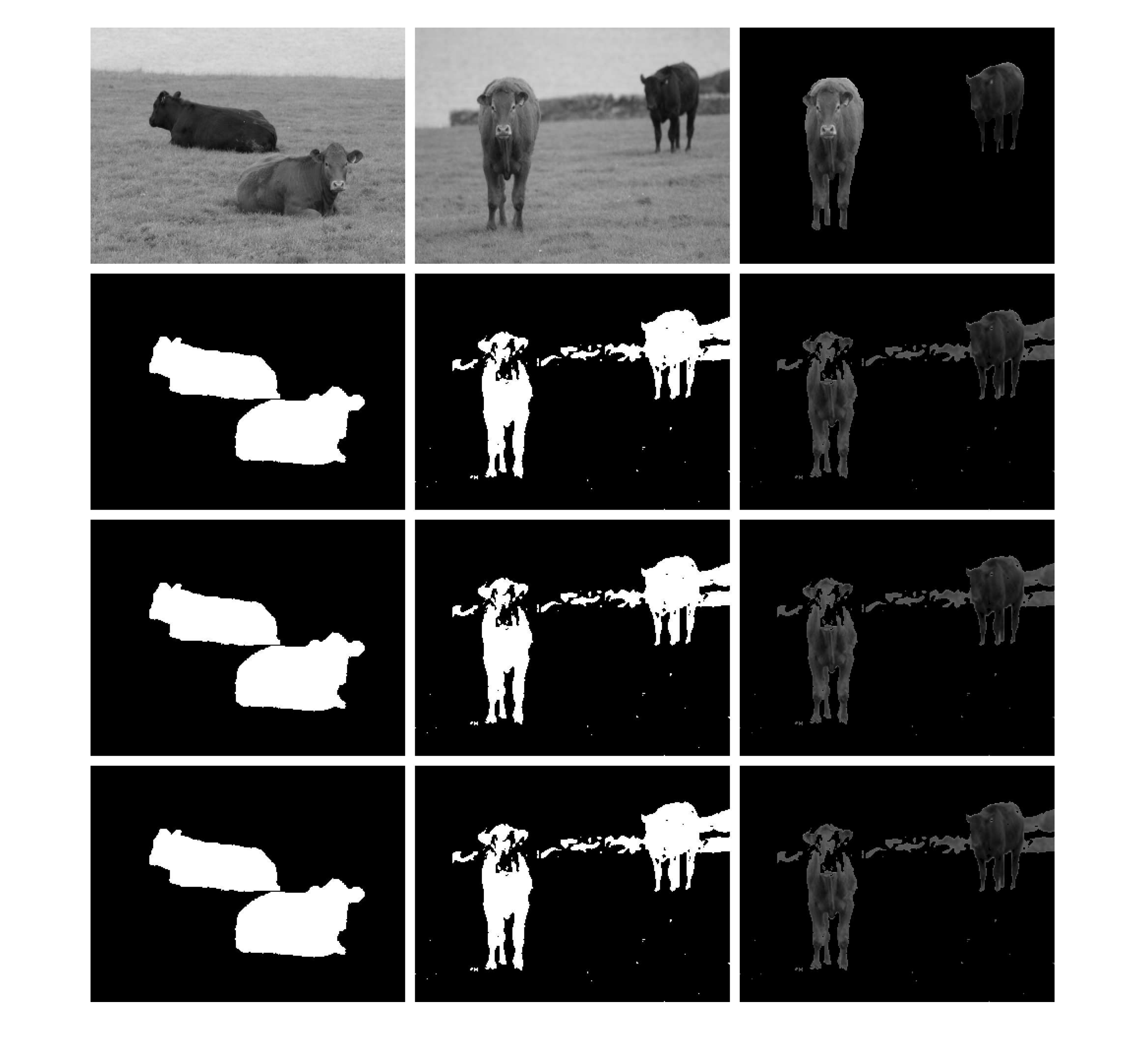}\vspace{-0.25cm}
    \caption{$\hat\mu = 150, \sigma = 35$, error = 5.8245 \%, time = 7.0 sec.}\vspace{0.25cm}
    \end{subfigure}
       \begin{subfigure}{0.24\textwidth}
    \includegraphics[width=\textwidth]{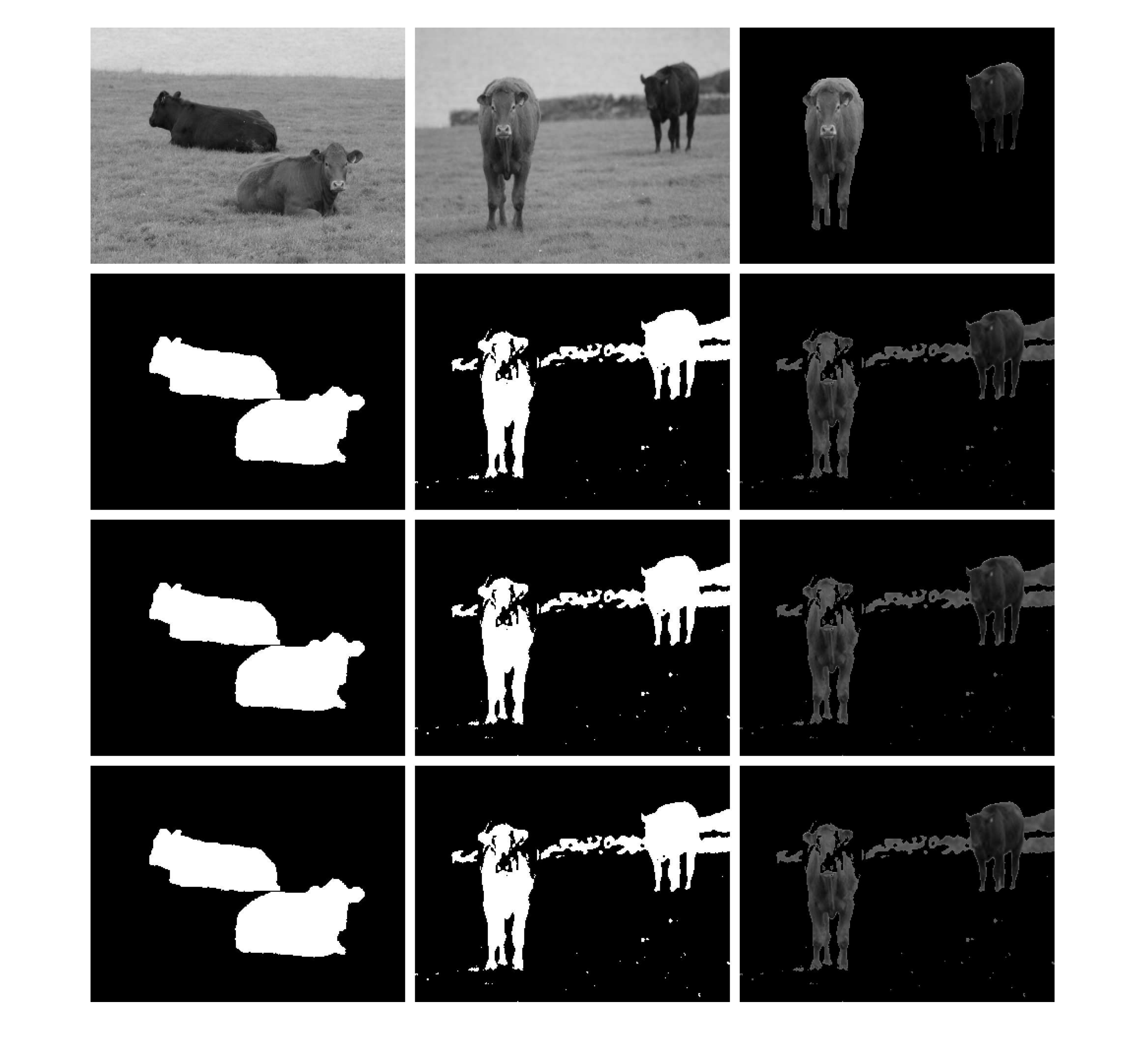}\vspace{-0.25cm}
    \caption{$\hat\mu = 200, \sigma = 35$, error = 6.1882 \%, time = 7.1 sec.}\vspace{0.25cm}
    \end{subfigure} \\
        \begin{subfigure}{0.24\textwidth}
    \includegraphics[width=\textwidth]{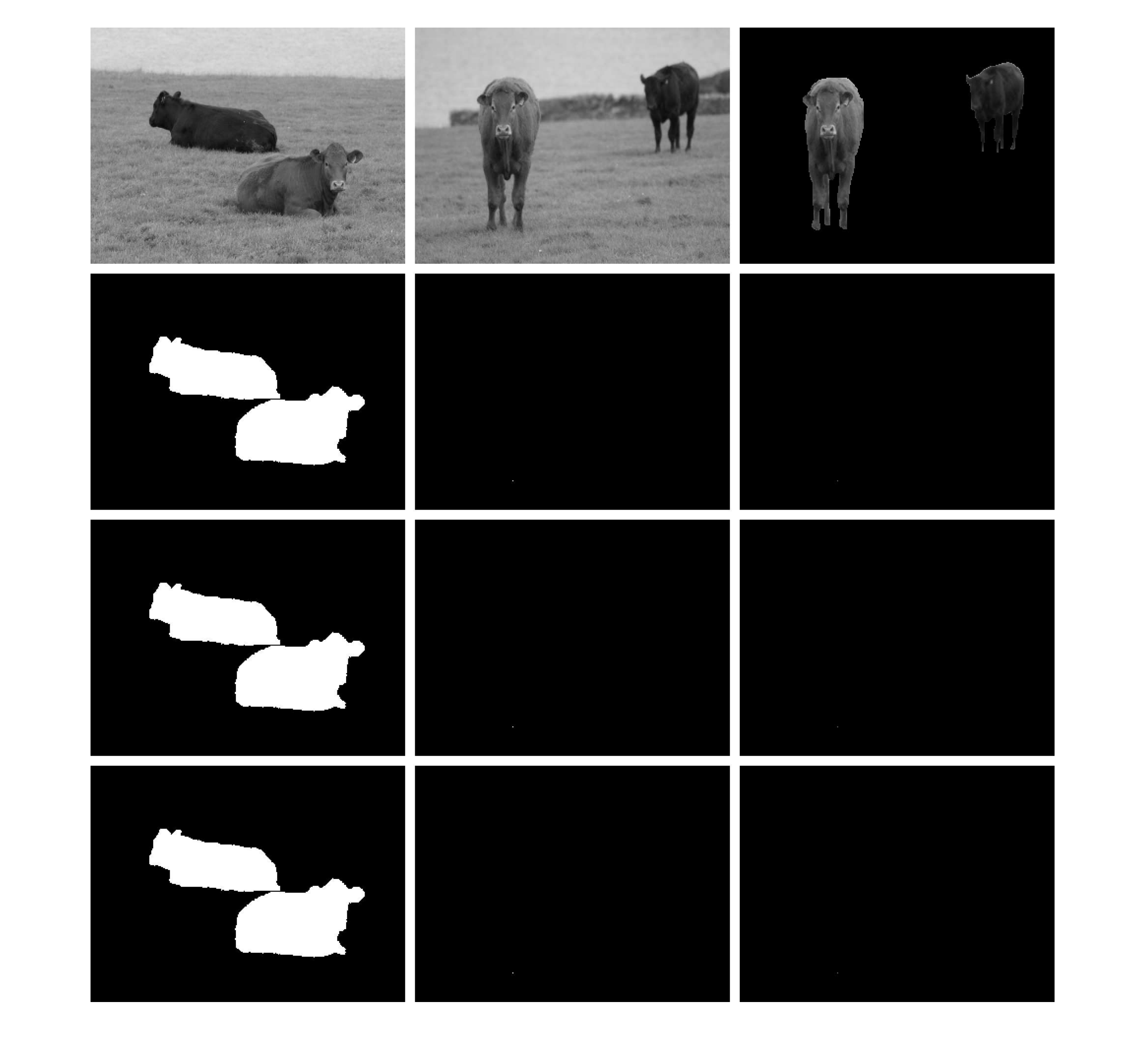}\vspace{-0.25cm}
    \caption{$\hat\mu = 50, \sigma = 50$, error = 12.7995 \%, time = 7.0 sec.}\vspace{0.25cm}
    \end{subfigure}
       \begin{subfigure}{0.24\textwidth}
    \includegraphics[width=\textwidth]{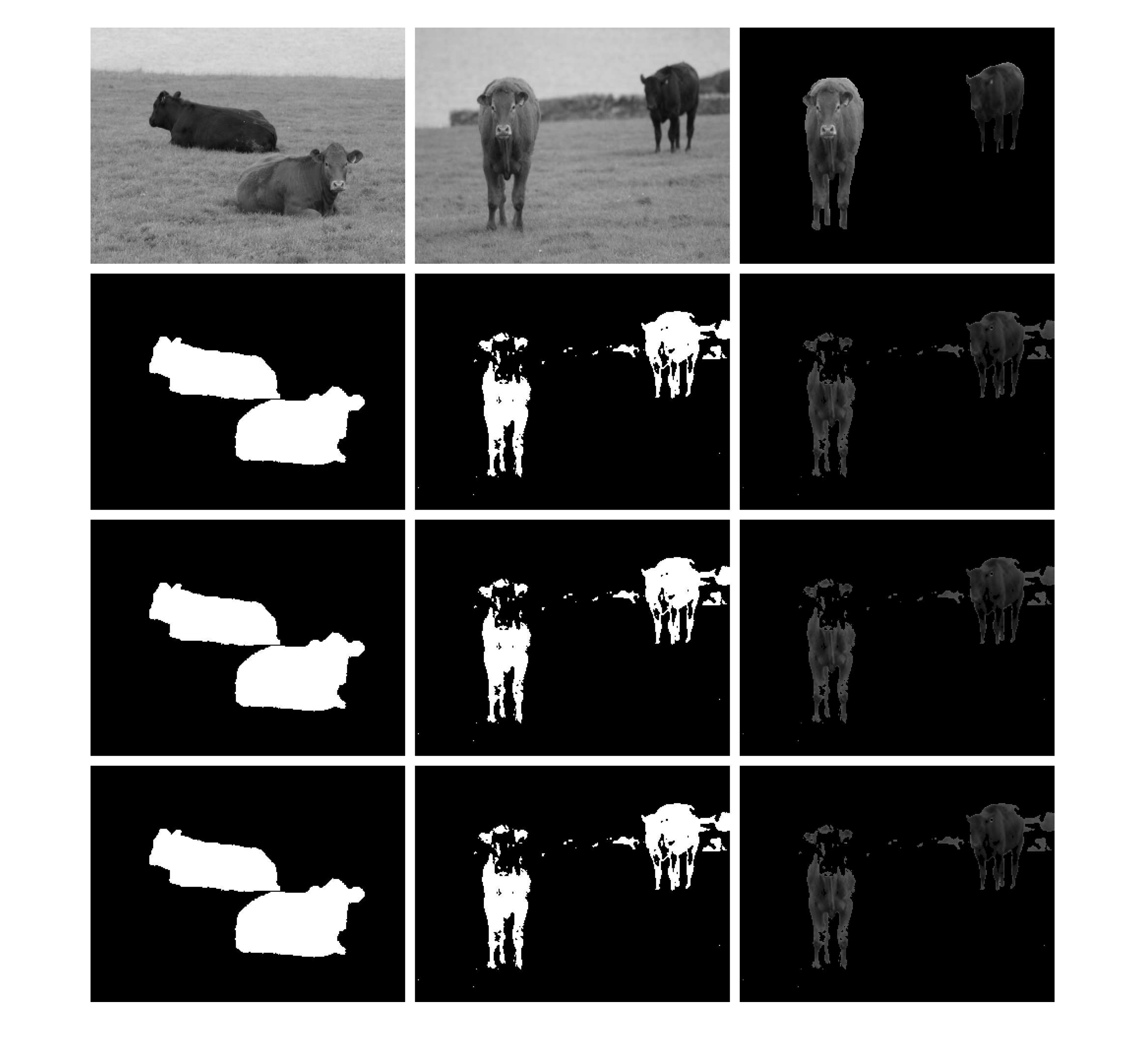}\vspace{-0.25cm}
    \caption{$\hat\mu = 100, \sigma = 50$, error = 5.7868 \%, time = 6.9 sec.}\vspace{0.25cm}
    \end{subfigure}
       \begin{subfigure}{0.24\textwidth}
    \includegraphics[width=\textwidth]{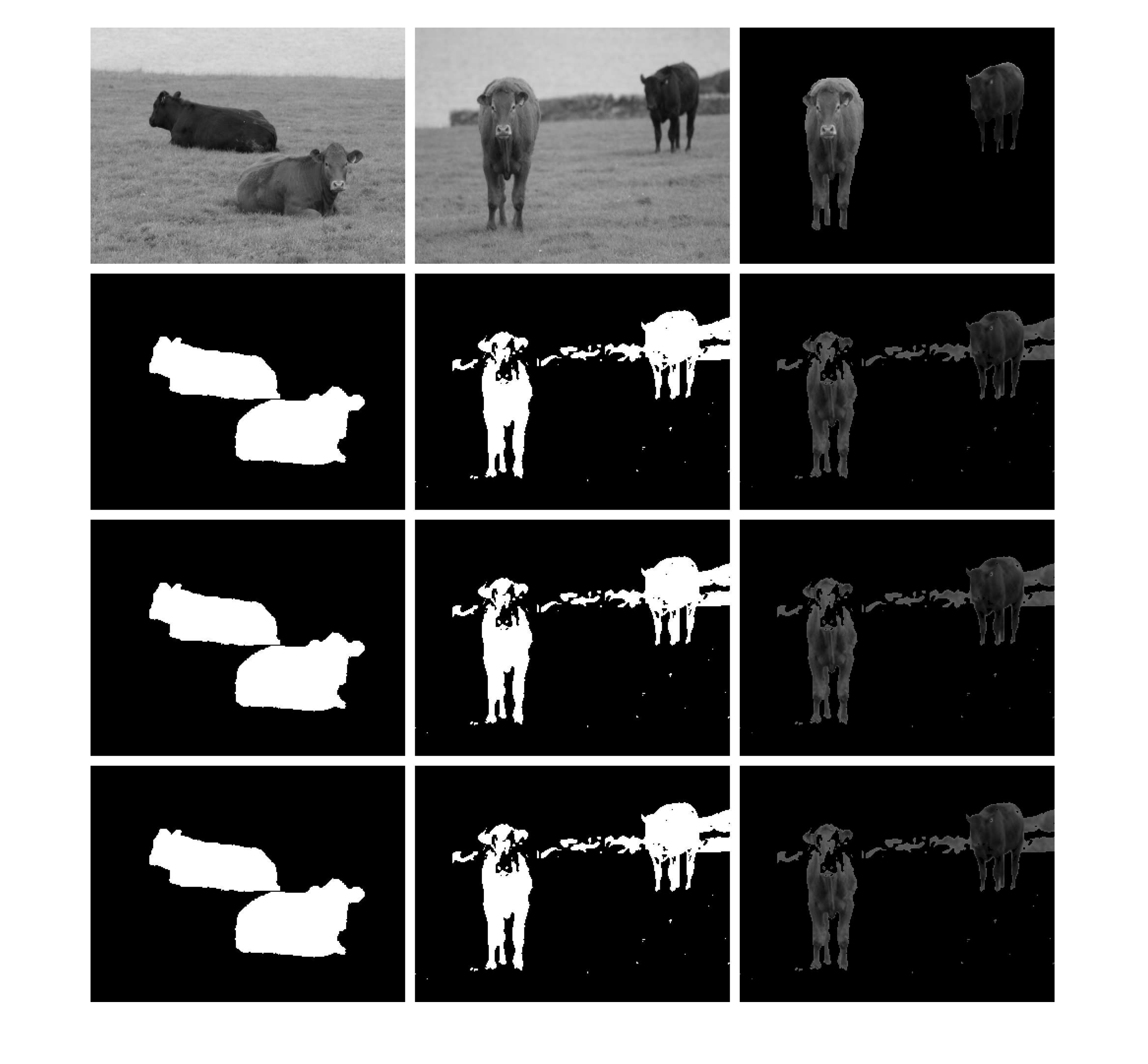}\vspace{-0.25cm}
    \caption{$\hat\mu = 150, \sigma = 50$, error = 5.8861 \%, time = 6.9 sec.}\vspace{0.25cm}
    \end{subfigure}
       \begin{subfigure}{0.24\textwidth}
    \includegraphics[width=\textwidth]{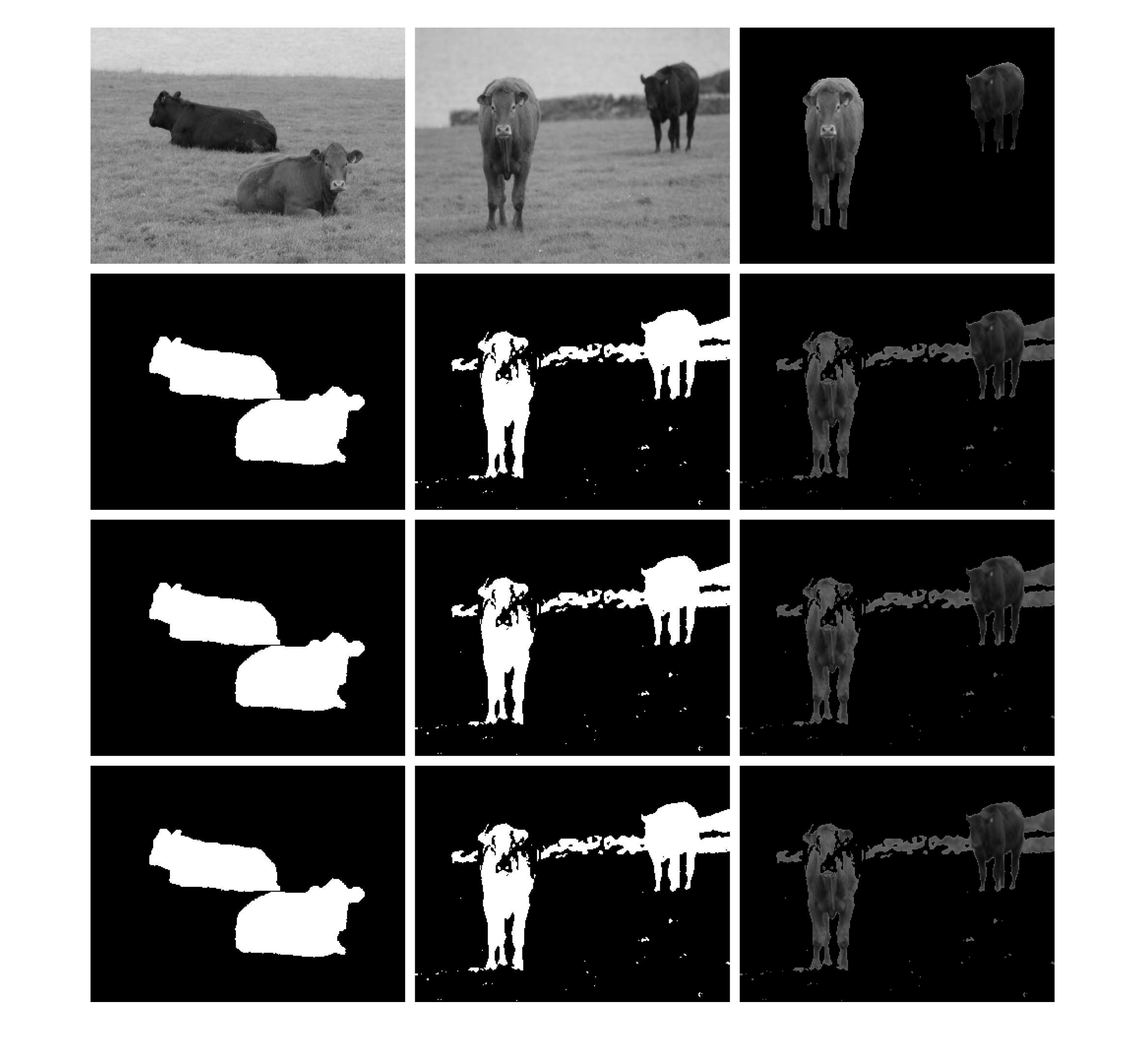}\vspace{-0.25cm}
    \caption{$\hat\mu = 200, \sigma = 50$, error = 6.291 \%, time = 7.1 sec.}\vspace{0.25cm}
    \end{subfigure} 
    \caption{Progression of the MBO SDIE scheme for the \emph{greyscale} segmentation task. In each subfigure: The first row shows the reference data, the image to be segmented, and the ground truth segmentation. The middle rows, showing the reshaped label vector $u_n$ and the image masked by the label, each represent one iteration of the considered algorithm, to be read from top to bottom. The last row gives the state returned by the scheme, i.e. the state satisfying the termination criterion.}
    \label{Fig_Progression_GS}
\end{figure}
We observe that this segmentation problem is indeed considerably harder than the \emph{two cows} problem, as we anticipated after stating Example~\ref{ex_twocows_gs}. The difference in shade between the left cow and the wall is less visible than in Example~\ref{ex_twocows}, and the left cow's snout is less identifiable as part of the cow. Thus, the segmentation errors we incur are about 3 times larger than before. There is hardly any visible influence from changing $\sigma$ in the given range. However, $\hat\mu$ has a significant impact on the result. Indeed,  for $\mu=50$ the algorithm does not find any segmentation. For $\hat\mu \geq 100$, we get more practical segmentations. Interestingly, for $\hat\mu = 150$ and $\hat\mu = 200$ we get almost all of the left cow, but misclassify most of the wall in the background; 
for $\hat\mu = 100$, we miss a large part of the left cow, but classify more accurately the wall.
The interpretation of $\hat\mu$ as the statistical precision of the reference explains this effect well. For $\hat\mu = 100$, we assume that the reference is less precise, leading us (due to the smoothing effect of the Ginzburg--Landau regularisation) to classify accurately most of the wall. With $\hat\mu \geq 150$, we assume that the reference is more precise, leading us to misclassify the wall (due to its similarity to the cows in the reference data image) but classify accurately more of the cows. At $\hat\mu = 200$, this effect even leads to a larger total segmentation error.
The runtimes are approximately equal across all choices of parameters, except for $\hat\mu = 200$ and $\sigma = 20$ for which the runtime is much higher.

\subsection{Many cows}

\begin{figure}[ht]
\centering
\begin{subfigure}{0.99\textwidth}
\centering
         \includegraphics[width=\textwidth]{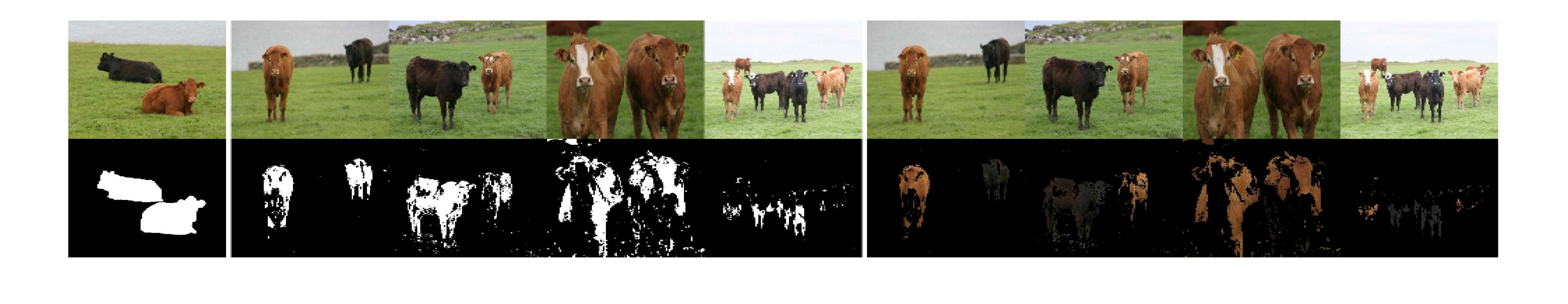}\vspace{-0.5cm}
         \caption{Segmentation with parameters $\varepsilon = \tau = 0.003, K = 70, k_b = 1, k = 1, \hat\mu = 30, \sigma = 35$, and the symmetric normalised Laplacian. Elapsed time for segmentation: 9.1 sec.}
\end{subfigure} \\
\begin{subfigure}{0.99\textwidth}
\centering
        \includegraphics[width=\textwidth]{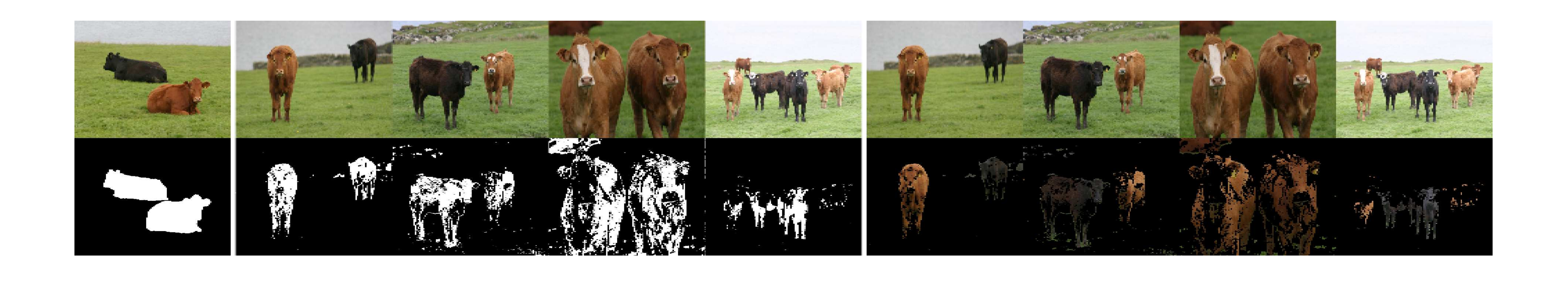}\vspace{-0.5cm}
\caption{Segmentation with parameters $\varepsilon = \tau = 0.00025, K = 100, k_b = 1, k = 1, \hat\mu = 500, \sigma = 35$, and the symmetric normalised Laplacian. Elapsed time for segmentation: 14.0 sec.}
\end{subfigure}\\
 \begin{subfigure}{0.99\textwidth}
\centering
         \includegraphics[width=\textwidth]{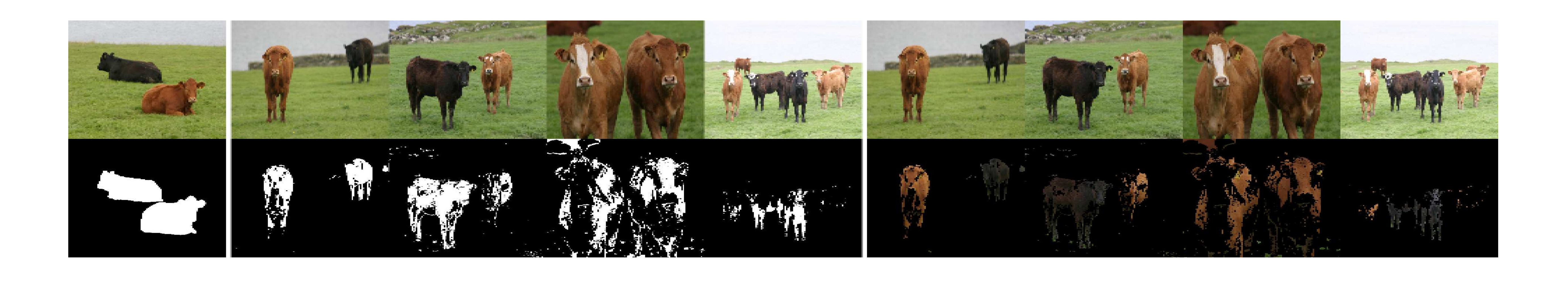}\vspace{-0.5cm}
\caption{Segmentation with parameters $\varepsilon = \tau = 0.00025, K = 100, k_b = 10, k = 10, \hat\mu = 400, \sigma = 35$, and the symmetric normalised Laplacian. Elapsed time for segmentation: 19.1 sec.}
\end{subfigure}
\caption{Segmentations of the MBO scheme for the \emph{many cows} segmentation task. In each subfigure: the top row shows the reference data image and twice the image that is to be segmented, and the bottom row shows the reference $\tilde f$, the segmentation returned by the respective algorithm, and the original image masked with the segmentation.}
\label{Fig_Results_manycows} 
\end{figure}

We finally study the \emph{many cows} example, i.e. Example~\ref{ex_manycows}. The main differences to the earlier examples are the larger size of the image that is to be segmented and the variety within it. We first comment on the size. The image is given by a $480 \times 2560 \times 3$ tensor, which is a manageable size. The graph Laplacian, however, is a dense matrix with $1.536 \times 10^6$ rows and columns. A matrix of this size requires $17.6$ TB of memory to be constructed in \textsc{Matlab}R2019a. This image is much more difficult to segment than the previous examples, in which the cows in the image to be segmented are very similar to the cows in the reference data. Here, we have concatenated images of cows that look very different, e.g. cows with a white blaze on their nose.

As the \emph{two cows} image is part of the \emph{many cows} image, we first test the algorithmic set-up that was successful at segmenting the former. We show the result (and remind the reader of the set-up) in Fig.~\ref{Fig_Results_manycows}(a). The segmentation obtained in this experiment is rather mediocre---even the \emph{two cows} part is only coarsely reconstructed.  We present two more attempts at segmenting the \emph{many cows} image in Fig.~\ref{Fig_Results_manycows}: we choose $\varepsilon = \tau = 0.00025$, a slightly larger Nystr\"om rank $K = 100$, and vary $(k_b, k, \mu) \in \{(1,1,500), (10,10,400)\}$. In both cases, we obtain a considerably better segmentation of the image. In the case where $\hat\mu = 500$, we see a good, but slightly noisy segmentation of the brown and black parts of the cows. In the case where $\hat\mu = 400$, we reduce the noise in the segmentation, but then also misclassify some parts of the cows. The blaze (and surrounding fur) is not recognised as `cow' in any of the segmentations, likely because the blaze is not represented in the reference data image.
The influence of the set-up on the runtimes is now much more pronounced. For the given segmentations, however, all the runtimes are at most a factor of eight larger than the smallest runtimes in the previous examples, despite the larger image size.

\section{Conclusion}

Extending the set-up and results in \cite{Budd}, we defined the continuous-in-time graph Allen--Cahn equation (ACE) with fidelity forcing as in \cite{BF} and a semi-discrete implicit Euler (SDIE) time discretisation scheme for this equation. We proved well-posedness of the ACE and showed that solutions of the SDIE scheme converge to the solution of the ACE. Moreover, we proved that the graph Merriman--Bence--Osher (MBO) scheme with fidelity forcing \cite{MKB} is a special case of an SDIE scheme.

We provided an algorithm to solve the SDIE scheme, which---besides the obvious extension from the MBO scheme to the SDIE scheme---differs in two key places from the existing algorithm for graph MBO with fidelity forcing: it implements the Nystr\"om extension via a QR decomposition and it replaces the Euler discretisation of the diffusion step with a computation based on the Strang formula for matrix exponentials. The former of these changes appears to have been a quite significant improvement: in experiments the Nystr\"om-QR method proved to be faster, more accurate, and more stable than the Nystr\"om method used in previous literature \cite{MKB,BF}, and it is less conceptually troubling than that method, as it does not involve taking the square root of a non-positive-semi-definite matrix. 

We applied this algorithm to a number of image segmentation examples concerning images of cows from the Microsoft Research Cambridge Object Recognition Image Database. We found that whilst the SDIE scheme yielded no improvement over the MBO scheme (and took longer to run in the non-MBO case) the other improvements that we made led to a substantial qualitative improvement over the segmentations of the corresponding examples in \cite{MKB,BF}. We furthermore investigated empirically various properties of this numerical method and the role of different parameters. In particular:
\begin{itemize}
    \item We found that the symmetric normalised Laplacian incurred consistently low segmentation error when approximated to a low rank, whilst the random walk Laplacian was more reliably accurate at higher ranks (where `higher' is still less than 0.1\% of the full rank). Thus for applications that require scalabity, and thus very low-rank approximations, we recommend using the symmetric normalised Laplacian.
    \item We investigated empirically the uncertainty inherited from the randomisation in the Nystr\"om extension. We found that the rank reduction and the normalisation of the graph Laplacian had the most influence on the uncertainty, and we furthermore observed that at higher ranks the segmentations had high variance at those pixels which are genuinely difficult to classify, e.g. the boundaries of the cows.
    \item We noted that the fidelity parameter $\mu$ corresponds to ascribing a statistical precision to the reference data. We observed that when the reference data were not fully informative, as in Examples \ref{ex_twocows_gs} and \ref{ex_manycows}, it was important to tune this parameter to get an accurate segmentation.
\end{itemize}

To conclude, we give a number of directions we hope to explore in future work, in no particular order. 

We seek to investigate the use of the very recent method of \cite{BSV} combined with methods such as the HMT method \cite{HMT} or the (even more recent) generalised Nystr\"om method \cite{Yuji} to further increase the accuracy of the low-rank approximations to the graph Laplacian. Unfortunately, we became aware of these works too near to finalising this paper to use those methods here.

We will seek to extend both our theoretical and numerical framework to multi-class graph-based classification, as considered for example in \cite{BayesianGraphs,GCP,Auction}. The groundwork for this extension was laid by the authors in  \cite[Section 6]{volumeBudd}.

This work dovetails with currently ongoing work of the authors with Carola-Bibiane Sch\"onlieb and Simone Parisotto exploring graph-based joint reconstruction and segmentation with the SDIE scheme. In practice, we often do not observe images directly, but must reconstruct an image from what we observe. For example, when the image is noisy, blurred, has regions missing, or we observe only a transform of the image, such as the output of a CT or MRI scanner. `Joint reconstruction and segmentation' (see \cite{jointrecon1,jointrecon2}) is the task of simultaneously reconstructing and segmenting an image, which can reduce computational time and improve the quality of both the reconstruction and the segmentation.

We will also seek to develop an SDIE classification algorithm that is scalable towards larger data sets containing not just one, but hundreds or thousands of reference data images. This can be attempted via data subsampling: rather than evolving the SDIE scheme with regard to the full reference data set, we use only a data subset that is randomly replaced by a different subset after some time interval has passed.  Data subsampling is the fundamental principle of various algorithms used in machine learning, such as stochastic gradient descent \cite{RobbinsMonro}. A subsampling strategy that is applicable in continuous-in-time algorithms has recently been proposed and analysed by \cite{Jonas}.

\appendix 
\section{An analysis of the method from \cite{MKB}} 
\label{MKBapp} 
In \cite{MKB}, the authors approximated $\mathcal{S}_\tau u$ by a semi-implicit Euler method for fidelity forced diffusion. That is, since $\mathcal{S}_\tau u$ is defined to equal $v(\tau)$ where $v$ is defined by
	\[\frac{dv}{dt} = -\Delta v  -M(v-\tilde f), \qquad v(0) = u, \] 
the authors of \cite{MKB} approximate trajectories of this ODE by a semi-implicit Euler method with time step $\delta t>0$ (such that $\tau/\delta t\in\mathbb{N}$), i.e. $v_0:=u$ and
\[
\frac{v_{k+1}-v_k}{\delta t} = -\Delta v_{k+1} -M(v_k-\tilde f)
\]
with solution (recalling that $f:=M\tilde f$)
\be \label{Eulersoln}
\begin{split}
v_{k+1} &= \left(I+\delta t \Delta\right)^{-1}\left(v_k - \delta t  M(v_k-\tilde f)\right)\\
&=\left(I+\delta t \Delta\right)^{-1}\left(I -\delta t  M\right)v_k+\delta t \left(I+\delta t \Delta\right)^{-1} f
\end{split}
\ee
and then \eqref{Eulersoln} leads to the approximation $\mathcal{S}_\tau u\approx v_{\tau/\delta t}$. 
To compute \eqref{Eulersoln}, they use the Nystr\"om decomposition to compute the leading eigenvectors and eigenvalues of $\Delta$. 
\begin{nb}
In fact, in \cite{MKB} the authors use $\Delta_s$ not $\Delta$. It makes no difference to this analysis which Laplacian is used.
\end{nb}
Given $\Delta\approx U_1\Lambda U_2^T$, an approximate SVD of low rank, the authors approximate \eqref{Eulersoln} by 
\be
\label{Eulersoln2}
\begin{split}
\hat v_{k+1} &= U_1(I_K + \delta t \Lambda)^{-1} U_2^T(\hat v_k -\delta t M(\hat v_k-\tilde f))\\
&=  U_1(I_K + \delta t \Lambda)^{-1} U_2^T(I-\delta t M)\hat v_k + \delta t  U_1(I_K + \delta t \Lambda)^{-1} U_2^T f.
\end{split}
\ee
Therefore, the final approximation for $\mathcal{S}_\tau u$ in \cite{MKB} is computed by setting $\hat v_0 = u$ and $\mathcal{S}_\tau u\approx \hat v_{\tau/\delta t}$.
\subsection{Analysis}
Both \eqref{Eulersoln} and \eqref{Eulersoln2} are of the form 
\[
v_{k+1} = \mathcal{A}v_k + g.
\]
By induction, this has $k^\text{th}$ term 
\be\label{kthterm}
v_k = \mathcal{A}^k v_0 + \sum_{r=0}^{k-1} \mathcal{A}^r g = \mathcal{A}^k v_0 + (\mathcal{A}-I)^{-1}(\mathcal{A}^{k}-I) g. 
\ee
Thus taking $k = \tau/\delta t$ and $v_0 = u$, we get successive approximations 
\begin{align}
\notag &\mathcal{S}_\tau u \\
\label{approx1}&\approx \left[\left(I+\delta t \Delta\right)^{-1}\left(I - \delta t M\right)\right]^k u\\ 
\notag &\quad + \left[\left(I+\delta t \Delta\right)^{-1}\left(I - \delta t  M\right)-I\right]^{-1} \left(\left[\left(I+\delta t \Delta\right)^{-1}\left(I - \delta t M\right)\right]^k - I\right) \delta t \left(I+\delta t \Delta\right)^{-1} f\\
\label{approx2}&\approx \left[  U_1(I_K + \delta t \Lambda)^{-1} U_2^T(I-\delta t M)\right]^k u\\ 
\notag &\quad + \left[U_1(I_K + \delta t \Lambda)^{-1} U_2^T\left(I -\delta t  M\right)-I\right]^{-1} \left(\left[U_1(I_K + \delta t \Lambda)^{-1} U_2^T\left(I - \delta t M\right)\right]^k - I\right) \delta t U_1(I_K + \delta t \Lambda)^{-1} U_2^Tf.
\end{align}
To see what these approximations are doing, note that 
\[
I + \delta t X = e^{\delta t X} +\bigO(\delta t^{2})
\]
and note the Lie product formula \cite[Theorem 2.11]{Hall}
\begin{align*}
&e^{(X + Y)/k} = e^{X/k}e^{Y/k} +\bigO(k^{-2}) &\text{and therefore}& &e^{X + Y} = \left(e^{X/k}e^{Y/k}\right)^k +\bigO(k^{-1}). 
\end{align*}
Then, since $k\delta t = \tau$,
\be \label{LPFapprox}
\begin{split}
\left[\left(I+\delta t \Delta\right)^{-1}\left(I - \delta t  M\right)\right]^k &= \left[e^{-\delta t\Delta}e^{-\delta t M} +\bigO(\delta t^2)\right]^k \\
&= \left[e^{-\delta t\Delta}e^{-\delta t M}\right]^k +\bigO(k\delta t^2)\\
&= \left(e^{-\tau(\Delta +M)} + \bigO(k\delta t^2) \right) +\bigO(k\delta t^2) \\&= e^{-\tau A} + \bigO(\delta t)
\end{split}
\ee
and the second term in \eqref{approx1} becomes 
\begin{align*}
&\left[I-\left(I+\delta t \Delta\right)^{-1}\left(I - \delta t  M\right)\right]^{-1} \left(I-e^{-\tau A} + \bigO(\delta t) \right) \delta t \left(I+\delta t \Delta\right)^{-1}\\
&= \left( \Delta + M\right)^{-1}(I+\delta t \Delta) \left(I-e^{-\tau A} + \bigO(\delta t) \right)\left(I+\delta t \Delta\right)^{-1} \\
&=A^{-1}(I-e^{-\tau A}) + E + \bigO(\delta t)
\end{align*}
where (writing $[X,Y]:=XY - YX$ for the commutator of $X$ and $Y$)
\begin{align*}E&:= A^{-1}(I+\delta t \Delta) \left [I-e^{-\tau A},(I+\delta t \Delta)^{-1}\right] \\
&= -A^{-1}(I+\delta t \Delta) \left [e^{-\tau A},(I+\delta t \Delta)^{-1}\right] \\
&=A^{-1} \left [e^{-\tau A},(I+\delta t \Delta)\right](I+\delta t \Delta)^{-1}\\
&=\delta t A^{-1} \left [e^{-\tau A},  \Delta\right](I+\delta t \Delta)^{-1} = \bigO(\delta t)
\end{align*}
 is the commutation error. Hence the overall error for $\eqref{approx1}$ is $\bigO(\delta t)$. The error for \eqref{approx2} is similar but also contains an extra error from the spectrum truncation. 
 
We can also relate this Euler method to a modified quadrature rule. It is easy to see from \eqref{fdiffusesoln} that 
\[
S_\tau u = e^{-\tau A} u + \int_0^\tau e^{-tA}f \; dt.
\]
We understand the Euler approximation for the $e^{-\tau A}u$ term by \eqref{LPFapprox}. By \eqref{kthterm}, we can write the Euler approximation for the integral term as 
\begin{align}
    \label{quad1} \int_0^\tau e^{-tA}f \; dt &\approx \delta t \sum_{r=0}^{k-1} \left[\left(I+\delta t \Delta\right)^{-1}\left(I - \delta t  M\right)\right]^r\left(I+\delta t \Delta\right)^{-1} f \\
    \label{quad2} &\approx \delta t \sum_{r=0}^{k-1} \left[  U_1(I_K + \delta t \Lambda)^{-1} U_2^T(I-\delta t M)\right]^rU_1(I_K + \delta t \Lambda)^{-1} U_2^Tf.
\end{align}
We note that, since $M=\operatorname{diag}(\mu)$ (and assuming that $\delta t ||\mu||_\infty <1 $), 
\[(I-\delta t M)^{-1} = \operatorname{diag}\left((\mathbf{1}-\delta t\mu)^{-1}\right) \]
applying the reciprocation elementwise. Therefore, we can rewrite \eqref{quad1} as 
\begin{align*}
b:=\int_0^\tau e^{-tA}f \; dt &\approx \delta t \sum_{r=0}^{k-1} \left[\left(I+\delta t \Delta\right)^{-1}\left(I - \delta t  M\right)\right]^{r+1} \left((\mathbf{1}-\delta t\mu)^{-1}\odot f\right) & \\
&= \delta t \sum_{r=1}^{k} \left(e^{-r\delta t A}\left((\mathbf{1}-\delta t\mu)^{-1}\odot f\right)  + \bigO(r\delta t^2)\right) &\text{by \eqref{LPFapprox} and relabelling $r$}\\
&= \left(\delta t\sum_{r=1}^{k} e^{-r\delta t A}\left((\mathbf{1}-\delta t\mu)^{-1}\odot f\right)  \right)+ \bigO(\tau^2\delta t) &
\end{align*}
recalling that $k\delta t = \tau$. This can be seen to be a quadrature by the right-hand rule of the integral 
\[
\int_0^\tau e^{-tA} \left((\mathbf{1}-\delta t\mu)^{-1}\odot f\right) \; dt.
\]
Likewise, we can rewrite \eqref{quad2} as 
\[
\int_0^\tau e^{-tA}f \; dt \approx \delta t \sum_{r=1}^{k} \left[ U_1(I_K + \delta t \Lambda)^{-1} U_2^T(I-\mu\delta t P_{Z})\right]^{r} \left((\mathbf{1}-\delta t\mu)^{-1}\odot f\right) 
\]
where going from \eqref{quad1} to \eqref{quad2} has incurred an extra error from the spectrum truncation alongside the quadrature and Lie product formula errors. 

\subsection{Conclusions from analysis}

The key takeaway from these calculations (besides the verification that we have the usual $\bigO(\delta t)$ Euler error) concerns \eqref{LPFapprox}. That equation shows that the Euler approximation for the $e^{-\tau A}$ term is in fact an approximation of a Lie product formula approximation for $e^{-\tau A}$. This motivates our method of cutting out the middleman and using a matrix exponential formula directly, and furthermore motivates our replacement of the linear error Lie product formula with the quadratic error Strang formula. 

We have also shown how the Euler method approximation for $b$ can be written as a form of quadrature, motivating our investigation of other quadrature methods as potential improvements for computing $b$.

\section*{Acknowledgements}

We thank Andrea Bertozzi, Arjuna Flenner, and Arieh Iserles for very helpful comments. This work dovetails with ongoing work of the authors joint with Carola-Bibiane Schönlieb and Simone Parisotto, whom we therefore also thank for many indirect contributions.

The work of the first and second authors was supported by the European Union's Horizon 2020 research and innovation program under Marie Skłodowska-Curie grant 777826. The work of the third author was supported by the EPSRC grant EP/S026045/1.

\end{document}